\documentclass[11pt, twoside]{article}
\usepackage{textcomp}
\usepackage[utf8]{inputenc}
\usepackage[english]{babel} 
\usepackage[T1]{fontenc} 
\usepackage{soul}
\usepackage[normalem]{ulem}
\usepackage{fancybox}
\usepackage{fancyvrb}
\usepackage{listings}
\usepackage{moreverb}
\usepackage{url}
\usepackage{graphicx}
\graphicspath{ {./images/} }
\usepackage{textcomp}
\usepackage{lmodern}
\usepackage{eurosym}
\usepackage[cyr]{aeguill}
\usepackage{amsmath}
\usepackage{wrapfig} 
\usepackage{makeidx} 
\usepackage{picinpar} 
\usepackage[final]{pdfpages} 
\usepackage{array,multirow,tabularx}
\usepackage{amssymb}
\usepackage[a4paper]{geometry}
\usepackage{amsthm}
\usepackage{tikz}
\usepackage{fancyhdr}
\usepackage{colortbl}
\usepackage{color}
\usepackage{setspace}
\usepackage{longtable}
\usepackage{hhline}
\usepackage{arydshln}
\usepackage{makecell}
\usepackage{mathtools}
\usepackage{tkz-tab}
\usepackage{tikz-cd}
\usepackage{multicol}
\usepackage{enumerate} 
\usepackage{hyperref}
\usepackage[normalem]{ulem}
\tikzset{node distance=2.0cm, auto}
\usepackage{fancyhdr}
\usepackage{mathtools}
\usepackage{microtype}
\renewcommand{\sectionmark}[1]{}
\renewcommand{\subsectionmark}[1]{}

\usepackage[section]{placeins}

\DeclareMathOperator{\p}{\mathcal{P}}

\DeclareMathOperator{\ad}{ad}
\DeclareMathOperator{\g}{\mathfrak{g}}

\allowdisplaybreaks

\title{Cohomology, Extensions and Deformations of restricted Lie triple systems}

\author{Jon Beristain and Abdenacer Makhlouf}

\date{{\small{Universit\'e de Haute-Alsace, IRIMAS UR 7499, F-68100 Mulhouse, France,\\
jon.beristain@uha.fr, abdenacer.makhlouf@uha.fr}}}

\begin{document}
	
	\newtheorem{thm}{Theorem}[section]
	\newtheorem{prop}[thm]{Proposition}
	\newtheorem{lem}[thm]{Lemma}
	\newtheorem{cor}[thm]{Corollary}
	\theoremstyle{definition}
	\newtheorem{defi}[thm]{Definition}
	\newtheorem*{rmq}{Remark}
	\newtheorem{ex}[thm]{Example}
	\maketitle
	\begin{abstract}
		The main purpose of this paper is to provide a cohomology theory of restricted Lie triple systems in low degrees and their algebraic interpretations. 
        We define restricted cochains up to degree
5, together with the corresponding coboundary operators. Moreover, we provide 
interpretations of the low-degree restricted cohomology groups in terms of restricted derivations,
extensions of restricted modules, formal deformations, and extensions of restricted Lie triple systems.
	\end{abstract}
	
	\noindent\textbf{Keywords: }Lie triple system, restricted Lie triple system, cohomology, restricted derivation, restricted module, deformation, extension. \\
	\noindent\textbf{MSC 2020 classification:} 17A40, 17B56, 17B15.
	\tableofcontents
     \markboth{Beristain and Makhlouf}{Cohomology, Extensions and Deformations of restricted Lie triple systems}
     \markright{Cohomology, Extensions and Deformations of restricted Lie triple systems}

\section{Introduction}
Lie triple systems were introduced by Jacobson in order to study representations of Jordan algebras (\cite{J51}). They are related to symmetric spaces in the same way that Lie algebras are related to Lie groups; they correspond to tangent spaces of symmetric spaces. Moreover, Jacobson showed that Lie triple systems are exactly the subspaces of Lie algebras closed under the bracket  $[[\cdot,\cdot],\cdot] $. Indeed, a subspace $S$ of a Lie algebra $L$ such that $[[x,y],z] $ lies in $S$ for all $x,y,z$ in $S$ is a Lie triple system for this operation. Moreover, it was shown in \cite{J51} that we can always embed a Lie triple system $T$ into a Lie algebra $\mathcal{G} $ in such a way that $T$ is a subspace of $\mathcal{G} $ closed under the bracket $[[\cdot,\cdot],\cdot] $. Lie triple systems were strongly studied by Lister in \cite{L51}.

In positive characteristic $p>0$, an additional structure arises naturally on certain Lie algebras. This was noticed by Jacobson, who observed that in an associative algebra $A$ over a base field of characteristic $p>0$, for every derivation $D$, $D^p$ is again a derivation (\cite{J37}). Moreover, the Frobenius map $x\in A\mapsto x^p$ is linked with the Lie bracket induced by the Jacobson identities. This led Jacobson to define an abstract $p$-map and the notion of  restricted Lie algebra (\cite{J41}). Later, Hodge generalized this notion to Lie triple systems and defined restricted Lie triple systems (\cite{H01}).

The cohomology of Lie algebras was first defined through the Chevalley-Eilenberg construction. This approach provides an explicit definition of a cochain complex together with the corresponding coboundary operators. With the development of category theory and homological algebra, another construction of Lie algebra cohomology was introduced in terms of the derived cohomology associated with a free resolution. This is the Cartan-Eilenberg cohomology. It was subsequently shown that these two constructions yield the same cohomology, the Chevalley-Eilenberg complex being obtained from a free resolution of the trivial module.

For restricted Lie algebras, a definition in terms of free resolutions was given by Hochschild (\cite{H54}). However, this construction is not suitable for explicit cohomological computations. In his thesis (\cite{E96}), Evans proposed an explicit construction of the first terms of a free resolution of the trivial module in the category of $U_{\mathrm{res}}$-modules for abelian restricted Lie algebras, up to degree $p$. For general restricted Lie algebras, he provided an explicit construction of the cochains up to degree $3$. However, it was not shown that this complex, up to degree $3$, admits an interpretation in terms of the derived cohomology associated with a free resolution.

Yamaguti and Harris studied the cohomology of Lie triple systems at roughly the same time, although their approaches differ. In \cite{H60}, Harris considered Lie triple systems with the $(-1)$-eigenspaces of involutive automorphisms of Lie algebras. He then defined the cohomology of Lie algebras with involution and obtained the cohomology of Lie triple systems as a particular case. This cohomology is also defined in terms of derived cohomology associated with a free resolution. In \cite{Y60}, Yamaguti introduced the notion of  representation of a Lie triple system and explicitly constructed two cochain complexes associated with such a representation. This construction has the advantage of being explicit and provides meaningful algebraic interpretations of the low-degree cohomology spaces.

The cohomology of restricted Lie triple systems was studied by Hodge and Parshall in \cite{HP02}. In the present work, we revisit Yamaguti’s definition of the cohomology of Lie triple systems and define the cohomology of restricted Lie triple systems by adapting Evans’s ideas.

The paper is organized as follows. In Sections \eqref{s2} and \eqref{s3}, we recall the basics on Chevalley–Eilenberg cohomology for Lie algebras and Evans’s explicit construction of restricted Lie algebra cohomology, together with their algebraic interpretations. In Section \eqref{s4}, we recall the definition of Lie triple systems and Yamaguti’s cohomology theory. We extend Yamaguti’s complex to degree $-1$ and give interpretations of low-degree cohomology groups in terms of derivations, extensions of modules and extensions of Lie triple systems. Our main contributions are presented in Sections \eqref{s5} and \eqref{s6}. In Section \eqref{s5}, we define restricted cochains up to degree $5$, together with the corresponding coboundary operators. In Section \eqref{s6}, we provide algebraic interpretations of the low-degree restricted cohomology groups in terms of restricted derivations, extensions of restricted modules, deformations, and extensions of restricted Lie triple systems. 

\section{Cohomology of Lie algebras}\label{s2}

\subsection{Definition}

Let $L$ be a Lie algebra over a field $\mathbf{K} $, and let $M$ be an  $L$-module. Chevalley and Eilenberg constructed a complex for computing Lie algebra cohomology as follows. A $q$-dimensional cochain of $L$ with coefficients in $M$ is a skew-symmetric, $q$-linear map on $L$ taking values in $M$. The set of all such maps forms a vector space $C^q(L,M)= \mathrm{Hom}_\mathbf{K}(\Lambda^qL,M) $ over $\mathbf{K}$ under pointwise addition and scalar multiplication. We set $C^q(L,M)=0 $ if $q<0$ and if $q=0 $, we identify $C^0(L,M) $ with $M\cong \mathrm{Hom}_\mathbf{K}(\mathbf{K},M) $. If $\varphi \in C^q(L,M) $, then $\varphi$ determines an element $d^q\varphi\in C^{q+1}(L,M) $ by the formula \begin{align*}
    d^q\varphi(x_1,\ldots,x_{q+1})=& \sum\limits_{1\leq s<t\leq q+1}(-1)^{s+t-1}\varphi([x_s,x_t],x_1,\ldots,\widehat{x_s},\ldots,\widehat{x_t},\ldots,x_{q+1}) \\ & + \sum\limits_{1\leq s\leq q+1} (-1)^sx_s\varphi(x_1,\ldots,\widehat{x_s},\ldots,x_{q+1}) ,
\end{align*}  where the symbols $\widehat{x_s} $ indicate that this term is to be omitted. The map $d^q: C^q(L,m)\to C^{q+1}(L,M),$ $\varphi\mapsto d^q\varphi$ is a linear map  satisfying $d^{q+1}\circ d^q=0 $. Therefore, $\left\{C^\star(L,M),d \right\}$ is a complex and its $q$-th cohomology group is called the $q$-dimensional cohomology group of $L$ with coefficients in $M$ and is denoted by $H^q(L,M) $. We will denote the $q$-dimensional cocycles and coboundaries by $Z^q(L,M) $ and $B^q(L,M) $, respectively. In degrees $q=0,1$ and $2$, the coboundary operator formula reduces to \begin{align*}
    d^0(m)(g) =&-gm ,\\ d^1\varphi(g,h) =& -g\varphi(h)+ h\varphi(g)+\varphi([g,h]), \\ d^2\varphi(g,h,f) =& \varphi([g,h],f)-\varphi([g,f],h)+ \varphi([h,f],g) -g\varphi(h,f)+h\varphi(g,f)-f\varphi(g,h),\\  d^3\varphi(g,h,f,u)=& \varphi([g,h],f,u)-\varphi([g,f],h,u)+ \varphi([g,u],h,f)+\varphi([h,f],g,u)-\varphi([h,u],g,f)\\ & +\varphi([f,u],g,h) - g\varphi(h,f,u)+h\varphi(g,f,u)-f\varphi(g,h,u)+ u\varphi(g,h,f).
\end{align*}

\subsection{Algebraic interpretations}

\begin{defi}
    Let $L$ be a Lie algebra. A derivation of $L$ is a linear map $D:L\to L$ such that $$D([x,y])=[Dx,y]+[x,Dy],  \forall x\in L.  $$ We denote by $\mathrm{Der}(L)$ the vector space of derivations.
\end{defi}

\begin{ex}
    Let $x$ be an element of a Lie algebra $L$. The map $y\in L\mapsto[x,y] $ is a derivation, called an inner derivation. We denote by $\mathrm{InnDer}(L)$ the space spanned by inner derivations.
\end{ex}

\begin{prop}
   The first cohomology group $H^1(L,L) $ is canonically isomorphic to the space $\mathrm{Der}(L)/\mathrm{InnDer}(L) $.
\end{prop}

A one-dimensional right extension of $L$ is a short exact sequence of Lie algebras and their morphisms $$0\to L\to L'\to \mathbf{K}\to 0, $$ with the Lie bracket in $L'=L\oplus \mathbf{K}$, defined by $$[(g_1,\lambda_1),(g_2,\lambda_2)]= ([g_1,g_2]-\lambda_1c(g_2)+\lambda_2c(g_1),0), $$ where $c:L\to L$ is a linear map. Two one-dimensional right extensions of $L$ are equivalent if they can be included in a commutative diagram \[\begin{tikzcd}
	0 & L & {L'} & {\mathbf{K}} & 0 \\
	0 & L & {L''} & {\mathbf{K}} & 0
	\arrow[from=1-1, to=1-2]
	\arrow[from=1-2, to=1-3]
	\arrow[from=1-2, to=2-2]
	\arrow[from=1-3, to=1-4]
	\arrow[from=1-3, to=2-3]
	\arrow[from=1-4, to=1-5]
	\arrow[from=1-4, to=2-4]
	\arrow[from=2-1, to=2-2]
	\arrow[from=2-2, to=2-3]
	\arrow[from=2-3, to=2-4]
	\arrow[from=2-4, to=2-5]
\end{tikzcd}\] The $1$-cochain $c\in C^1(L,L)$ is a cocycle if and only if the bracket in $L\oplus \mathbf{K}$ satisfies the Jacobi identity. 

\begin{prop}
    There is a one-to-one correspondence between $H^1(L,L)$ and equivalence classes of one-dimensional right extensions of $L$.
\end{prop}

A one-dimensional right extension of an $L$-module $M$ is a short exact sequence $$0\to M\to M'\to \mathbf{K}\to 0 $$ of $L$-modules and their morphisms. 

 \begin{prop}
     The first cohomology group $H^1(L,M)$ is naturally isomorphic to the set of equivalence classes of one-dimensional right extensions of $M$.
 \end{prop} 

A central extension of a Lie algebra $L$ is an exact sequence of Lie algebras $$0\to \mathbf{K}\to L'\to L\to 0 ,$$ with the Lie bracket in $L'=\mathbf{K}\oplus L$ given by $$[(\lambda_1,x_1),(\lambda_2,x_2)]= (c(x_1,x_2),[x_1,x_2]) ,$$ where $c:L\times L\to \mathbf{K}$ is a skew-symmetric bilinear map. Here, the Jacobi identity of the bracket is equivalent to $c\in C^2(L,\mathbf{K})$ being a cocycle. 

\begin{prop}
    The second cohomology group $H^2(L,\mathbf{K})$ is naturally isomorphic to the set of equivalence classes of central extensions of $L$.
\end{prop}

Finally, the second cohomology group controls the infinitesimal deformations of $L$. An infinitesimal deformation of $L$ is a map $\eta:L\times L\times \mathbf{K}\to L $ which can be written  $$\eta(g,h,t)= [g,h]_t =[g,h]+tc(g,h) ,$$ where $c\in C^2(L,L)$ and $[g,h]_t$ is a Lie bracket (mod $t^2$) in $L$ for all $t\in \mathbf{K}$. 

\begin{prop}
    There is a natural identification between the set of equivalent classes of infinitesimal deformations and $H^2(L,L)$.
\end{prop}

\section{Cohomology of restricted Lie algebras}\label{s3}

\subsection{Definition}

 As in the ordinary case, we construct the restricted universal enveloping algebra $U_{\mathrm{res.}}(L) $ by taking the quotient of $U(L)$ by the two-sided ideal generated by all elements of the form $x^{[p]}-x^p $. In order to construct the cohomology theory in the same way as the Cartan-Eilenberg cohomology, we attempt to construct a free resolution of $\mathbf{K}$ in the category of $U_{\mathrm{res.}}(L) $-modules. In the abelian case, Evans constructed a complex explicitly such that $H^k(L,M)= \mathrm{Ext}_{U_{\mathrm{res.}}(L)}^k(\mathbf{K},M) $ for $0\leq k<p$. 

Evans gave an explicit construction of the restricted cochains up to degree $3$, although it has not been proved that $H^k(L,M)= \mathrm{Ext}_{U_{\mathrm{res.}}(L)}^k(\mathbf{K},M) $  for $0\leq k<3  $. The construction is the following: 

For $k\leq 1$, define $C_\mathrm{res}^k(L,M)=C^k(L,M)$ and $d^0: C_\mathrm{res}^0(L,M)\to C_\mathrm{res}^1(L,M) $ by $d_\mathrm{res}^0=d^0$. It follows immediately that $H_\mathrm{res}^0(L,M)=H^0(L,M) $. 

\begin{defi}\cite{E96}
    Let $\varphi:\Lambda^2L\to M $ be a skew-symmetric bilinear form on $L$ with values in $M$ and $\omega:L\to M$ be a map. We say that $\omega$ has the $\ast $-property with respect to $\varphi$ if for all $\lambda\in \mathbf{K}$ and all $g,h\in L$, \begin{enumerate}
    \item $\omega(\lambda g)=\lambda^pw(g)$,
    \item \begin{align*}
        \omega(g+h)= &\omega(g)+\omega(h) \\ & + \sum_{\substack{g_j\in \left\{g,h \right\},\\g_1=g,g_2=h}}\frac{1}{\#(g)}\sum\limits_{k=0}^{p-2}(-1)^kg_p\cdots g_{p-k+1}\varphi([g_1,\ldots,g_{p-k-1}],g_{p-k}).
    \end{align*} 
\end{enumerate} The space of $2$-dimensional cochains is  defined as $$C_\mathrm{res}^2(L,M)=\left\{(\varphi,\omega)\mid \varphi:\Lambda^2L\to M, \omega:L\to M \quad \text{has the $\ast$-property w.r.t $\varphi$} \right\}. $$
\end{defi}

The space $C_\mathrm{res}^2(L,M)$ is a vector space over $\mathbf{K} $ by pointwise addition and scalar multiplication. Let $\psi:L\to M $ be an element of $C^1(L,M) $. We define $\tilde{\psi}:L\to M $ by the formula $$\tilde{\psi}(g)= \psi(g^{[p]})- g^{p-1}\psi(g).$$ Then $\tilde{\psi} $ has the $\ast$-property with respect to $d\psi $. We then define the coboundary operator $$d_\mathrm{res}^1:C_\mathrm{res}^1(L,M)\to C_\mathrm{res}^2(L,M), d_\mathrm{res}^1:\psi\mapsto (d^1\psi,\tilde{\psi}). $$

 \begin{thm}\cite{E96}
 We have $d^1_\mathrm{res}d_\mathrm{res}^0=0 $ and $H_\mathrm{res}^1(L,M)=\mathrm{Ker}(d_\mathrm{res}^1)/\mathrm{Im}(d_\mathrm{res}^0) $ induced  as a subspace of $H^1(L,M) $.
 \end{thm}
 
A cohomology class $ \alpha\in H^1(L,M)$ represents a restricted cohomology class if and only if $\tilde{\psi}=0 $ for all representative $\psi\in \alpha$. 

The situation for $C_\mathrm{res}^3(L,M)$ is more complicated. 

\begin{defi}\cite{E96}
    Let $\alpha:\Lambda^3L\to M$ be a skew-symmetric multilinear map on $L$ and $\beta:L\times L\to M $ be a map.  We say that $\beta$ has the $\ast\ast$-property with respect to $\alpha$ if the following conditions hold: \begin{enumerate}
    \item $\beta(g,h)$ is linear with respect to $g$, 
    \item $\beta(g,\lambda h)= \lambda^p\beta(g,h) $ for all $\lambda\in \mathbf{K}$,
    \item \begin{align*}
        &\beta(g,h_1+h_2)\\ &=\beta(g,h_1)+\beta(g,h_2)-\\& \sum_{\substack{l_i\in \left\{1,2\right\},\\ l_1=1,l_2=2}}\frac{1}{\#\left\{l_i=1\right\}}\sum\limits_{j=0}^{p-2}(-1)^j\sum\limits_{k=0}^j\tbinom{j}{k}h_{l_p}\cdots h_{l_{p-k+1}}\\ & \,\,\,\,\,\,\,\,\,\,\,\,\,\,\,\,\,\,\,\,\,\,\,\,\,\,\,\,\,\,\,\,\,\,\,\,\,\,\,\,\,\,\,\,\,\,\,\,\,\,\,\,\,\,\,\,\,\,\,\,\,\,\,\,\,\,\,\,\,\,\,\,\,\,\,\,\,\,\,\,\,\,\,\,\,\,\,\,\,\,\,\,\,\,\,\,\,\,\,\,\,\,\,\,\,\,\,\alpha([g,h_{l_{p-k}},\ldots,h_{l_{p-j+1}}],[h_{l_1},\ldots,h_{l_{p-j-1}}],h_{l_{p-j}}).
    \end{align*}
\end{enumerate} The space of $3$-dimensional cochains is then defined as $$C_\mathrm{res}^3(L,M)=\left\{ (\alpha,\beta)\mid \alpha\in C^3(L,M), \beta:L\times L\to M \quad \text{has the $\ast\ast$-property w.r.t. $\alpha$}  \right\}. $$
\end{defi}

An element $(\varphi,\omega)\in C_\mathrm{res}^2(L,M)$ induces a map $\beta:L\times L\to M $ by the formula $$\beta(g,h)=\varphi(g,h^{[p]})-\sum\limits_{i+j=p-1}(-1)^ih^i\varphi([g,h,\ldots,h],h)+ g\omega(h). $$ If $(\varphi,\omega)\in C_\mathrm{res}^2(L,M) $, the map $\beta$ defined above satisfies the $\ast\ast $-property with respect to $d^2\varphi $. We define the coboundary operator $d_\mathrm{res}^2:C_\mathrm{res}^2(L,M)\to C_\mathrm{res}^3(L,M)$ by the formula $$d_\mathrm{res}^2:(\varphi,\omega)\mapsto (d^2\varphi,\beta) $$

\begin{thm}\cite{E96}
    We have $d_\mathrm{res}^2d_\mathrm{res}^1=0 $. Thus, the quotient $H_\mathrm{res}^2(L,M)=\mathrm{Ker}(d_\mathrm{res}^2)/\mathrm{Im}(d_\mathrm{res}^1) $ is well defined.
\end{thm}
 
We have the following commutative diagram \[\begin{tikzcd}
	0 & {C_\mathrm{res}^0(L,M)} & {C_\mathrm{res}^1(L,M)} & {C_\mathrm{res}^2(L,M)} & {C_\mathrm{res}^3(L,M)} \\
	0 & {C^0(L,M)} & {C^1(L,M)} & {C^2(L,M)} & {C^3(L,M)}
	\arrow[from=1-1, to=1-2]
	\arrow[from=1-2, to=1-3]
	\arrow["{\mathrm{id}}"', from=1-2, to=2-2]
	\arrow[from=1-3, to=1-4]
	\arrow["{\mathrm{id}}"', from=1-3, to=2-3]
	\arrow[from=1-4, to=1-5]
	\arrow[from=1-4, to=2-4]
	\arrow[from=1-5, to=2-5]
	\arrow[from=2-1, to=2-2]
	\arrow[from=2-2, to=2-3]
	\arrow[from=2-3, to=2-4]
	\arrow[from=2-4, to=2-5]
\end{tikzcd}\] where the horizontal maps are the coboundary operators. This diagram shows that we have a map $H_\mathrm{res}^k(L,M)\to H^k(L,M) $ for $k\leq 2$. This map is injective for $k=1$, but the map fails to be injective for $k=2$. Indeed, if $(\varphi,w)\in C^2_\mathrm{res}(L,M) $ is such that $\varphi=d^1f $, we do not have necessary $\tilde{f}=w $. However, we can give a necessary and sufficient condition for this map to be injective. Let $$H^1(L,M)\to \mathrm{Hom}_{p-\mathrm{sl}}(L,M^L), \overline{\varphi}\mapsto \tilde{\varphi} $$ where $\mathrm{Hom}_{p-\mathrm{sl}}(L,M^L) $ is the space of $p$-semilinear maps taking values in $M^L$ and $$M^L=\left\{ u\in M\mid \rho(x)(u)=0, \forall x\in L\right\}.$$ This map is well-defined. Indeed, for every $m\in M$, $\widetilde{d^0(m)}=0$ because $M$ is a restricted $L$-module. Moreover, if $\varphi $ is a $1$-cocycle, then $\tilde{\varphi} $ has the $*$-property with respect to $0$, and hence is a $p$-semilinear map. Finally, $d^2_\mathrm{res}(0,\tilde{\varphi})=0 $ thus $\rho(x)\tilde{\varphi}(y)=0 $ for all $x,y\in L $ which means that $\tilde{\varphi} $ takes values in $M^L$. Assume that this map is surjective and let $\overline{(\varphi,w)}\in H^2_\mathrm{res}(L,M) $ be an element such that $\overline{\varphi}=0 $ in $H^2(L,M)$. There exists $f\in C^1(L,M)$ such that $\varphi=d^1f $. Thus, $(\varphi,w)-d^1_\mathrm{res}f=(0,w-\tilde{f}) $. Since $d^2_\mathrm{res}(0,w-\tilde{f})=0 $, $w-\tilde{f} $ takes values in $M^L$. Moreover, since $w$ and $\tilde{f}$ has the $*$-property with respect to $\varphi$, $w-\tilde{f}$ has the $*$-property with respect to $0$ and so is $p$-semilinear. Hence, there exists $k\in Z^1(L,M)$ such that $\tilde{k}=w-\tilde{f} $. Then, $d^1_\mathrm{res}(f+k)=(\varphi,w) $ which means that $\overline{(\varphi,w)}=0 $ in $H^2_\mathrm{res}(L,M)$ and the map $H^2_\mathrm{res}(L,M)\to H^2(L,M) $ is injective. Conversely, assume that the map $H^2_\mathrm{res}(L,M)\to H^2(L,M) $ is injective. Let $w=L\to M^L $ be a $p$-semilinear map. Since, $w $ is $p$-semilinear, $w$ has the $*$-property with respect to $0$. Moreover, $w $ takes values in $M^L$, thus $d^2_\mathrm{res}(0,w)=0 $ and $(0,w)\in Z^2_\mathrm{res}(L,M) $. Since $\overline{(0,w)} $ is map to $0$, by injectivity of the map, there exists $\varphi\in C^1(L,M)$ such that $w=\tilde{\varphi} $. Hence, $H^1(L,M)\to \mathrm{Hom}_{p-\mathrm{sl}}(L,M^L) $ is surjective.

\begin{prop}\label{propinj}
    The map $H^2_\mathrm{res}(L,M)\to H^2(L,M) $ is injective if and only if the map $H^1(L,M)\to \mathrm{Hom}_{p-\mathrm{sl}}(L,M^L) $ is surjective. 
\end{prop}

\subsection{Algebraic interpretations}

We have algebraic interpretations of low dimensional cohomology for restricted Lie algebras similar to those in the non-restricted case. As $H_\mathrm{res}^1(L,M) $ injects in $H^1(L,M) $, which is equal to $\mathrm{Der}(L)/\mathrm{InnDer}(L) $, Evans gave the following definition.

\begin{defi}\cite{E96}
    Let $L$ be a restricted Lie algebra. A linear map $D:L\to L$ is called a restricted derivation of $L$ if it is a derivation of $L$ and $$D(g^{[p]})=(\ad g)^{p-1}D(g) $$ for all $g\in L$. We denote the set of all restricted derivations of $L$ by $\mathrm{Der}_\mathrm{res.}(L)$.
\end{defi}

For all $g\in L$, $\ad g$ is a restricted derivation and the map $\ad:L\to \mathrm{Der}_\mathrm{res.}(L) $ is a morphism of restricted Lie algebras. This map gives $L$ the structure of an $L$-module and the notation $C_\mathrm{res}^k(L,L) $ will mean this particular structure. 

\begin{prop}
    The space $\mathrm{Der}_\mathrm{res.}(L)/\ad (L) $ is equal to $H_\mathrm{res}^1(L,L)$.
\end{prop}

 As in the case of ordinary Lie algebras, $H_\mathrm{res}^1(L,L)$ has another interpretation.

\begin{defi}\cite{E96}
    If $M,N$ are restricted $L$-modules, then a restricted extension of $N$ by $M$ is an exact sequence $$0\to M\to E\to N\to 0 $$ of restricted $L$-modules and morphisms.
\end{defi}

\begin{lem}\cite{E96}
    If $M$ and $N$ are restricted $L$-modules, then the space $\mathrm{Hom}_\mathbf{K}(N,M)$ is a restricted $L$-module, where we define for $g\in L$, $\varphi\in \mathrm{Hom}_\mathbf{K}(N,M)$ and $n\in N$,  $$(g\varphi)(n)= g\varphi(n)- \varphi(gn).$$
\end{lem} 

\begin{thm}\cite{E96}
    If $M$ and $N$ are restricted $L$-modules, then the set $\mathrm{Ext}(N,M)$ of equivalence classes of restricted extensions of $N$ by $M$ is in one to one correspondence with  \newline$H_\mathrm{res}^1(L,\mathrm{Hom}_\mathbf{K}(N,M))$.
\end{thm}

\begin{cor}\cite{E96}
   The space $H_\mathrm{res}^1(L,M)$ is in one-to-one correspondence with equivalence classes of one dimensional right extensions of the restricted $L$-module $M$.
\end{cor}

\begin{cor}\cite{E96}
   The space  $H_\mathrm{res}^1(L,L)$ is in one-to-one correspondence with equivalence classes of one dimensional right extensions of the restricted Lie algebra $L$.
\end{cor}

As in the non-restricted case, $H_\mathrm{res}^2(L,M) $ has interpretations with extensions of restricted Lie algebras. We say that a restricted Lie algebra $A$ is strongly abelian if in addition to $[A,A]=0$, we also have $A^{[p]}=0$.

\begin{defi}\cite{E96}
    If $L$ is a restricted Lie algebra and $A$ is a strongly abelian restricted Lie algebra, then a restricted extension of $L$ by $A$ is an exact sequence $$0\to A\to E\to L\to 0 $$ of restricted Lie algebras and their morphisms. We say that the extension is central if $A $ is in the center of $E$.
\end{defi} 

\begin{thm}\cite{E96}
Let $L$ be a restricted Lie algebra and $A$ be a strongly abelian Lie algebra regarded as a trivial $L$-module. The set of equivalence classes of restricted central extensions of $L$ by $A$ is in one-to-one correspondence with $H_\mathrm{res}^2(L,A) $. 
\end{thm}

\begin{defi}\cite{E96}
    A restricted infinitesimal deformation of a restricted Lie algebra $L$ is a skew-symmetric bilinear map $\varphi:L\times L\to L$ together with a map $\omega:L\to L$ such that for all $t\in \mathbf{K}$,  the maps $$(g,h)\mapsto [g,h]_t= [g,h]+\varphi(g,h)t $$ $$g\mapsto g^{[p]_t}= g^{[p]}+\omega(g)t $$ give the vector space $L$ a restricted $L$-module structure (mod $t^2$).
\end{defi} 

Equivalently, a restricted infinitesimal deformation is a restricted Lie algebra structure on the tensor product $(\mathbf{K}[t]/(t^2))\otimes_\mathbf{K}L $.

\begin{thm}\cite{E96}
    The equivalence classes of restricted infinitesimal deformations of a restricted Lie algebra $L$ coincide with elements of $H_\mathrm{res}^2(L,L)$.
\end{thm}
 
\section{Cohomology of Lie triple systems}\label{s4}

\subsection{Definition}

In this section, we recall the definition of Lie triple systems, their representations and the cohomology defined by Yamaguti in \cite{Y60}.
In \cite{H60}, Harris defined the cohomology of Lie triple systems by identifying them with the $-1$-eigenspace of involutive automorphism of Lie algebras. Then, he gave interpretations of the low-degree cohomology groups. Yamaguti defined, in a different point of view, the cohomology of Lie triple systems by providing  an explicit complex. Then, he gave an interpretation of low-degree cohomology groups as well. These two points of view are not equivalent in general; a module in Harris's sense does not give rise in general to a module in Yamaguti's sense.

\begin{defi}
    A Lie triple system is a pair  $(T,[\cdot,\cdot,\cdot]) $ consisting of a vector space $T$ and a trilinear map $[\cdot,\cdot,\cdot]:T\times T\times T\to T $ such that, for all $u,v,x,y,z\in T$, the following identities are satisfied: \begin{enumerate}
        \item  $[x,y,z]= -[y,x,z] $,
        \item  $[x,y,z]+ [y,z,x]+[z,x,y]=0 $,
        \item $[u,v,[x,y,z]]=[[u,v,x],y,z]+ [x,[u,v,y],z] + [x,y,[u,v,z]] $. 
    \end{enumerate} 
\end{defi}

\begin{ex}
     Let $L$ be a Lie algebra. Then any subspace $T$ of $L$ closed under the operation $(x,y,z) \mapsto[[x,y],z] $, for all $x,y,z \in T $, is a Lie triple system denoted by $L_\mathrm{trip} $.
\end{ex}

\begin{defi}
     Let $T$ be a Lie triple system. The center of $T$ is defined by \begin{center}
        $Z(T)=\left\{x\in T \mid [x,y,z]=0, \forall y,z\in T \right\}  $.
    \end{center}
\end{defi}

\begin{defi}\cite{Y60}
    Let $T$ be a Lie triple system, $V$ be a $\mathbf{K}$-vector space and $\theta:T\times T\to \mathrm{End}(V) $ be a bilinear map. A pair $(V,\theta)$ is called a representation of $T$, if for all $x,y,z,t\in T$, the following identities hold:
           \begin{enumerate}
           \item 
               \begin{equation}\label{rep1}
                   \theta(z,t)\theta(x,y)-\theta(y,t)\theta(x,z)-\theta(x,[y,z,t]) + D(y,z)\theta(x,t)=0,
               \end{equation}
               \item \begin{equation}\label{rep2}
                   \theta(z,t)D(x,y)-D(x,y)\theta(z,t)+\theta([x,y,z],t)+ \theta(z,[x,y,t])=0,
               \end{equation}
          \end{enumerate} where $D(x,y)=\theta(y,x)-\theta(x,y) $.
\end{defi}

\begin{ex}
    The pair $(T,\theta)$, where $\theta(x,y)(z)=[z,x,y] $, is a representation, called the adjoint representation.
\end{ex}

Here is a computational result which will be used in the next two sections.

\begin{prop}\label{p5.14}\cite{BM26}
    Let $T$ be a Lie triple system. Let $x,y\in T $ and $\theta$ be a representation of $T$. Then we have \begin{align*}
      &  \theta(y,(x,y,\dots,y))= \sum\limits_{k=0}^n\tbinom{2n}{2k}\theta(y,y)^{n-k}\theta(y,x)\theta(y,y)^k - \sum\limits_{k=0}^{n-1}\tbinom{2n}{2k+1}\theta(y,y)^k\theta(x,y)\theta(y,y)^{n-k} ,
    \end{align*} and \begin{align*}
       & \theta((x,y,\dots,y),y)= \sum\limits_{k=0}^n\tbinom{2n}{2k}\theta(y,y)^k\theta(x,y)\theta(y,y)^{n-k} - \sum\limits_{k=0}^{n-1}\tbinom{2n}{2k+1}\theta(y,y)^{n-k}\theta(y,x)\theta(y,y)^k ,
    \end{align*} where $y$ occurs $2n$ times in the brackets.
\end{prop}

\begin{prop}\label{p5.16}\cite{BM26}
    Let $T$ be a Lie triple system, $x,y,z\in T $. Then we have \begin{align*}
        \theta((x,y,\dots,y),z)= &\theta(y,z)\sum\limits_{k=1}^n \tbinom{2n}{2k}\theta(y,y)^{k-1}\theta(x,y)\theta(y,y)^{n-k} \\ &- \theta(y,z)\sum\limits_{k=0}^{n-1}\tbinom{2n}{2k+1}\theta(y,y)^{n-k-1}\theta(y,x)\theta(y,y)^k + \theta(x,z)\theta(y,y)^n ,
    \end{align*} where $y$ occurs $2n$ times in the bracket.
\end{prop}

\begin{prop}
    Let $(V,\theta)$ be a representation of a Lie triple system $T$. Consider the operation $[\cdot,\cdot,\cdot]_V: (T\oplus V)\times (T\oplus V)\times (T\oplus V)\to T\oplus V  $ defined by \begin{center}
        $[(x,u),(y,v),(z,w)]_V=([x,y,z], \theta(y,z)(u)-\theta(x,z)(v)+ D(x,y)(w)) $.
    \end{center} Then, $T\oplus V$ is a Lie triple system, called the semi-direct product of $T$ by the representation $(V,\theta)$.
\end{prop}

Yamaguti defined $n$-cochains and coboundary maps for all $n\in \mathbf{N}$.

Let $(V,\theta) $ be a representation of a Lie triple system $T$. For $n\geq 1$, denote by $C^{n}(T,V) $ the vector space of all $n$-linear maps $\omega:T\times \cdots\times T\to V$ satisfying $$\omega(x_1,x_2,\ldots,x_{n-3},x,x,y)=0 $$ and \begin{equation}\label{eqcycl}
    \omega(x_1,x_2,\ldots,x_{n-3},x,y,z)+\omega(x_1,x_2,\ldots,x_{n-3},y,z,x)+\omega(x_1,x_2,\ldots,x_{n-3},z,x,y)=0. 
\end{equation} The Yamaguti coboundary operator $\delta^{n}:C^{n}(T,V)\to C^{n+2}(T,V) $ is defined as follows: If $\omega\in C^0(T,V)=V$,
\begin{align*}
    \delta^0\omega(x_1,x_2)= \theta(x_1,x_2)(\omega).
\end{align*} If $\omega\in C^{2n-1}(T,V) $,
\begin{align*}
    \delta^{2n-1}\omega(x_1,&\ldots,x_{2n+1})\\=& \theta(x_{2n},x_{2n+1})\omega(x_1,\ldots,x_{2n-1}) - \theta(x_{2n-1},x_{2n+1})\omega(x_1,\ldots,x_{2n-2},x_{2n}) \\& + \sum\limits_{k=1}^n(-1)^{n+k}D(x_{2k-1},x_{2k})\omega(x_1,\ldots,\widehat{x_{2k-1}},\widehat{x_{2k}},\ldots,x_{2n+1}) \\ &+ \sum\limits_{k=1}^n\sum\limits_{j=2k+1}^{2n+1}(-1)^{n+k+1}\omega(x_1,x_2,\ldots,\widehat{x_{2k-1}},\widehat{x_{2k}},\ldots,[x_{2k-1},x_{2k},x_j],\ldots,x_{2n+1}).
\end{align*}  If $\omega\in C^{2n}(T,V) $,
 \begin{align*}
    \delta^{2n}\omega(y,x_1,x_2&,\ldots,x_{2n+1})\\ =& \theta(x_{2n},x_{2n+1})\omega(y,x_1,\ldots,x_{2n-1}) - \theta(x_{2n-1},x_{2n+1})\omega(y,x_1,\ldots,x_{2n-2},x_{2n}) \\& + \sum\limits_{k=1}^n(-1)^{n+k}D(x_{2k-1},x_{2k})\omega(y,x_1,\ldots,\widehat{x_{2k-1}},\widehat{x_{2k}},\ldots,x_{2n+1}) \\ &+ \sum\limits_{k=1}^n\sum\limits_{j=2k+1}^{2n+1}(-1)^{n+k+1}\omega(y,x_1,x_2,\ldots,\widehat{x_{2k-1}},\widehat{x_{2k}},\ldots,[x_{2k-1},x_{2k},x_j],\ldots,x_{2n+1}),
\end{align*}  where $\,\hat{}\,$ denotes the omission. \\ We define $Z^{n}(T,V)= \mathrm{Ker}(\delta^{n}) $ the $n$-cocycles and $B^{n}=\mathrm{Im}(\delta^{n-2}) $ the $n$-coboundaries.

\begin{thm}\cite{Y60}
    Let $T$ be a Lie triple system and $(V,\theta) $ be a representation of $T$. With the above notation, $\delta^{n+2}\circ \delta^{n} =0$ for all $n\in \mathbf{N}$. Hence, we get the cohomology group by $H^{n}(T,V)= Z^{n}(T,V)/B^{n}(T,V) $.
\end{thm}

In low degrees, the operator $\delta$ is given as follows: 

If $f\in C^1(T,V)$,  \begin{equation}\label{lwo1}
    \delta^1 f(x_1,x_2,x_3)= \theta(x_2,x_3)(f(x_1))- \theta(x_1,x_3)(f(x_2))+D(x_1,x_2)(f(x_3))- f([x_1,x_2,x_3]). 
\end{equation}If $f\in C^2(T,V)$, \begin{align}
    \delta^2 f(x_1,x_2,x_3,x_4)= \theta&(x_3,x_4)f(x_1,x_2)-\theta(x_2,x_4)f(x_1,x_3)+ D(x_2,x_3)f(x_1,x_4)\notag \\&- f(x_1,[x_2,x_3,x_4]).\label{lwo2}
\end{align}  If $f\in C^3(T,V)$, \begin{align}
    \delta^3 f(x_1,x_2,x_3,x_4,x_5)=& \theta(x_4,x_5)f(x_1,x_2,x_3)-\theta(x_3,x_5)f(x_1,x_2,x_4)-D(x_1,x_2)f(x_3,x_4,x_5)\notag\\ &+ D(x_3,x_4)f(x_1,x_2,x_5)+ f([x_1,x_2,x_3],x_4,x_5)+ f(x_3,[x_1,x_2,x_4],x_5)\notag\\ &+f(x_3,x_4,[x_1,x_2,x_5])- f(x_1,x_2,[x_3,x_4,x_5]).
\end{align}

We propose here to extend  Yamaguti's complex in order to define the first cohomology group $H^1(T,V)$ and give an algebraic interpretation of $H^1(T,T) $.

We define $C^{-1}(T,V)= T\otimes V $ and $$\delta^{-1}:C^{-1}(T,V)\to C^1(T,V) ,\delta^{-1}(x\otimes u)=(t\mapsto \theta(x,t)(u)). $$ Then, we have $\delta^1\circ\delta^{-1}=0 $. Indeed, let $x,x_1,x_2,x_3\in T, u\in V$, \begin{align*}
    \delta^1\delta^{-1}(x\otimes u)(x_1,x_2,x_3)=& \theta(x_2,x_3)(\theta(x,x_1)(u))- \theta(x_1,x_3)(\theta(x,x_2)(u))\\&+D(x_1,x_2)(\theta(x,x_3)(u))- \theta(x,[x_1,x_2,x_3])(u) \\ =&0,
\end{align*} since $\theta$ is a representation. Hence, we obtain the following complex $$C^{-1}(T,V)\to C^1(T,V)\to C^3(T,V)\to \cdots \to C^{2n+1}(T,V). $$

\begin{defi}
    Let $T$ be a Lie triple system and $(V,\theta),(W,\theta')$ be two representations of $T$. A morphism of representations is a linear map $\varphi:V\to W $ such that $$\varphi(\theta(x,y)(v))= \theta'(x,y)(\varphi(v)) $$ for all $x,y\in T, v\in V$.
\end{defi}

A natural question is whether the cohomologies induced by two isomorphic representations are isomorphic. Let $(V,\theta)$ and $(W,\theta')$ be two representations of a Lie triple system $T$. Let $f:V\to W$ be an isomorphism, which is a morphism of representations. Then, for all $n\in \mathbf{N}$, there is a natural map $$C^n(T,V)\to C^n(T,W), \,\, w\mapsto f\circ w ,$$ which is an isomorphism. We define $C^{-1}(T,V)\to C^{-1}(T,W), x\otimes v\mapsto x\otimes f(v) $ and it is also an isomorphism. Since $f$ is a morphism of representations, we easily see that the following diagram is commutative \[\begin{tikzcd}
	{C^{n-2}(T,V)} & {C^n(T,V)} & {C^{n+2}(T,V)} \\
	{C^{n-2}(T,W)} & {C^{n}(T,W)} & {C^{n+2}(T,W)}
	\arrow["{\delta^{n-2}}", from=1-1, to=1-2]
	\arrow["\simeq", from=1-1, to=2-1]
	\arrow["{\delta^n}", from=1-2, to=1-3]
	\arrow["\simeq", from=1-2, to=2-2]
	\arrow["\simeq", from=1-3, to=2-3]
	\arrow["{\delta^{n-2}}"', from=2-1, to=2-2]
	\arrow["{\delta^n}"', from=2-2, to=2-3]
\end{tikzcd}\] and  induces an isomorphism $H^n(T,V)\simeq H^n(T,W) $.

\subsection{Connection between  cohomology of a Lie algebras and  cohomology of the induced Lie triple systems}\label{subs4}

Another natural question is about the connection between  Chevalley-Eilenberg cohomology of Lie algebras and  Yamaguti cohomology of Lie triple systems. Let $L$ be a Lie algebra, and let $\rho:L\to \mathrm{End}(M)$ be a representation of $L$. We know that $L_\mathrm{trip} $ is a Lie triple system and that the pair $(M,\theta) $, where $\theta(x,y)(m)=\rho(y)\rho(x)(m) $, is a representation of $L_\mathrm{trip} $. Moreover,  we may define the cochains $C^n(L,M) $ and $C^n(L_\mathrm{trip},M) $. Let $\varphi\in C^2(L,M) $ be a $2$-cochain of $L$. We define $\phi_2{\varphi}:T\times T\times T\to M $ by \begin{equation}\label{eqtilde}
    \phi_2\varphi(x,y,z)=\rho(z)\varphi(x,y)-\varphi([x,y],z).
\end{equation}  Then, we have  for $\varphi\in C^1(L,M)$, \begin{equation}\label{eqcobor1}
    \phi_2d^1\varphi=\delta^1\varphi.
\end{equation}  Thus, the map $\phi_2$ commutes with the coboundary operators $d^1$ and $\delta^1 $. Moreover, let $\varphi\in C^3(L,M) $ be a $3$-cochain of $L$. We define $\phi_3{\varphi}:T^{\otimes 5}\to M $ by $$\phi_3\varphi(x_1,x_2,x_3,x_4,x_5)= \rho(x_5)\varphi([x_1,x_2],x_3,x_4)- \varphi([x_1,x_2],[x_3,x_4],x_5 ). $$ We have  \begin{equation}\label{eqcobor2}
    \delta^3(\phi_2{\varphi})=\phi_3({d^2\varphi}).
\end{equation}  However, for arbitrary $\varphi\in C^2(L,M)$ and $\psi\in C^3(L,M)$, the maps $\phi_2\varphi$ and $\phi_3\psi$ do not necessarily satisfy the cochain conditions of the Yamaguti complex. Nevertheless, we have that $$\phi_2{\varphi}(x,y,z)+\phi_2{\varphi}(y,z,x)+\phi_2{\varphi}(z,x,y)=-d^2\varphi(x,y,z). $$ Hence, if $\varphi\in Z^2(L,M) $ then $\phi_2\varphi\in C^3(L_\mathrm{trip},M) $. Together with \eqref{eqcobor1} and \eqref{eqcobor2} this shows that $\phi_2$ induces a map $Z^2(L,M)\to Z^3(L_\mathrm{trip},M) $. Hence, we obtain the following commutative diagram \[\begin{tikzcd}
	{C^1(L,M)} & {Z^2(L,M)} & {C^3(L,M)} \\
	{C^1(L_\mathrm{trip},M)} & {Z^3(L_\mathrm{trip},M)} & {\mathrm{Hom}(T^{\otimes 5},M)}
	\arrow["{d^1}", from=1-1, to=1-2]
	\arrow["{\mathrm{id}}"', from=1-1, to=2-1]
	\arrow["{d^2}", from=1-2, to=1-3]
	\arrow[from=1-2, to=2-2]
	\arrow[from=1-3, to=2-3]
	\arrow["{\delta^1}"', from=2-1, to=2-2]
	\arrow["{\delta^3}"', from=2-2, to=2-3]
\end{tikzcd}\] Therefore, $\phi_2$ induces a linear map $\overline{\phi_2}:H^2(L,M)\to H^3(L_\mathrm{trip},M),\overline{\varphi}\mapsto\overline{\phi_2{\varphi}} $, which in each class of $\varphi\in H^2(L,M)$ associates the class of $\phi_2{\varphi}$ in $H^3(L_\mathrm{trip},M)$. In general, this map is neither injective nor surjective.

\begin{ex}
    Let $L=\mathbf{K}e_1\oplus\mathbf{K}e_2$ be an abelian Lie algebra and assume that $\mathbf{K} $ is endowed with the trivial $L$-module structure. Since $d^1=0 $, $$H^2(L,\mathbf{K})=\Lambda^2L^*\neq 0. $$ Moreover, $\phi_2(\varphi)=0$ for every $\varphi\in C^2(L,\mathbf{K})$. Hence $H^2(L,\mathbf{K})\to H^3(L_\mathrm{trip},\mathbf{K})$ is not injective. On the other hand, $L_\mathrm{trip}$ is an abelian Lie triple system and the induced representation is trivial. Therefore, $H^3(L_\mathrm{trip},\mathbf{K})=C^3(L_\mathrm{trip},\mathbf{K})\neq 0 $. Since $\overline{\phi_2}=0 $, the map is not surjective either.
\end{ex}

Let $\varphi\in Z^2(L,M)$ such that $\overline{\phi_2\varphi}=0. $ Then there exists $f:L\to M$ such that $\phi_2\varphi=\delta^1f $. Since $\delta^1f=\phi_2(d^1f)$, the cocycle $g=\varphi-d^1f $ represents the same cohomology class as $\varphi $ and satisfies $\phi_2(g)=0$. Therefore, every element of $\mathrm{Ker}(\overline{\phi_2})$ admits a representative $g\in Z^2(L,M)$ satisfying \begin{equation}\label{eqI}
    \rho(z)g(x,y)=g([x,y],z),
\end{equation} for all $x,y,z\in L$. Assume now that $L=[L,L]+Z(L) $ and let $g$ be a cocycle satisfying \eqref{eqI}. Consider the surjective linear map $b:\Lambda^2L\to [L,L] $. If $(x,y)\in \mathrm{Ker}(b)$, then \eqref{eqI} implies that $\rho(z)g(x,y)=g([x,y],z)=0 $ for every $z\in L$. Hence, $g(x,y)\in M^L $, where $$M^L=\left\{ u\in M\mid \rho(x)(u)=0, \forall x\in L\right\}.$$ Let $\pi:M\to M/M^L $ denote the canonical projection. Then, $\pi\circ g $ vanishes on $\mathrm{Ker}(b) $ and therefore factors uniquely through $b$. Thus, there exists a unique linear map $\tilde{h}:[L,L]\to M/M^L $ such that the following diagram is commutative \[\begin{tikzcd}
	{\Lambda^2L} & M \\
	{[L,L]} & {M/M^L}
	\arrow["g", from=1-1, to=1-2]
	\arrow["b"', from=1-1, to=2-1]
	\arrow["\pi", from=1-2, to=2-2]
	\arrow["{\tilde{h}}"', dashed, from=2-1, to=2-2]
\end{tikzcd}\] Since $L=[L,L]+Z(L) $, there exists a subspace $C\subset Z(L)$ such that $L=[L,L]\oplus C $. We extend $\tilde{h}:[L,L]\to M/M^L $ to a linear map $\overline{h}:L\to M/M^L $ by setting $\overline{h}([x,y]+c)=\tilde{h}([x,y]) $ for all $x,y,c\in L$. The quotient $M/M^L$ is naturally an $L$-module via $\rho(x)(\pi(m))=\pi(\rho(x)(m)) $, which is well-defined since $M^L$ is an $L$-submodule. By construction, $\overline{h}([x,y])=\pi(g(x,y)) $ for all $x,y\in L$. Let $u=\sum\limits_i[x_i,y_i] \in [L,L]$. Then, for every $z\in L$, \begin{align*}
     \rho(z)\overline{h}(u)&=\rho(z)\sum\limits_i\pi(g(x_i,y_i))\\&=\pi(\sum\limits_ig([x_i,y_i],z))\\&= \pi (g(u,z)).
\end{align*} Thus, for $u_1,u_2\in [L,L]$, \begin{align*}
    d^1\overline{h}(u_1,u_2)&= -\rho(u_1)\overline{h}(u_2)+\rho(u_2)\overline{h}(u_1)+\overline{h}([u_1,u_2])\\ &= 3\pi(g(u_1,u_2)).
\end{align*} Moreover, if $u\in [L,L],c\in C $ then $[u,c]=0 $ and $\rho(y)g(c,u)=g([c,u],y)=0 $ for all $y\in L$, which means that $g(c,u)\in M^L$ and $\pi(g(c,u))=0 $. Thus, \begin{align*}
    d^1\overline{h}(u,c)&= -\rho(u)\overline{h}(c)+\rho(c)\overline{h}(u)+\overline{h}([u,c])\\ &=0\\ &= 3\pi(g(u,c)).
\end{align*}  Finally, if $c_1,c_2\in C $, we have $$d^1\overline{h}(c_1,c_2)=0=3\pi(g(c_1,c_2)). $$ Combining the three cases, we obtain $d^1\overline{h}=3\pi\circ g $. If the characteristic of the base field $\mathbf{K}$ is not $3$, we can conclude that $\pi\circ g=\frac{1}{3}d^1\overline{h} $. We have the following short exact sequence of $L$-modules \[\begin{tikzcd}
	0 & {M^L} & M & {M/M^L} & 0
	\arrow[from=1-1, to=1-2]
	\arrow["\iota", from=1-2, to=1-3]
	\arrow["\pi", from=1-3, to=1-4]
	\arrow[from=1-4, to=1-5]
\end{tikzcd}\] As $\iota $ and $\pi$ are morphism of modules, this induces a short exact sequence of complexes \[\begin{tikzcd}
	0 & {C^\bullet (L,M^L)} & {C^\bullet(L,M)} & {C^\bullet(L,M/M^L)} & 0
	\arrow[from=1-1, to=1-2]
	\arrow["{\iota^\bullet}", from=1-2, to=1-3]
	\arrow["{\pi^\bullet}", from=1-3, to=1-4]
	\arrow[from=1-4, to=1-5]
\end{tikzcd}\] which gives in particular the following exact sequence of cohomology \[\begin{tikzcd}
	 {H^2(L,M^L)} & {H^2(L,M)} & {H^2(L,M/M^L)} 
	\arrow["{\iota^*}", from=1-1, to=1-2]
	\arrow["{\pi^*}", from=1-2, to=1-3]
\end{tikzcd}\] We have seen that if $g\in Z^2(L,M)$ satisfies \eqref{eqI}, then the class of $g $ is in the image $\iota^*$. Then, the class of $g $ is in the kernel of $\pi^* $. Thus, we have obtain the following result.

\begin{prop}\label{propinj1}
    If $L=[L,L]+Z(L) $ and the characteristic of the base field is not $3$, then $$\mathrm{Ker}(\overline{\phi_2})\subset \mathrm{Im}(H^2(L,M^L)\to H^2(L,M)). $$ In particular, if $L=[L,L]+Z(L) $, $\mathrm{char}(\mathbf{K})\neq 3$, and the map $H^2(L,M^L)\to H^2(L,M) $ is zero, then the map $H^2(L,M)\to H^3(L_\mathrm{trip},M)$ is injective.
\end{prop}

\begin{cor}
    If $L=[L,L] $, $M^L=0$ and $\mathrm{char}(\mathbf{K})\neq 3 $, then the map $H^2(L,M)\to H^3(L_\mathrm{trip},M)$ is injective.
\end{cor}

We now give another sufficient condition for injectivity, which is independent of the characteristic of the base field. Let $g:L\times L\to M$ satisfying \eqref{eqI} and assume that $M^{Z(L)}=0 $. Let $c\in Z(L),x,y\in L $. By \eqref{eqI}, $\rho(y)g(c,x)=g([c,x],y)=0 $. Hence, $g(c,x)\in M^L\subset M^{Z(L)}$ and therefore \begin{equation}\label{eqjsp}
    g(c,x)=0.
\end{equation} Moreover, $\rho(c)g(x,y)=g([x,y],c)=0 $ by \eqref{eqjsp}. Thus, $g(x,y)\in M^{Z(L)} $ and consequently $g(x,y)=0 $. 

\begin{prop}
    If $M^{Z(L)}=0 $, then the map $H^2(L,M)\to H^3(L_\mathrm{trip},M) $ is injective.
\end{prop}

We can similarly define a map $H^1(L,M)\to H^2(L_\mathrm{trip},M)$. For $\varphi\in C^1(L,M)$, define $\phi_1\varphi\in C^2(L_\mathrm{trip},M)$ by $$\phi_1\varphi(x,y)=-\rho(y)\varphi(x). $$ For every $m\in M$, $$\phi_1(d^0m)=\delta^0. $$ Moreover, $\phi_1\varphi$ is a $2$-cocycle for every $\varphi\in C^1(L,M)$. Hence, $\phi_1$ induces a linear map $$\overline{\phi_1}:H^1(L,M)\to H^2(L_\mathrm{trip},M), \overline{\varphi}\mapsto \overline{\phi_1\varphi}.  $$ Let $\overline{\varphi}\in \mathrm{Ker}(\overline{\phi_1})$ with $\varphi\in Z^1(L,M)$. Then, there exists $m\in M$ such that $\phi_1\varphi=\delta^0(m)$. Therefore, $-\rho(y)\varphi(x)=\rho(y)\rho(x)(m) $ for all $x,y\in L$. Then, \begin{equation}\label{eqML}
    \rho(y)(\varphi(x)+\rho(x)(m))=0
\end{equation} for all $x,y\in L$. Define $$g:L\to M, x\mapsto \varphi(x)+\rho(x)(m).$$ Equation \eqref{eqML} implies that $g(x)\in M^L$ for every $x\in L$. Since $\varphi$ is a $1$-cocycle, so is $g$. Moreover, $\varphi(x)-g(x)=\rho(x)(m)=d^0(-m)$, so $\varphi$ and $g$ represent the same class in $H^1(L,M)$. Conversely, suppose that $\overline{\varphi}$ admits a representative $g:L\to M^L$. Then $\rho(y)g(x)=0 $ for all $x,y\in L$ and hence $\phi_1(g)=0$. Therefore, $\overline{\phi_1g}=\overline{\phi_1\varphi}=0 $. We conclude that $\mathrm{Ker}\overline{\phi_1}= \mathrm{Im}(H^1(L,M^L)\to H^1(L,M))$. 

\begin{prop}\label{propinj2}
    The map $H^1(L,M)\to H^2(L_\mathrm{trip},M)$ is injective if and only if the map $H^1(L,M^L)\to H^1(L,M) $ is zero.
\end{prop}

\hyperref[propinj1]{Proposition~\ref*{propinj1}} and \hyperref[propinj2]{Proposition~\ref*{propinj2}} show that, in both cases, the obstruction to injectivity is related to cohomology classes with values in the invariant submodule $M^L$.

\subsection{Algebraic interpretations}

\subsubsection{Derivations}

\begin{defi}
    Let $T$ be a Lie triple system. Then a linear map $D:T\to T$ is a derivation if, for all $x,y,z \in T$, we have \begin{center}
        $D([x,y,z])= [D(x),y,z]+ [x,D(y),z]+ [x,y,D(z)] $.
    \end{center}  We denote by $\mathrm{Der}(T)$ the vector space of all  derivations.
\end{defi}

\begin{ex}
    Maps of the form $x\mapsto\sum\limits_i[a_i,b_i,x]  $ are derivations, called inner derivations. We denote by $\mathrm{InnDer}(T)$ the vector space of all inner derivations.
\end{ex}

\begin{prop}
   The first cohomology group $H^1(T,T) $ is equal to the space $\mathrm{Der}(T)/\mathrm{InnDer}(T) $.
\end{prop}

\begin{proof}
    We consider here the adjoint representation $\theta(x,y)(z)=[z,x,y] $. If $f\in C^1(T,T) $, \begin{align*}
    \delta f &= [f(x_1),x_2,x_3]- [f(x_2),x_1,x_3]+ [f(x_3),x_2,x_1]- [f(x_3),x_1,x_2]- f([x_1,x_2,x_3]) \\ &= [f(x_1),x_2,x_3]+ [x_1,f(x_2),x_3]+ [x_1,x_2,f(x_3)]- f([x_1,x_2,x_3]).
\end{align*} Thus, $f\in C^1(T,T)$ is a $1$-cocycle if and only if $f $ is a derivation of the Lie triple system. We can see that $\mathrm{InnDer}(T)=\delta^{-1}C^1(T,T) $ and the result follows.
\end{proof}

\subsubsection{Extensions of modules}

In this subsection, we give an interpretation of the second cohomology group in terms of extensions of modules. It will be shown that $H^2(T,V) $ does not classify extensions of modules, it naturally classifies a weaker notion that will be defined.

\begin{defi}
    Let $T$ be a Lie triple system. A one-dimensional right extension of a module $V$ is an exact sequence of modules and their morphisms \[\begin{tikzcd}
	0 & V & W & {\mathbf{K}} & 0
	\arrow[from=1-1, to=1-2]
	\arrow["f", from=1-2, to=1-3]
	\arrow["g", from=1-3, to=1-4]
	\arrow[from=1-4, to=1-5]
\end{tikzcd}\] where $\mathbf{K}$ has the structure of a trivial $T$-module.
\end{defi}

Two extensions are equivalent if the following diagram is commutative \[\begin{tikzcd}
	& 0 & V & W & {\mathbf{K}} & 0 \\
	& 0 & V & {W'} & {\mathbf{K}} & 0 \\
	{}
	\arrow[from=1-2, to=1-3]
	\arrow["f", from=1-3, to=1-4]
	\arrow["{\mathrm{id}}"', from=1-3, to=2-3]
	\arrow["g", from=1-4, to=1-5]
	\arrow["h", from=1-4, to=2-4]
	\arrow[from=1-5, to=1-6]
	\arrow["{\mathrm{id}}", from=1-5, to=2-5]
	\arrow[from=2-2, to=2-3]
	\arrow["{f'}"', from=2-3, to=2-4]
	\arrow["{g'}"', from=2-4, to=2-5]
	\arrow[from=2-5, to=2-6]
\end{tikzcd}\] where $h:W\to W'$ is a morphism of representations. 

Let \[\begin{tikzcd}
	0 & V & W & {\mathbf{K}} & 0
	\arrow[from=1-1, to=1-2]
	\arrow["f", from=1-2, to=1-3]
	\arrow["g", from=1-3, to=1-4]
	\arrow[from=1-4, to=1-5]
\end{tikzcd}\] be a one-dimensional right extension of $V$. We denote by $(V,\theta) $ and $(W,\theta') $ the representations. We construct an element $\varphi\in C^2(T,V) $ in the following way. Let $u\in W$ such that $g(u)=1 $ and define $\varphi:T\times T\to V$ by $$\varphi(x,y)= f^{-1}(\theta'(x,y)(u)), $$ for all $x,y\in T$. The map $\varphi $ is well defined. Indeed, since  $g$ is a morphism of $T$-modules, $g(\theta'(x,y)(w))=0 $ for all $w\in W$. Hence $\theta'(x,y)(u)\in \mathrm{Ker}(g)$, so it has a pre-image under $f$. The map $\varphi$ is a $2$-cocycle of the Yamaguti cohomology. Indeed, \begin{align*}
    f\delta^2 \varphi(x_1,x_2,x_3,x_4)= &f(\theta (x_3,x_4)\varphi(x_1,x_2)) - f(\theta(x_2,x_4)\varphi(x_1,x_3)) + f(D(x_2,x_3)\varphi(x_1,x_4))\\ &- f(\varphi(x_1,[x_2,x_3,x_4])) \\ =& \theta'(x_3,x_4)\theta'(x_1,x_2)(u)- \theta'(x_2,x_4)\theta'(x_1,x_3)(u) + D'(x_2,x_3)\theta'(x_1,x_4)(u)\\ & - \theta'(x_1,[x_2,x_3,x_4])(u) \\ =& 0, 
\end{align*} because $f$ is a morphism of modules and $\theta'$ is a representation. So $\delta^2\varphi(x_1,x_2,x_3,x_4)=0 $ because $f$ is injective. It remains to show that the class of $\varphi\in H^2(T,V)$ depends only on the equivalence class of the extension.

Let  \[\begin{tikzcd}
	& 0 & V & W & {\mathbf{K}} & 0 \\
	& 0 & V & {W'} & {\mathbf{K}} & 0 \\
	{}
	\arrow[from=1-2, to=1-3]
	\arrow["f", from=1-3, to=1-4]
	\arrow["{\mathrm{id}}"', from=1-3, to=2-3]
	\arrow["g", from=1-4, to=1-5]
	\arrow["h", from=1-4, to=2-4]
	\arrow[from=1-5, to=1-6]
	\arrow["{\mathrm{id}}", from=1-5, to=2-5]
	\arrow[from=2-2, to=2-3]
	\arrow["{f'}"', from=2-3, to=2-4]
	\arrow["{g'}"', from=2-4, to=2-5]
	\arrow[from=2-5, to=2-6]
\end{tikzcd}\] be two equivalent one-dimensional extensions. We denote by $(V,\theta),(W,\theta_W)$ and $(W',\theta_{W'}) $ the representations. Let $\varphi $ and $\varphi'$ be the two cocycles defined as above and let us show that $\varphi $ and $\varphi'$ differ by a $2$-coboundary. Let $u\in W$, $u'\in W'$ such that $g(u)=1$ and $g'(u')=1$. We define $$\varphi(x,y)= f^{-1}(\theta_W(x,y)(u)), \ \varphi'(x,y)=f'^{-1}(\theta_W'(x,y)(u')). $$ The diagram is commutative, so $h\circ f=f' $ and $g'\circ h=g $ and $h$ is a morphism of modules, so $h(\theta_W(x,y)(u))=\theta'_W(x,y)(h(u))$. Thus, \begin{align*}
    f^{-1}(\theta_W(x,y)(u)) &= f'^{-1}(h(\theta_W(x,y)(u)))\\ & = f'^{-1}(\theta_W'(x,y)(h(u)))
\end{align*} and $h(u)$ is such that $g'(h(u))=1$. So, it remains to show that if we have a one-dimensional extension \[\begin{tikzcd}
	0 & V & W & {\mathbf{K}} & 0
	\arrow[from=1-1, to=1-2]
	\arrow["f", from=1-2, to=1-3]
	\arrow["g", from=1-3, to=1-4]
	\arrow[from=1-4, to=1-5]
\end{tikzcd}\] and $u,v\in W$ are such that $g(u)=g(v)=1$, the two cocycles induced are cohomologous. Let $x,y\in T$, $\varphi(x,y)=f^{-1}(\theta_W(x,y)(u)) $ and $\varphi'(x,y)=f^{-1}(\theta_W(x,y)(u')) $. As $g(u)=g(u') $, there exists $v\in V$ such that $f(v)=u-u'$. So \begin{align*}
    \varphi(x,y) -\varphi'(x,y) &= f^{-1}(\theta_W(x,y)(u-u'))\\&= f^{-1}(\theta_W(x,y)(f(v)))\\&= \theta(x,y)(v)\\ & =\delta(v)(x,y).
\end{align*} Finally, we have  $\varphi(x,y)= f'^{-1}(\theta'_W(x,y)(h(u))) $ which is in the same cohomology class as $\varphi' $.  Thus, $\varphi $ and $\varphi'$ are indeed in the same cohomology class. We have constructed a map which associates to each equivalence class of a one-dimensional right extension of a representation $(V,\theta)$  a $2$-cocycle of the cohomology with coefficients in this representation. 

Let $(V,\theta)$ be a representation of $T$ and $\varphi$ be a $2$-cocycle in the Yamaguti cohomology of $T$ with coefficients in $V$. We want to construct a one-dimensional right extension of $V$. We set $W= V\oplus\mathbf{K} $ and we have the exact sequence of vector spaces \[\begin{tikzcd}
	& 0 & V & {V\oplus \mathbf{K}} & {\mathbf{K}} & 0 
	\arrow[from=1-2, to=1-3]
	\arrow["\iota", from=1-3, to=1-4]
	\arrow["\pi", from=1-4, to=1-5]
	\arrow[from=1-5, to=1-6]
\end{tikzcd}\] We have to define a structure of $T$-module on $V\oplus \mathbf{K}$ such that this exact sequence is an exact sequence of $T$-modules. Let $x,y\in T,u\in V,\lambda\in \mathbf{K}$, we define $$\theta'(x,y)((u,\lambda))= (\lambda\varphi(x,y)+ \theta(x,y)(u),0). $$ Then we have, for all $x,y,z,t\in T,u\in V$ and $\lambda\in\mathbf{K}$, \begin{align*}
    \theta'(z,t)\theta&'(x,y)((u,\lambda)- \theta'(y,t)\theta'(x,z)((u,\lambda))- \theta'(x,[y,z,t])((u,\lambda))+ D'(y,z)\theta'(x,t)((u,\lambda))\\ =& \theta'(z,t)(\lambda\varphi(x,y)+\theta(x,y)(u),0)- \theta'(y,t)(\lambda\varphi(x,z)+\theta(x,z)(u),0) \\& - (\lambda\varphi(x,[y,z,t])+\theta(x,[y,z,t])(u),0)+ D'(y,z)(\lambda\varphi(x,t)+\theta(x,t)(u),0)\\ =& (\theta(z,t)(\lambda\varphi(x,y)+\theta(x,y)(u)),0) - (\theta(y,t)( \lambda\varphi(x,z)+\theta(x,z)(u)),0)\\ &-(\lambda\varphi(x,[y,z,t])+\theta(x,[y,z,t])(u),0) + (D(y,z)(\lambda\varphi(x,t)+\theta(x,t)(u)),0)\\ =& (0,0), 
\end{align*} because $\theta$ is a representation and $\varphi$ is a $2$-cocycle. Thus, $\theta'$ satisfies \eqref{rep1}. The condition to be a $2$-cocycle is not sufficient to ensure that $\theta'$ satisfies \eqref{rep2}. So we introduce the following definition.

\begin{defi}
    Let $T$ be a Lie triple system. A pseudo-representation of $T$ is a pair $(V,\theta) $, where $V$ is a vector space and $\theta:T\times T\to \mathrm{End}(V)$ is a bilinear map such that  $$\theta(z,t)\theta(x,y)-\theta(y,t)\theta(x,z)-\theta(x,[y,z,t]) + D(y,z)\theta(x,t)=0$$ for all $x,y,z,t\in T$.
\end{defi}

Morphisms between pseudo-modules are defined in the same way as morphisms of representations.  A one-dimensional right pseudo-extension of a representation $(V,\theta)$ is an exact sequence of pseudo-modules and their morphisms \[\begin{tikzcd}
	0 & V & W & {\mathbf{K}} & 0
	\arrow[from=1-1, to=1-2]
	\arrow["f", from=1-2, to=1-3]
	\arrow["g", from=1-3, to=1-4]
	\arrow[from=1-4, to=1-5]
\end{tikzcd}\] where $\mathbf{K}$ has the structure of a trivial $T$-pseudo-module. Two one-dimensional right pseudo-extensions of $V$ are equivalent if the obvious diagram commutes. We denote by $\mathrm{Ext}(V) $ the set of equivalence classes of one-dimensional right pseudo-extensions of $(V,\theta)$.

Let $\varphi,\varphi'\in C^2(T,V)$ be two $2$-cocycles such that there exists $v\in V$ such that $\varphi(x,y)-\varphi'(x,y)=\theta(x,y)(v). $ We define two $T$-pseudo-module structures on $W=V\oplus \mathbf{K} $ by defining $$\theta_W(x,y)((u,\lambda))=(\lambda\varphi(x,y)+\theta(x,y)(u),0), \ \ \ \theta'(x,y)((u,\lambda))=(\lambda\varphi'(x,y)+\theta(x,y)(u),0). $$ We want to define a map $h:V\oplus \mathbf{K}\to V\oplus\mathbf{K}$ such that the diagram \[\begin{tikzcd}
	& 0 & V & V\oplus \mathbf{K} & {\mathbf{K}} & 0 \\
	& 0 & V & V\oplus\mathbf{K} & {\mathbf{K}} & 0 \\
	{}
	\arrow[from=1-2, to=1-3]
	\arrow["\iota", from=1-3, to=1-4]
	\arrow["{\mathrm{id}}"', from=1-3, to=2-3]
	\arrow["\pi", from=1-4, to=1-5]
	\arrow["h", from=1-4, to=2-4]
	\arrow[from=1-5, to=1-6]
	\arrow["{\mathrm{id}}", from=1-5, to=2-5]
	\arrow[from=2-2, to=2-3]
	\arrow["\iota"', from=2-3, to=2-4]
	\arrow["\pi"', from=2-4, to=2-5]
	\arrow[from=2-5, to=2-6]
\end{tikzcd}\] commutes. We define $h$ by $h(u,\lambda)=(u+\lambda v,\lambda) $. The diagram of vector spaces commutes. Moreover, $$h(\theta_W(x,y)(u,\lambda))=h(\lambda\varphi(x,y)+\theta(x,y)(u),0)= (\lambda\varphi(x,y)+\theta(x,y)(u),0) $$ and $$\theta_W'(x,y)(h(u,\lambda))=\theta_W'(x,y)(u+\lambda v,\lambda)=(\lambda\varphi'(x,y)+\theta(x,y)(u+\lambda v),0) .$$ So  \begin{align*}
    h(\theta_W(x,y)(u,\lambda))- \theta_W'(x,y)(h(u,\lambda))=& (\lambda(\varphi(x,y)- \varphi'(x,y)) + \theta(x,y)(u)-\theta(x,y)(u+\lambda v),0)\\ =& (0,0),  
\end{align*} according to  the definition of $v$. Hence, $h$ is a morphism of $T$-pseudo-modules and this construction leads to a map $H^2(T,V)\to \mathrm{Ext(V)} $ which is inverse to the map $\mathrm{Ext}(V)\to H^2(T,V) $. Thus, we have the following result.

\begin{prop}
    There is a one-to-one correspondence between $H^2(T,V)$ and $\mathrm{Ext}(V) $.
\end{prop} 

\begin{rmq}
    Yamaguti studied extensions of modules of Lie triple systems in \cite{Y68} and, for this purpose, introduced the notion of a weak representations, which is a pair $(V,\theta) $, consisting of a vector space $V$ and a bilinear map $\theta:T\times T\to \mathrm{End}(V)$ satisfying \eqref{rep2}. A new cochain complex adapted to weak representations was introduced. We will explore the adaptation of these notions to the restricted case in a future work.
\end{rmq}

\subsubsection{Extensions of Lie triple systems}

In this subsection, we investigate the link between extensions of Lie triple systems and the third cohomology group of the Yamaguti's cohomology theory.

\begin{defi}
    An ideal of a Lie triple system $T$ is a subspace $I\subset T$ such that $[I,T,T]\subset I $.
\end{defi}

Because of the Jacobi identity of a Lie triple system, this is equivalent to the conditions $[T,T,I]\subset I $ and $[T,I,T]\subset I $. An ideal of a Lie triple system is called abelian if in addition $[T,I,I]=0 $. Again, this implies that $[I,T,I]=0 $ and $[I,I,T]=0 $.

\begin{defi}
    Let $T,
    \mathcal{U},\mathcal{M}$ be Lie triple systems. We say that $T$ is an extension of $\mathcal{U}$ by $\mathcal{M}$ if there exists an exact sequence of Lie triple systems $$0\to \mathcal{M}\to T\to \mathcal{U}\to 0. $$ Two extensions \[\begin{tikzcd}
	0 & \mathcal{M} & T & {\mathcal{U}} & 0
	\arrow[from=1-1, to=1-2]
	\arrow["\iota", from=1-2, to=1-3]
	\arrow["\pi", from=1-3, to=1-4]
	\arrow[from=1-4, to=1-5]
\end{tikzcd}\] and \[\begin{tikzcd}
	0 & \mathcal{M} & T' & {\mathcal{U}} & 0
	\arrow[from=1-1, to=1-2]
	\arrow["\iota'", from=1-2, to=1-3]
	\arrow["\pi'", from=1-3, to=1-4]
	\arrow[from=1-4, to=1-5]
\end{tikzcd}\] of $\mathcal{U}$ by $\mathcal{M}$ are equivalent if there exists a map $F:T\to T'$ such that the diagram \[\begin{tikzcd}
	0 & \mathcal{M} & T & {\mathcal{U}} & 0 \\
	0 & \mathcal{M} & {T'} & {\mathcal{U}} & 0
	\arrow[from=1-1, to=1-2]
	\arrow["\iota", from=1-2, to=1-3]
	\arrow["{\mathrm{id}}", from=1-2, to=2-2]
	\arrow["\pi", from=1-3, to=1-4]
	\arrow["F", from=1-3, to=2-3]
	\arrow[from=1-4, to=1-5]
	\arrow["{\mathrm{id}}", from=1-4, to=2-4]
	\arrow[from=2-1, to=2-2]
	\arrow["{\iota'}", from=2-2, to=2-3]
	\arrow["{\pi'}", from=2-3, to=2-4]
	\arrow[from=2-4, to=2-5]
\end{tikzcd}\] is commutative. 
\end{defi}

An extension \[\begin{tikzcd}
	0 & \mathcal{M} & T & {\mathcal{U}} & 0
	\arrow[from=1-1, to=1-2]
	\arrow["\iota", from=1-2, to=1-3]
	\arrow["\pi", from=1-3, to=1-4]
	\arrow[from=1-4, to=1-5]
\end{tikzcd}\] is called abelian if $\iota(\mathcal{M}) $ is an abelian ideal of $T$. We denote by $\mathrm{Ext}(\mathcal{U},\mathcal{M}) $ the set of equivalence classes of abelian extensions of $\mathcal{U} $ by $\mathcal{M}$. We will only consider abelian extensions. Let $l:\mathcal{U}\to T$ be a section of this exact sequence i.e. a linear map such that $\pi\circ l=\mathrm{id}$. We can define a structure of $\mathcal{U}- $module on $\mathcal{M}$ by: \begin{equation}\label{eqth}
    \theta(u,v)(m)= [m,l(u),l(v)],
\end{equation} for all $m\in \mathcal{M}, u,v\in \mathcal{U}$.

\begin{lem}\cite{Z14}
    With the above notations, $(\mathcal{M},\theta) $ is a representation of $\mathcal{U} $. This representation does not depend on the choice of the section $l$. Moreover, two equivalent abelian extensions give the same representation. 
\end{lem}

Given a section $l$, we define the map $f:\mathcal{U}\times\mathcal{U}\times \mathcal{U}\to \mathcal{M} $, $$f(x_1,x_2,x_3)= [l(x_1),l(x_2),l(x_3)]-l([x_1,x_2,x_3]) $$ which takes values in $\mathcal{M}$ because $\pi$ is a morphism of Lie triple systems and $\pi\circ l=\mathrm{id} $. The map $f$ lies in $C^3(\mathcal{U},\mathcal{M}) $. 

\begin{lem}\cite{Y60}
    With the above notations, $f $ is a $3$-cocycle of the cohomology defined by the representation defined in \eqref{eqth} and its cohomology class does not depend on the choice of the section.
\end{lem}

 We show here that its cohomology class depends only on the equivalence class of the abelian extension. Let $T$ and $T'$ be two equivalent abelian extensions \[\begin{tikzcd}
	0 & \mathcal{M} & T & {\mathcal{U}} & 0 \\
	0 & \mathcal{M} & {T'} & {\mathcal{U}} & 0
	\arrow[from=1-1, to=1-2]
	\arrow["\iota", from=1-2, to=1-3]
	\arrow["{\mathrm{id}}", from=1-2, to=2-2]
	\arrow["\pi", from=1-3, to=1-4]
	\arrow["F", from=1-3, to=2-3]
	\arrow[from=1-4, to=1-5]
	\arrow["{\mathrm{id}}", from=1-4, to=2-4]
	\arrow[from=2-1, to=2-2]
	\arrow["{\iota'}", from=2-2, to=2-3]
	\arrow["{\pi'}", from=2-3, to=2-4]
	\arrow[from=2-4, to=2-5]
\end{tikzcd}\] Let $l$ and $l'$ be two sections such that $\pi\circ l=\mathrm{id} $ and $\pi'\circ l'=\mathrm{id} $. We recall that these two abelian extensions define the same representation $\theta:\mathcal{U}\times \mathcal{U}\to \mathcal{M}$. Let $$f(x_1,x_2,x_3)=[l(x_1),l(x_2),l(x_3)]-l([x_1,x_2,x_3]) $$ and $$f'(x_1,x_2,x_3)= [l'(x_1),l'(x_2),l'(x_3)]-l'([x_1,x_2,x_3]) $$ be the two cocycles defined as above. We have $\pi'\circ F\circ l= \pi\circ l=\mathrm{id}= \pi'\circ l' $. So $F\circ l$ is a section of $T' $ and $F\circ l -l' $ takes values in $\mathcal{M}$. Let $g= l' -F\circ l  $. \begin{align*}
    &f'(x_1,x_2,x_3)= [F\circ l(x_1)+g(x_1),F\circ l(x_2)+g(x_2),F\circ l(x_3)+g(x_3)]- F\circ l([x_1,x_2,x_3])\\ & - g([x_1,x_2,x_3]) \\ =& [g(x_1),F\circ l(x_2),F\circ l(x_3)]+[F\circ l(x_1),g(x_2),F\circ l(x_3)]+ [F\circ l(x_1),F\circ l(x_2),g(x_3)]\\ &-g([x_1,x_2,x_3]) + [F\circ l(x_1),F\circ l(x_2),F\circ l(x_3)] - F\circ l([x_1,x_2,x_3])\\ =& \theta(x_2,x_3)(g(x_1)) - \theta(x_1,x_3)(g(x_2)) + D(x_1,x_2)(g(x_3)) - g([x_1,x_2,x_3]) + F\circ f(x_1,x_2,x_3)\\ =& f(x_1,x_2,x_3)+ \delta^1 g(x_1,x_2,x_3),
\end{align*} since $F $ acts as the identity on $\mathcal{M}$. Thus, $f$ and $f'$ are in the same cohomology class.

We know that we can identify $T$ with $\mathcal{M}\times \mathcal{U} $ by $(m,x)\mapsto \iota(m)+l(x) $. In $T$, because $\iota(\mathcal{M}) $ is abelian in $T$, we have the following relation: \begin{align*}
    [m_1+&l(x_1),m_2+l(x_2),m_3+l(x_3)] \\ &=[m_1,l(x_2),l(x_3)] + [l(x_1),m_2,l(x_3)]+[l(x_1),l(x_2),m_3]+ f(x_1,x_2,x_3)+l([x_1,x_2,x_3]).
\end{align*} This leads us to the following definition of a Lie triple system product on $\mathcal{M}\times \mathcal{U} $, where $(\mathcal{M},\theta)$ is a representation of $\mathcal{U}$ and $f$ is a $3$-cocycle of the associated cohomology: \begin{align*}
    [(m_1,x_1)&,(m_2,x_2),(m_3,x_3)] \\ &= (\theta(x_2,x_3)(m_1)- \theta(x_1,x_3)(m_2)+D(x_1,x_2)(m_3)+f(x_1,x_2,x_3),[x_1,x_2,x_3])
\end{align*} and $
\mathcal{M}$ is abelian in $\mathcal{M}\times \mathcal{U} $. Hence, this yields an abelian extension of $\mathcal{U}$ by $\mathcal{M}$: \[\begin{tikzcd}
	0 & \mathcal{M} & {\mathcal{M}\times\mathcal{U}} & {\mathcal{U}} & 0.
	\arrow[from=1-1, to=1-2]
	\arrow[from=1-2, to=1-3]
	\arrow[from=1-3, to=1-4]
	\arrow[from=1-4, to=1-5]
\end{tikzcd}\]

\begin{rmq}
    When the cocycle is equal to zero, it is the semi-direct product of $\mathcal{U}$ by the representation $(\mathcal{M},\theta)$.
\end{rmq}

Zhang showed that two abelian extensions defined with two cocycles $f $ and $f'$ are equivalent if and only if $f$ and $f'$ are in the same cohomology class.

\begin{lem}\cite{Z14}
    Let $\mathcal{U} $ be a Lie triple system and $(\mathcal{M},\theta) $ be a representation of $\mathcal{U}$. Let $f,f'\in C^3(\mathcal{U},m) $ be two $3$-cocycles. Then, the two abelian extension given by $f$ and $f'$ are equivalent if and only if $f$ and $f'$ are in the same cohomology class.
\end{lem}

Thus, we have a map from $H^3(\mathcal{U},\mathcal{M}) $ to $\mathrm{Ext}(\mathcal{U},\mathcal{M}) $ which is injective. The map from $\mathrm{Ext}(\mathcal{U},\mathcal{M}) $ to $H^3(\mathcal{U},\mathcal{M}) $ is not well defined because the representation defined by the extension depends on its equivalence class. 

Now, let $\mathcal{U}$ be a Lie triple system and $(\mathcal{M},\theta)$ be a representation of $\mathcal{U}$. Let $$0\to \mathcal{M} \to T\to \mathcal{U}\to 0 $$ be an abelian extension inducing a representation on $\mathcal{M}$ equal to $\theta$, i.e. $$\theta'(u,v)(m):=[m,l(u),l(v)]=\theta(u,v)(m) ,$$ where $l:\mathcal{U}\to T$ is a section of the short exact sequence. Let $f\in C^3(\mathcal{U},\mathcal{M}) $ be the corresponding cocycle. Then, the extension $$0\to \mathcal{M}\to \mathcal{M}\times \mathcal{U}\to \mathcal{U}\to 0 ,$$ where the bracket on $\mathcal{M}\times \mathcal{U} $ is defined by  \begin{align*}
    [(m_1,x_1)&,(m_2,x_2),(m_3,x_3)] \\ &= (\theta(x_2,x_3)(m_1)- \theta(x_1,x_3)(m_2)+D(x_1,x_2)(m_3)+f(x_1,x_2,x_3),[x_1,x_2,x_3])
\end{align*} is equivalent to $$0\to \mathcal{M} \to T\to \mathcal{U}\to 0 $$ with the isomorphism of Lie triple systems $$\psi:\mathcal{M}\times \mathcal{U}\to T, (m,u)\mapsto m+l(u). $$

Then we have the following result.

\begin{thm}
    Let $\mathcal{U} $ be a Lie triple system and $(\mathcal{M},\theta) $ be a representation of $\mathcal{U}$. There is a bijection between equivalence classes of abelian extensions of $\mathcal{U}$ by $\mathcal{M}$ inducing the representation $\theta$ and $H^3(\mathcal{U},\mathcal{M}) $.
\end{thm}

Now, we restrict ourselves to the case of central extensions. 

\begin{defi}
    We say that an abelian extension \[\begin{tikzcd}
	0 & \mathcal{M} & T & {\mathcal{U}} & 0
	\arrow[from=1-1, to=1-2]
	\arrow["\iota", from=1-2, to=1-3]
	\arrow["\pi", from=1-3, to=1-4]
	\arrow[from=1-4, to=1-5]
\end{tikzcd}\] is central if $\iota(\mathcal{M}) $ is included in the center of $T$. We denote by $\mathrm{Ext_c}(\mathcal{U},\mathcal{M}) $ the set of equivalence classes of central extensions of $\mathcal{U}$ by $\mathcal{M}$. 
\end{defi}

In the case of a central extension, the representation \eqref{eqth} induced on $\mathcal{M}$ is trivial. Conversely, given a trivial $\mathcal{U}$-module $\mathcal{M}$ and $f\in C^3(\mathcal{U},\mathcal{M})$ a $3$-cocycle, the extension induced is central. Indeed, the bracket defined on $\mathcal{M}\times \mathcal{U}$ is now $$ [(m_1,x_1),(m_2,x_2),(m_3,x_3)]=(f(x_1,x_2,x_3),[x_1,x_2,x_3]) $$ and it is zero whenever one of the $x_i$ is zero. Thus, we have the following result.

\begin{thm}
    Let $\mathcal{U}$ be a Lie triple system and $\mathcal{M}$ a vector space, seen as a trivial $\mathcal{U}$-module. Then, there is a bijection between equivalence classes of central extension of $\mathcal{U}$ by $\mathcal{M}$ and $H^3(\mathcal{U},\mathcal{M})$.
\end{thm}

\section{Cohomology of restricted Lie triple systems}\label{s5}
In this section, we introduce a cohomology for restricted Lie triple systems.
\subsection{Restricted Lie triple systems}

Let $T$ be any Lie triple system and $n\geq 3$ any positive odd integer. For elements $x_1,x_2,\dots , x_n $ of $T$, define \begin{center}
    $(x_1,x_2,\dots,x_n)= [[\cdots [[x_1,x_2,x_3],x_4,x_5],\cdots],x_{n-1},x_n] \in T$.
\end{center} For $x,y\in T $, the expression $(x,Zx+y,Zx+y,\dots , Zx+y) $ with $Zx+y$ occurring $p-1$ times, is a polynomial in $Z$ with coefficients in $T$. For $i=1,\dots,p $, we define $s_i(x,y) \in T$ by requiring $is_i(x,y)$ to be the coefficient of $Z^{i-1} $ in $(x,Zx+y,Zx+y,\dots,Zx+y) $. For any injective embedding $L$ of $T$, we have \begin{center}
    $(x,Zx+y,Zx+y,\cdots,Zx+y)= [Zx+y,[Zx+y,[\cdots ,[Zx+y,x]\cdots]]] = (\ad(Zx+y))^{p-1}(x) $.
\end{center} In this case, $is_i(x,y) $ is the coefficient of $Z^{i-1} $ in the expression $((\ad(Zx+y))^{p-1}(x) $, which agrees with the use of the notation $s_i(x,y) $ in the definition of a restricted Lie algebra. 

\begin{defi}\label{D5.1}
    A Lie triple system $T$ over a field of characteristic $p>2$ is restricted if there is a $p$-map $(\cdot)^{[p]}:T\to T $  such that the following conditions are satisfied: \begin{enumerate}        
        \item $(\alpha x)^{[p]}= \alpha^p x^{[p]} $ for all $\alpha$ in $\mathbf{K}$ and all $x$ in $T$
        \item $(x+y)^{[p]}= x^{[p]}+y^{[p]}+ \sum\limits_{i=1}^{p-1}s_i(x,y) $ for all $x,y$ in $T$,
        \item  $[x,y^{[p]},z]=(x,\underbrace{y,\ldots,y}_p,z) $ for all $x,y,z$ in $T$,
    \end{enumerate} where $is_i(x,y)$ is the coefficient of $Z^{i-1} $ in $(x,Zx+y,Zx+y,\dots,Zx+y) $, and $Zx+y $ appears $p-1$ times.
\end{defi}

\begin{ex}
    Let $L$ be a restricted Lie algebra. Then $L$ equipped with the bracket $[x,y,z]=[[x,y],z] $ is a restricted Lie triple system. 
\end{ex}

We have an explicit formula for  $s_i$: $$is_i(x,y)=  \sum_{\substack{x_j\in \left\{x,y \right\}\\\#\left\{j, x_j=x\right\}=i-1}}(x,y,x_1,\ldots,x_{p-2}) ,$$ where $\#\left\{j, x_j=x\right\} $ refers to the number of $x_j$'s equal to $x$. Then we have \begin{align*}
    \sum\limits_{i=1}^{p-1}s_i(x,y)= &\sum\limits_{i=1}^{p-1}\frac{1}{i}\sum_{\substack{x_j\in \left\{x,y \right\}\\\#\left\{j, x_j=x\right\}=i-1}}(x,y,x_1,\ldots,x_{p-2}) \\ =&  \sum_{\substack{x_j\in \left\{x,y \right\}\\x_1=x,x_2=y}}\frac{1}{\#(x)}(x_1,\ldots,x_p).
\end{align*}

\begin{rmq} 
The $s_i$ coincide with those occurring for restricted Lie algebras. 
\end{rmq}

\begin{defi}
    Let $(T,(\cdot)^{[p]})$ and $ (T',(\cdot)^{[p]'})$ be two restricted Lie triple systems. A morphism of restricted Lie triple systems between $T$ and $T'$ is a morphism of Lie triple systems $f:T\to T' $ such that $f(x^{[p]})=f(x)^{[p]'} $ for all $x\in T$.
\end{defi}

\begin{defi}\cite{BM26}
    Let $T$ be a restricted Lie triple system and $(V,\theta)$ be a representation of $T$. A pair $(V,\theta)$ is called a restricted representation if, in addition, we have  \begin{center}
    $\theta(x^{[p]},y)=\theta(x,y) \circ \theta(x,x)^{\frac{p-1}{2}} ,$ 
    \end{center} 
    \begin{center} 
    $\theta(x,y^{[p]})= \theta(y,y)^{\frac{p-1}{2}}\circ\theta(x,y), $
    \end{center} for all $ x, y\in T  $.
\end{defi}

\begin{ex}
    The adjoint representation is a restricted representation.
\end{ex}

\subsection{The restricted complex}

In this subsection, we give an explicit construction of a complex for the cohomology of a restricted Lie triple system up to degree $5$. 

Let $T$ be a restricted Lie triple system and $(V,\theta) $ be a restricted representation of $T$. We define $C^{-1}_\mathrm{res}(T,V)= C^{-1}(T,V) $, $C^0_\mathrm{res}(T,V)=C^0(T,V) $, $C^1_\mathrm{res}(T,V)= C^1(T,V) $, $C^2_\mathrm{res}(T,V)=C^2(T,V) $, $\delta^{-1}_\mathrm{res}=\delta^{-1} $ and $\delta^0_\mathrm{res}=\delta^0$.

\begin{defi}
    Let $\varphi\in C^3(T,V)$ be a $3$-cochain of $T$. We say that $w:T\to V$ has the $\star $-property with respect to $\varphi$ if for all $\lambda\in \mathbf{K}$, $x,y\in T$, \begin{enumerate}
    \item $w(\lambda x)=\lambda^pw(x) $,
    \item \small{\begin{align*}
        &w(x+y)= w(x)+w(y) \\ &+  \sum_{\substack{x_j\in \left\{x,y \right\}\\x_1=x,x_2=y}}\frac{1}{\#(x)}\sum\limits_{k=0}^{\frac{p-1}{2}-1}\theta(x_{p-1},x_p)\cdots\theta(x_{p-2k+1},x_{p-2k+2}) \varphi((x_1,\ldots,x_{p-2k-2}),x_{p-2k-1},x_{p-2k}).
    \end{align*}}
\end{enumerate}
We define $C_{\mathrm{res}}^3(T,V) $ by $$C_{\mathrm{res}}^3(T,V)= \left\{(\varphi,w), \varphi\in C^3(T,V), w  \text{ has the  $\star$-property w.r.t}\, \varphi \right\}. $$
\end{defi}

Let $\psi\in C^1(T,V)$, we define $\tilde{\psi}:T\to V $ by $$\tilde{\psi}(x)=-\psi(x^{[p]})+\theta(x,x)^\frac{p-1}{2}(\psi(x)). $$ 

\begin{prop}\label{prop5.7}
    Let $\psi\in C^1_\mathrm{res}(T,V) $. Then $\tilde{\psi} $ has the $\star$-property with respect to $\delta^1\psi$.
\end{prop}

\begin{defi}
    Let $\left\{u_1,\ldots,u_r \right\} $, for $r\geq 2$, be a set of integers such that $u_1<u_2<\cdots<u_r $ and $I=\left\{i_1,\ldots,i_s \right\} $, $J= \left\{j_1,\ldots,j_t \right\} $. We write $I+J=\left\{u_1,\ldots,u_r \right\} $ if $I\cup J=\left\{u_1,\ldots,u_r\right\} $, $I\cap J=\varnothing$ and \begin{center}
        $u_1\leq i_1<i_2<\cdots<i_s\leq u_r $ and  $u_1\leq j_1<j_2<\cdots<j_t\leq u_r $.
    \end{center} Either $I$ or $J$ may be empty.
\end{defi} 

\begin{lem}\label{lemcalc}
Let $\theta$ be a representation of $T$. Then, for all $k\in \mathbf{N}^* $, $m>2k+1 $,
    \begin{align*}
        -&\theta((x_1,\ldots,x_{2k+1}),x_m)\\ &= \sum_{\substack{A+B=\left\{2,\ldots,2k+1,m\right\}\\ m\notin B}}(-1)^s\theta(x_{a_{s-1}},x_{a_s})\theta(x_{a_{s-3}},x_{a_{s-2}})\cdots \left\{\begin{matrix}
\theta (x_1,x_{a_1})\theta(x_{b_2},x_{b_1})\cdots\theta(x_{b_t},x_{b_{t-1}}) \\\theta (x_{b_1},x_1)\theta(x_{b_3},x_{b_2})\cdots\theta(x_{b_t},x_{b_{t-1}})
\end{matrix}\right. \end{align*} 
\begin{align*}
    &\theta(x_m,(x_1,\ldots,x_{2k+1}))\\ &= \sum_{\substack{A+B=\left\{2,\ldots,2k+1,m\right\}\\ m\notin A}}(-1)^s\theta(x_{a_{s-1}},x_{a_s})\theta(x_{a_{s-3}},x_{a_{s-2}})\cdots \left\{\begin{matrix}
\theta (x_1,x_{a_1})\theta(x_{b_2},x_{b_1})\cdots\theta(x_{b_{t}},x_{b_{t-1}}) \\\theta (x_{b_1},x_1)\theta(x_{b_3},x_{b_2})\cdots\theta(x_{b_t},x_{b_{t-1}})
\end{matrix}\right. \end{align*}
\end{lem}

\begin{proof}
    We proceed by induction on $k\geq 1$. The base case is clear. Assume that the two formulas hold for $k-1\geq 1$. Let $x_1,\ldots,x_{2k+1},x_m\in T$, for $m>2k+1$, \begin{align*}
        -\theta&((x_1,\ldots,x_{2k+1}),x_m)= -\theta([(x_1,\ldots,x_{2k-1}),x_{2k},x_{2k+1}],x_m) \\ =& \theta(x_{2k+1},x_m)\theta(x_{2k},(x_1,\ldots,x_{2k-1}))- \theta(x_{2k+1},x_m)\theta((x_1,\ldots,x_{2k-1}),x_{2k})\\ & + \theta(x_{2k},x_m)\theta(x_{2k+1},(x_1,\ldots,x_{2k-1})) - \theta((x_1,\ldots,x_{2k-1}),x_m)\theta(x_{2k+1},x_{2k}) \\ =& \theta(x_{2k+1},x_m)\sum_{\substack{A+B=\left\{2,\ldots,2k-1,2k\right\}\\ 2k\notin A}} C_{A,B}+ \theta(x_{2k+1},x_m)\sum_{\substack{A+B=\left\{2,\ldots,2k-1,2k\right\}\\ 2k\notin B}} C_{A,B} \\ & + \theta(x_{2k},x_m)\sum_{\substack{A+B=\left\{2,\ldots,2k-1,2k+1\right\}\\ 2k+1\notin A}} C_{A,B} + \sum_{\substack{A+B=\left\{2,\ldots,2k-1,m\right\}\\ m\notin B}} C_{A,B}\, \theta(x_{2k+1},x_{2k}) \\ =& \theta(x_{2k+1},x_m)\sum_{\substack{A+B=\left\{2,\ldots,2k\right\}}} C_{A,B} + \theta(x_{2k},x_m)\sum_{\substack{A+B=\left\{2,\ldots,2k-1,2k+1\right\}\\ 2k+1\notin A}}  C_{A,B}\\ &+ \sum_{\substack{A+B=\left\{2,\ldots,2k-1,m\right\}\\ m\notin B}} C_{A,B}\, \theta(x_{2k+1},x_{2k}) \\ =& \sum_{\substack{A+B=\left\{2,\ldots,2k+1,m\right\}\\ m\notin B}}(-1)^s\theta(x_{a_{s-1}},x_{a_s})\theta(x_{a_{s-3}},x_{a_{s-2}})\cdots \left\{\begin{matrix}
\theta (x_1,x_{a_1})\theta(x_{b_2},x_{b_1})\cdots\theta(x_{b_t},x_{b_{t-1}}) \\\theta (x_{b_1},x_1)\theta(x_{b_3},x_{b_2})\cdots\theta(x_{b_t},x_{b_{t-1}})
\end{matrix}\right. 
    \end{align*} where \begin{align*}
        C_{A,B}=(-1)^s\theta(x_{a_{s-1}},x_{a_s})\theta(x_{a_{s-3}},x_{a_{s-2}})\cdots \left\{\begin{matrix}
\theta (x_1,x_{a_1})\theta(x_{b_2},x_{b_1})\cdots\theta(x_{b_t},x_{b_{t-1}}) \\\theta (x_{b_1},x_1)\theta(x_{b_3},x_{b_2})\cdots\theta(x_{b_t},x_{b_{t-1}})
\end{matrix}\right.
    \end{align*} Similarly, \begin{align*}
        \theta(&x_m,(x_1,\ldots,x_{2k+1}))= \theta(x_m,[(x_1,\ldots,x_{2k-1}),x_{2k},x_{2k+1}]) \\ =& \theta(x_{2k},x_{2k+1})\theta(x_m,(x_1,\ldots,x_{2k-1}))-\theta((x_1,\ldots,x_{2k-1}),x_{2k+1})\theta(x_m,x_{2k}) \\ &+ \theta(x_{2k},(x_1,\ldots,x_{2k-1}))\theta(x_m,x_{2k+1}) - \theta((x_1,\ldots,x_{2k-1}),x_{2k})\theta(x_m,x_{2k+1}) \\ =& \theta(x_{2k},x_{2k+1})\sum_{\substack{A+B=\left\{2,\ldots,2k-1,m\right\}\\ m\notin A}}  C_{A,B}+ \sum_{\substack{A+B=\left\{2,\ldots,2k-1,2k+1\right\}\\ 2k+1\notin B}}  C_{A,B}\,\theta(x_m,x_{2k}) \\ &+ \sum_{\substack{A+B=\left\{2,\ldots,2k-1,2k\right\}\\ 2k\notin A}} C_{A,B}\,\theta(x_m,x_{2k+1})+ \sum_{\substack{A+B=\left\{2,\ldots,2k-1,2k\right\}\\ 2k\notin B}} C_{A,B}\,\theta(x_m,x_{2k+1}) \\ =& \theta(x_{2k},x_{2k+1})\sum_{\substack{A+B=\left\{2,\ldots,2k-1,m\right\}\\ m\notin A}} C_{A,B} + \sum_{\substack{A+B=\left\{2,\ldots,2k-1,2k+1\right\}\\ 2k+1\notin B}}  C_{A,B}\,\theta(x_m,x_{2k})\\ & + \sum_{\substack{A+B=\left\{2,\ldots,2k\right\}}} C_{A,B}\,\theta(x_m,x_{2k+1}) \\ =& \sum_{\substack{A+B=\left\{2,\ldots,2k+1,m\right\}\\ m\notin A}}(-1)^s\theta(x_{a_{s-1}},x_{a_s})\theta(x_{a_{s-3}},x_{a_{s-2}})\cdots \left\{\begin{matrix}
\theta (x_1,x_{a_1})\theta(x_{b_2},x_{b_1})\cdots\theta(x_{b_{t}},x_{b_{t-1}}) \\\theta (x_{b_1},x_1)\theta(x_{b_3},x_{b_2})\cdots\theta(x_{b_t},x_{b_{t-1}})
\end{matrix}\right.
    \end{align*}
\end{proof}

In particular, \begin{align*}
    &D((x_1,\cdots,x_{2k+1}),x_m)\\ &=\sum_{\substack{A+B=\left\{2,\ldots,2k+1,m\right\}}}(-1)^s\theta(x_{a_{s-1}},x_{a_s})\theta(x_{a_{s-3}},x_{a_{s-2}})\cdots \left\{\begin{matrix}
\theta (x_1,x_{a_1})\theta(x_{b_2},x_{b_1})\cdots\theta(x_{b_{t}},x_{b_{t-1}}) \\\theta (x_{b_1},x_1)\theta(x_{b_3},x_{b_2})\cdots\theta(x_{b_t},x_{b_{t-1}})
\end{matrix}\right. . \end{align*}

\begin{proof}[proof of the proposition]
The map $\tilde{\psi} $ clearly verify the $p$-homogeneity. Let $x,y\in T$, 

\begin{align*}
    \tilde{\psi}&(x+y)\\=& -\psi((x+y)^{[p]}) +\theta(x+y,x+y)^\frac{p-1}{2}(\psi(x)+\psi(y)) \\ =& -\psi(x^{[p]})- \psi(y^{[p]}) -  \sum_{\substack{x_j\in \left\{x,y \right\}\\x_1=x,x_2=y}}\frac{1}{\#(x)}\psi((x_1,\ldots,x_p)) + \theta(x+y,x+y)^\frac{p-1}{2}(\psi(x)+\psi(y)).
\end{align*}

Using the equation \begin{align*}
    \psi((x_1,\ldots,x_p))=&\theta(x_{p-1},x_p)\psi((x_1,\ldots,x_{p-2}))-\theta((x_1,\ldots,x_{p-2}),x_p)\psi(x_{p-1})\\ &+ D((x_1,\ldots,x_{p-2}),x_{p-1})\psi(x_p)-\delta^1\psi((x_1,\ldots,x_{p-2}),x_{p-1},x_p)
\end{align*} which follows from \eqref{lwo1}, we obtain the following \begin{align*}
    \psi&((x_1,\ldots,x_{p}))\\&=\theta(x_{p-1},x_p)\theta(x_{p-3},x_{p-2}) \cdots\theta(x_2,x_3)\psi(x_1) \\& -\sum\limits_{k=0}^{\frac{p-1}{2}-1}\theta(x_{p-1},x_p)\theta(x_{p-3},x_{p-2}) \cdots\theta(x_{p-2k+1},x_{p-2k+2})\theta((x_1,\ldots,x_{p-2k-2}),x_{p-2k})\psi(x_{p-2k-1}) \\ &+ \sum\limits_{k=0}^{\frac{p-1}{2}-1}\theta(x_{p-1},x_p)\theta(x_{p-3},x_{p-2})\cdots\theta(x_{p-2k+1},x_{p-2k+2})D((x_1,\ldots,x_{p-2k-2}),x_{p-2k-1})\psi(x_{p-2k}) \\ &- \sum\limits_{k=0}^{\frac{p-1}{2}-1}\theta(x_{p-1},x_p)\theta(x_{p-3},x_{p-2})\cdots\theta(x_{p-2k+1},x_{p-2k+2}) \delta\psi([x_1,\ldots,x_{p-2k-2}],x_{p-2k-1},x_{p-2k}) .
\end{align*} So it remains to show  \begin{align*} &\sum_{\substack{x_j\in \left\{x,y \right\}\\x_1=x,x_2=y}}\frac{1}{\#(x)}\theta(x_{p-1},x_p)\theta(x_{p-3},x_{p-2}) \cdots\theta(x_2,x_3)\psi(x_1) \\ 
     &- \sum_{\substack{x_j\in \left\{x,y \right\}\\x_1=x,x_2=y}}\frac{1}{\#(x)}\sum\limits_{k=0}^{\frac{p-1}{2}-1}\theta(x_{p-1},x_p)\cdots\theta(x_{p-2k+1},x_{p-2k+2})\theta((x_1,\ldots,x_{p-2k-2}),x_{p-2k})\psi(x_{p-2k-1}) \\ &+  \sum_{\substack{x_j\in \left\{x,y \right\}\\x_1=x,x_2=y}}\frac{1}{\#(x)}\sum\limits_{k=0}^{\frac{p-1}{2}-1}\theta(x_{p-1},x_p)\cdots\theta(x_{p-2k+1},x_{p-2k+2})D((x_1,\ldots,x_{p-2k-2}),x_{p-2k-1})\psi(x_{p-2k})\\ & -\theta(x+y,x+y)^{\frac{p-1}{2}}(\psi(x)+\psi(y)) = -\theta(x,x)^{\frac{p-1}{2}}\psi(x) -\theta(y,y)^{\frac{p-1}{2}}\psi(y).
\end{align*} We obviously have $$\theta(x+y,x+y)^{\frac{p-1}{2}}(\psi(x)+\psi(y))= \sum\limits_{x_j=x,y}\theta(x_1,x_2)\cdots\theta(x_{p-2},x_{p-1})\psi(x_p).  $$
Using \hyperref[lemcalc]{Lemma~\ref*{lemcalc}}, we obtain  \begin{align*}
    &\sum\limits_{k=0}^{\frac{p-1}{2}-1}\theta(x_{p-1},x_p)\cdots\theta(x_{p-2k+1},x_{p-2k+2})(D((x_1,\ldots,x_{p-2k-2}),x_{p-2k-1})\psi(x_{p-2k}) \\ & \,\,\,\,\,\,\,\,\,\,\,\,\,\,\,\,\,\,\,\,\,\,\,\,\,\,\,\,\,\,\,\,\,\,\,\,\,\,\,\,\,\,\,\,\,\,\,\,\,\,\,\,\,\,\,\,\,\,\,\,\,\,\,\,\,\,\,\,\,\,\,\,\,\,\,\,\,\,\,\,\,\,\,\,\,\,\,\,\,\,\,\,\,\,\,\,\,\,\,\,\,\,\,\,\,\,\,\,\,\,\,\,\,\,\,\,\,\,\,\,\,\,\,\,-\theta((x_1,\ldots,x_{p-2k-2}),x_{p-2k})\psi(x_{p-2k-1})) \\ =&\sum\limits_{k=0}^{\frac{p-1}{2}-1}\sum_{\substack{A+B=\left\{2,\ldots,p\right\}\\p-2k\in B\\ p-2k+1,\ldots,p\notin B}}(-1)^s\theta(x_{a_{s-1}},x_{a_s})\cdots \left\{\begin{matrix}
\theta (x_1,x_{a_1})\theta(x_{b_2},x_{b_1})\cdots\theta(x_{b_{t-1}},x_{b_{t-2}})\psi(x_{b_t}) \\\theta (x_{b_1},x_1)\theta(x_{b_3},x_{b_2})\cdots\theta(x_{b_{t-1}},x_{b_{t-2}})\psi(x_{b_t})
\end{matrix}\right. \\ &+ \sum\limits_{k=0}^{\frac{p-1}{2}-1}\sum_{\substack{A+B=\left\{2,\ldots,p\right\}\\p-2k\notin B\\p-2k-1\in B \\ p-2k+1,\ldots,p\notin B}}(-1)^s\theta(x_{a_{s-1}},x_{a_s})\cdots \left\{\begin{matrix}
\theta (x_1,x_{a_1})\theta(x_{b_2},x_{b_1})\cdots\theta(x_{b_{t-1}},x_{b_{t-2}})\psi(x_{b_t}) \\\theta (x_{b_1},x_1)\theta(x_{b_3},x_{b_2})\cdots\theta(x_{b_{t-1}},x_{b_{t-2}})\psi(x_{b_t})
\end{matrix}\right.\\ =&  \sum_{\substack{A+B=\left\{2,\ldots,p\right\}\\B\neq \varnothing}}(-1)^s\theta(x_{a_{s-1}},x_{a_s})\cdots \left\{\begin{matrix}
\theta (x_1,x_{a_1})\theta(x_{b_2},x_{b_1})\cdots\theta(x_{b_{t-1}},x_{b_{t-2}})\psi(x_{b_t}) \\\theta (x_{b_1},x_1)\theta(x_{b_3},x_{b_2})\cdots\theta(x_{b_{t-1}},x_{b_{t-2}})\psi(x_{b_t})
\end{matrix}\right.
\end{align*} and then \begin{align*} &\sum_{\substack{x_j\in \left\{x,y \right\}\\x_1=x,x_2=y}}\frac{1}{\#(x)}\theta(x_{p-1},x_p)\theta(x_{p-3},x_{p-2}) \cdots\theta(x_2,x_3)\psi(x_1) \\ 
     &- \sum_{\substack{x_j\in \left\{x,y \right\}\\x_1=x,x_2=y}}\frac{1}{\#(x)}\sum\limits_{k=0}^{\frac{p-1}{2}-1}\theta(x_{p-1},x_p)\cdots\theta(x_{p-2k+1},x_{p-2k+2})\theta((x_1,\ldots,x_{p-2k-2}),x_{p-2k})\psi(x_{p-2k-1}) \\ &+  \sum_{\substack{x_j\in \left\{x,y \right\}\\x_1=x,x_2=y}}\frac{1}{\#(x)}\sum\limits_{k=0}^{\frac{p-1}{2}-1}\theta(x_{p-1},x_p)\cdots\theta(x_{p-2k+1},x_{p-2k+2})D((x_1,\ldots,x_{p-2k-2}),x_{p-2k-1})\psi(x_{p-2k})\\ & -\theta(x+y,x+y)^{\frac{p-1}{2}}(\psi(x)+\psi(y))\\ &=\sum\limits_{\substack{x_j\in \left\{x,y \right\}\\x_1=x,x_2=y}}\frac{1}{\#(x)} \sum_{\substack{A+B=\left\{2,\ldots,p\right\}}}(-1)^s\theta(x_{a_{s-1}},x_{a_s})\theta(x_{a_{s-3}},x_{a_{s-2}})\\ & \,\,\,\,\,\,\,\,\,\,\,\, \,\,\,\,\,\,\,\,\,\,\,\, \,\,\,\,\,\,\,\,\,\,\,\, \,\,\,\,\,\,\,\,\,\,\,\, \,\,\,\,\,\,\,\,\,\,\,\, \,\,\,\,\,\,\,\,\,\,\,\, \,\,\,\,\,\,\,\,\,\,\,\, \,\,\,\,\,\,\,\,\,\,\,\, \,\,\,\,\,\,\,\,\,\,\,\, \,\,\,\,\,\,\,\,\,\,\,\,\cdots \left\{\begin{matrix}
\theta (x_1,x_{a_1})\theta(x_{b_2},x_{b_1})\cdots\theta(x_{b_{t-1}},x_{b_{t-2}})\psi(x_{b_t}) \\\theta (x_{b_1},x_1)\theta(x_{b_3},x_{b_2})\cdots\theta(x_{b_{t-1}},x_{b_{t-2}})\psi(x_{b_t})
\end{matrix}\right. \\ & -  \sum\limits_{x_j=x,y}\theta(x_1,x_2)\cdots\theta(x_{p-2},x_{p-1})\psi(x_p).
\end{align*} Fix $(x_1,\ldots,x_p)\in \left\{x,y\right\}^p $ and count how many times $\theta(x_1,x_2)\cdots\theta(x_{p-2},x_{p-1})\psi(x_p)$ occurs in \begin{equation}\label{eqd}\begin{aligned}
&\sum\limits_{\substack{x_j\in \left\{x,y \right\}\\x_1=x,x_2=y}}\frac{1}{\#(x)} \\ &\sum_{\substack{A+B=\left\{2,\ldots,p\right\}}}(-1)^s\theta(x_{a_{s-1}},x_{a_s})\theta(x_{a_{s-3}},x_{a_{s-2}})\cdots \left\{\begin{matrix}
\theta (x_1,x_{a_1})\theta(x_{b_2},x_{b_1})\cdots\theta(x_{b_{t-1}},x_{b_{t-2}})\psi(x_{b_t}) \\\theta (x_{b_1},x_1)\theta(x_{b_3},x_{b_2})\cdots\theta(x_{b_{t-1}},x_{b_{t-2}})\psi(x_{b_t})
\end{matrix}\right..
\end{aligned}   
\end{equation}
 Assume that $\#(x)\neq 0,p $. Then there is at least one occurrence of $x$ and at least one occurrence of $y$ among the $x_i$, say $x_s=x $, with $s$ odd $1\leq s\leq p$ or $x_{s+2}=x$ with $s$ even $0\leq s\leq p-3$. Then, after relabeling the $x_i$, if necessary, in \eqref{eqd}, we have $$\binom{p-1}{s}\equiv (-1)^s[p] $$ choices for the $a_1,\ldots,a_s$. The $(-1)^s$ cancel. This is for each $x$ among the $x_i$ and then $\theta(x_1,x_2)\cdots\theta(x_{p-2},x_{p-1})\psi(x_p)$ occurs $\#(x) $ times and the factor $\frac{1}{\#(x)} $ cancel. Finally, we see that the terms \begin{align*}
     &  \sum_{\substack{x_j\in \left\{x,y \right\}\\x_1=x,x_2=y}}\frac{1}{\#(x)} \\ &\sum_{\substack{A+B=\left\{2,\ldots,p\right\}}}(-1)^s\theta(x_{a_{s-1}},x_{a_s})\theta(x_{a_{s-3}},x_{a_{s-2}})\cdots \left\{\begin{matrix}
\theta (x_1,x_{a_1})\theta(x_{b_2},x_{b_1})\cdots\theta(x_{b_{t-1}},x_{b_{t-2}})\psi(x_{b_t}) \\\theta (x_{b_1},x_1)\theta(x_{b_3},x_{b_2})\cdots\theta(x_{b_{t-1}},x_{b_{t-2}})\psi(x_{b_t})
\end{matrix}\right. \\ & -  \sum\limits_{x_j=x,y}\theta(x_1,x_2)\cdots\theta(x_{p-2},x_{p-1})\psi(x_p)
 \end{align*} cancel in pairs when $\#(x)\neq 0,p $. The remaining terms are $-\theta(x,x)^{\frac{p-1}{2}}\psi(x) -\theta(y,y)^\frac{p-1}{2}\psi(y) $ and we have $\tilde{\psi} $ has the $*$-property with respect to $\delta^1\psi$.
\end{proof}
 
Thus, the map $$\delta_\mathrm{res}^1: C^1_\mathrm{res}(T,V)\to C^3_\mathrm{res}(T,V), \psi\mapsto (\delta^1\psi,\tilde{\psi}) $$ is well defined.

 \begin{thm}
 With the above notation, $\delta_\mathrm{res}^1\circ\delta^{-1}_\mathrm{res}=0 $ and $H_\mathrm{res}^1(T,V)= \mathrm{Ker}(\delta_\mathrm{res}^1)/\mathrm{Im}(\delta^{-1}_\mathrm{res}) $ is well defined. Moreover, $H_\mathrm{res}^1(T,V)$ injects in $H^1(T,V)$.
 \end{thm}

\begin{proof} The same proof as in Evans's thesis works here. We already know that $\delta^1\circ \delta^{-1}=0 $. So it remains to prove that for all $x\in T,v\in V$, $\widetilde{\theta(x,\cdot)(v)}=0 $. However, $$\widetilde{\theta(x,\cdot)(v)}(t)=- \theta(x,t^{[p]})(v)+ \theta(t,t)^{\frac{p-1}{2}}\theta(x,t)(v)=0, $$ since $\theta$ is a restricted representation of $T$. Thus, the composition vanishes. Let $\alpha\in H_\mathrm{res}^1(T,V) $ be a restricted cohomology class. Let $\psi\in C^1(T,V) $ be its representative. We have $(\delta^1(\psi),\tilde{\psi})=(0,0) $, so $\psi$ represents an ordinary cohomology class $i(\alpha)\in H^1(T,V) $. If $\psi'\in C^1(T,V)$ represents $\alpha$, then $\psi-\psi'\in \mathrm{Im}(\delta^{-1}_\mathrm{res})=\mathrm{Im}(\delta^{-1}) $ and $\psi' $ represent the same ordinary class $i(\psi') $ as $\psi$. Thus, the map $i:H_\mathrm{res}^1(T,V)\to H^1(T,V) $ is well defined and linear. This map is injective. If $i(\alpha)=0 $, then $\psi$ is in the image of $\delta^{-1}=\delta_\mathrm{res}^{-1} $ and $\alpha=0$.
\end{proof}

We now provide a definition of $C^4_\mathrm{res}(T,V) $. 

\begin{defi}
    Let $\varphi\in C^4(T,V) $ and  $\alpha:T\times T\to V$. We say that $\alpha $ has the $\star'$-property with respect to $\varphi $, if for all $\lambda\in \mathbf{K}$ and $x,y,y_1,y_2\in T$, we have  \begin{enumerate}
    \item $\alpha(\cdot,y):T\to V$ is linear,
    \item $\alpha( x,\lambda y)= \lambda^p\alpha(x,y) $,
    \item {\small\begin{align*}
       & \alpha(x,y_1+y_2)= \alpha(x,y_1)+\alpha(x,y_2)\\ &  +  \sum_{\substack{y_j\in \left\{y_1,y_2 \right\}\\y_1=y_1,y_2=y_2}}\frac{1}{\#(y_1)}\sum\limits_{k=0}^{\frac{p-1}{2}-1}\theta(y_{p-1},y_p)\cdots\theta(y_{p-2k+1},y_{p-2k+2}) \varphi(x,[y_1,\ldots,y_{p-2k-2}],y_{p-2k-1},y_{p-2k}).
    \end{align*}}
\end{enumerate} We define $$C^4_\mathrm{res}(T,V)= \left\{(\psi,\alpha), \psi\in C^4(T,V),\text{ $\alpha$   has the  $\star'$-property w.r.t $\psi$ }\right\}. $$
\end{defi}

Let $\psi\in C^2(T,V) $, we define $\overline{\psi}:T\times T:\to V $ by \begin{equation}\label{psibar}
    \overline{\psi}(x,y)= -\psi(x,y^{[p]})+ \theta(y,y)^{\frac{p-1}{2}}\psi(x,y).
\end{equation} Then we have the following result. 

\begin{prop}\label{propstar'}
    Let $\psi\in C^2(T,V) $. Then $\overline{\psi} $ has the $\star'$-property with respect to $\delta^2\psi$.
\end{prop}

\begin{proof}
For all $y\in T$, the map $\overline{\psi}(\cdot,y)$ is linear and the $p$-homogeneity is verified. Let $x,y_1,y_2\in T$,
\begin{align*}
    \overline{\psi}(x,y_1+y_2)= &-\psi(x,(y_1+y_2)^{[p]})+ \theta(y_1+y_2,y_1+y_2)^{\frac{p-1}{2}}(\psi(x,y_1+y_2)) \\ =&-\psi(x,y_1^{[p]})-\psi(x,y_2^{[p]}) - \sum_{\substack{y_j\in \left\{y_1,y_2 \right\}\\y_1=y_1,y_2=y_2}}\frac{1}{\#(y_1)}\psi(x,(y_1,\ldots,y_p)) \\ & \,\,\, + \theta(y_1+y_2,y_1+y_2)^\frac{p-1}{2}(\psi(x,y_1+y_2)).
\end{align*}
    We use the equation \begin{align*}
        \psi(x,(y_1,\ldots,y_p))=&\theta(y_{p-1},y_p)\psi(x,(y_1,\ldots,y_{p-2}))-\theta((y_1,\ldots,y_{p-2}),y_p)\psi(x,y_{p-1})\\ &+ D((y_1,\ldots,y_{p-2}),y_{p-1})\psi(x,y_p)-\delta^2\psi(x,(y_1,\ldots,y_{p-2}),y_{p-1},y_p),
    \end{align*} that comes from \eqref{lwo2} and  obtain \begin{align*}
         &\psi(x,(y_1,\ldots,y_{p}))\\ &=\theta(y_{p-1},y_p)\theta(y_{p-3},y_{p-2}) \cdots\theta(y_2,y_3)\psi(x,y_1) \\& -\sum\limits_{k=0}^{\frac{p-1}{2}-1}\theta(y_{p-1},y_p)\theta(y_{p-3},y_{p-2}) \cdots\theta(y_{p-2k+1},y_{p-2k+2})\theta((y_1,\ldots,y_{p-2k-2}),y_{p-2k})\psi(x,y_{p-2k-1}) \\ &+ \sum\limits_{k=0}^{\frac{p-1}{2}-1}\theta(y_{p-1},y_p)\theta(y_{p-3},y_{p-2})\cdots\theta(y_{p-2k+1},y_{p-2k+2})D((y_1,\ldots,y_{p-2k-2}),y_{p-2k-1})\psi(x,y_{p-2k}) \\ &- \sum\limits_{k=0}^{\frac{p-1}{2}-1}\theta(y_{p-1},y_p)\theta(y_{p-3},y_{p-2})\cdots\theta(y_{p-2k+1},y_{p-2k+2}) \delta^2\psi(x,[y_1,\ldots,y_{p-2k-2}],y_{p-2k-1},y_{p-2k}). 
    \end{align*} Then we use \hyperref[lemcalc]{Lemma~\ref*{lemcalc}} and count the terms in the same way as for the $\star$-property in the \hyperref[prop5.7]{Proposition~\ref*{prop5.7}}. Hence, the result follows. 
\end{proof}

The \hyperref[propstar']{Proposition~\ref*{propstar'}} allows us to define a coboundary map $$\delta^2_\mathrm{res}:C^2(T,V)\to C^4_\mathrm{res}(T,V), \,\,\,\,\,\, \delta^2_\mathrm{res}(\psi)=(\delta^2\psi, \overline{\psi}). $$ We have the following result.

\begin{thm}
With the above notations, $\delta^2_\mathrm{res}\circ\delta^0_\mathrm{res}=0 $ and $H^2_\mathrm{res}(T,V)= \mathrm{Ker}(\delta^2_\mathrm{res})/\mathrm{Im}(\delta^0_\mathrm{res}) $ is well defined. Moreover, $H^2_\mathrm{res}(T,V) $ injects in $H^2(T,V)$.
\end{thm}

\begin{proof}
    Let $v\in C^0(T,V) $. We already know that $\delta^2\delta^0v=0 $. So we must show that $\overline{\delta^0v}=0 $. Let $x,y\in T$. \begin{align*}
        \overline{\delta^0v}(x,y)=& -\delta^0 v(x,y^{[p]})+ \theta(y,y)^\frac{p-1}{2}\delta^0v(x,y) \\ =& -\theta(x,y^{[p]})(v)+ \theta(y,y)^\frac{p-1}{2}\theta(x,y)(v)\\ =& 0,
    \end{align*} since $\theta$ is a restricted representation. Thus, the composition vanishes. Let $\alpha\in H_\mathrm{res}^2(T,V) $ be a restricted cohomology class. Let $\psi\in C^2(T,V) $ be its representative. We have $(\delta^2(\psi),\overline{\psi})=(0,0) $, so $\psi$ represents an ordinary cohomology class $i(\alpha)\in H^2(T,V) $. If $\psi'\in C^2(T,V)$ represents $\alpha$, then $\delta^{2}(\psi-\psi')=0 $ and $\psi' $ represent the same ordinary class $i(\psi') $ as $\psi$. Thus, the map $i:H_\mathrm{res}^2(T,V)\to H^2(T,V) $ is well defined and linear. This map is injective. If $i(\alpha)=0 $, then $\psi$ is in the image of $\delta^{0}=\delta_\mathrm{res}^{0} $ and $\alpha=0$.
\end{proof}

We now provide a definition of $C^5_\mathrm{res}(T,V)$. We set \[C^5_\mathrm{res}(T,V)= \left\{\begin{aligned}
    &(\psi,\beta), \psi\in C^5(T,V), \beta: T\times T\times T\to V,\beta(\cdot,y,\cdot)\, \text{is bilinear}\, \forall y\in T\, \\ &\text{and  } \, \beta(x,\lambda y,z)=\lambda^p\beta(x,y,z) \forall x,y,z\in T,\lambda\in \mathbf{K}
\end{aligned}
    \right\}.\] 

\begin{rmq}
The above definition of $ C_\mathrm{res}^5(T,V) $ is only provisional. An element of $C_\mathrm{res}^5 (T,V)$ should more naturally be a pair $(\varphi,\beta) $, where $\varphi\in C^5(T,V)$ and $\beta:T\times T\times T\to V $ satisfies an appropriate $\star\star $-property with respect to $\delta^3\varphi$, analogous to that introduced in Evans's thesis. We have not yet determined such a $\star\star $-property. Since our considerations are restricted to $H^3_\mathrm{res}(T,V)$, the present definition is sufficient for our purposes, as the coboundary operator below is well defined.      
\end{rmq}

Let $(\varphi,w)\in C^3_\mathrm{res}(T,V) $ be a $3$-cochain. This pair induces a map $\beta:T\times T\times T\to V $ defined by \begin{align*}
    \beta(x,y,z)=& \varphi(x,y^{[p]},z)- \sum_{i+j=\frac{p-3}{2}}\theta(y,z)\theta(y,y)^i\varphi((x,\underbrace{y,\ldots,y}_{2j}),y,y) - \varphi((x,\underbrace{y,\ldots,y}_{p-1}),y,z)\\&- \theta(x,z)(w(y)).
\end{align*} We define the coboundary operator $$\delta_\mathrm{res}^3:C_\mathrm{res}^3(T,V)\to C_\mathrm{res}^5(T,V), (\varphi,w)\mapsto (\delta^3\varphi,\beta). $$ Then, we have the following result.

\begin{thm}
    With the above notations, $\delta_\mathrm{res}^3\circ \delta_\mathrm{res}^1=0 $ and $H_\mathrm{res}^3(T,V)= \mathrm{Ker}(\delta_\mathrm{res}^3)/\mathrm{Im}(\delta_\mathrm{res}^1) $ is well defined. 
\end{thm}

\begin{proof}
    Let $\varphi\in C_\mathrm{res}^1(T,V). $ Then $\delta_\mathrm{res}^1(\varphi)=(\delta^1\varphi,\tilde{\varphi}) $ and $\delta_\mathrm{res}^3(\delta^1\varphi,\tilde{\varphi})= (\delta^3\circ\delta^1(\varphi),\beta) $. It is already known that $\delta^3\circ\delta^1(\varphi)=0 $. Thus, it remains to show that $\beta=0 $. The $\beta$ induced by $(\delta^1\varphi,\tilde{\varphi})$ is defined by \begin{equation}\label{eqbet}\begin{aligned}
        \beta(x,y,z)=& \delta^1\varphi(x,y^{[p]},z)- \sum_{i+j=\frac{p-3}{2}}\theta(y,z)\theta(y,y)^i\delta^1\varphi((x,y,\ldots,y),y,y) - \delta^1\varphi((x,y,\ldots,y),y,z) \\ &- \theta(x,z)(\theta(y,y)^{\frac{p-1}{2}}(\varphi(y)) - \varphi(y^{[p]})) \end{aligned}
    \end{equation} for $x,y,z\in T$. Using Equation $$\delta \varphi(x,y,z)= \theta(y,z)(\varphi(x))- \theta(x,z)(\varphi(y))+D(x,y)(\varphi(z))- \varphi([x,y,z]), $$ we rewrite \eqref{eqbet},  \begin{align}
        \beta&(x,y,z)\notag \\&= \theta(y^{[p]},z)\varphi(x)- \uwave{\theta(x,z)\varphi(y^{[p]})}+ D(x,y^{[p]})\varphi(z)-\uline{\varphi([x,y^{[p]},z])}\notag \\ &- \sum\limits_{i+j=\frac{p-3}{2}}\theta(y,z)\theta(y,y)^{i+1}\varphi((x,\underbrace{y,\ldots,y}_{2j}))\label{eqdem1} \\ &+ \sum\limits_{i+j=\frac{p-3}{2}}\theta(y,z)\theta(y,y)^i\theta((x,\underbrace{y,\ldots,y}_{2j}),y)\varphi(y) - \sum\limits_{i+j=\frac{p-3}{2}}\theta(y,z)\theta(y,y)^iD((x,\underbrace{y,\ldots,y}_{2j}),y)\varphi(y)\notag \\ &+ \sum\limits_{i+j=\frac{p-3}{2}}\theta(y,z)\theta(y,y)^i\varphi((x,\underbrace{y,\ldots,y}_{2(j+1)}))\label{eqdem2} \\ &- \theta(y,z)\varphi((x,\underbrace{y,\ldots,y}_{p-1}))+ \theta((x,\underbrace{y,\ldots,y}_{p-1}),z)\varphi(y)-D((x,\underbrace{y,\ldots,y}_{p-1}),y)\varphi(z)+ \uline{\varphi((x,\underbrace{y,\ldots,y}_{p},z))}\notag \\ &- \theta(x,z)\theta(y,y)^{\frac{p-1}{2}}\varphi(y) + \uwave{\theta(x,z)\varphi(y^{[p]})}\notag.
    \end{align} The underlined terms cancel. Moreover, most of the terms in \eqref{eqdem1} and \eqref{eqdem2} cancel and the remaining terms are $\theta(y,z)\varphi((x,\underbrace{y,\ldots,y}_{p-1}))$ and $- \theta(y,z)\theta(y,y)^{\frac{p-1}{2}}\varphi(x)= -\theta(y^{[p]},z)\varphi(x) $. Thus, \begin{align}
        \beta&(x,y,z) \notag \\ &= D(x,y^{[p]})\varphi(z)+ \sum\limits_{i+j=\frac{p-3}{2}}\theta(y,z)\theta(y,y)^i\theta((x,\underbrace{y,\ldots,y}_{2j}),y)\varphi(y) \notag \\ &- \sum\limits_{i+j=\frac{p-3}{2}}\theta(y,z)\theta(y,y)^iD((x,\underbrace{y,\ldots,y}_{2j}),y)\varphi(y) + \theta((x,\underbrace{y,\ldots,y}_{p-1}),z)\varphi(y) \notag\\ &-D((x,\underbrace{y,\ldots,y}_{p-1}),y)\varphi(z) -\theta(x,z)\theta(y,y)^\frac{p-1}{2}\varphi(y) \notag.
    \end{align} Using \hyperref[p5.14]{Proposition~\ref*{p5.14}}, \begin{align*}
        D((x,&\underbrace{y,\ldots,y}_{p-1}),y)\varphi(z)\\ =& \theta(y,(x,\underbrace{y,\ldots,y}_{p-1}))\varphi(z)-\theta((x,\underbrace{y,\ldots,y}_{p-1}),y)\varphi(z)\\ =& \sum\limits_{k=0}^\frac{p-1}{2}\tbinom{p-1}{2k}\theta(y,y)^{\frac{p-1}{2}-k}\theta(y,x)\theta(y,y)^{k}\varphi(z) - \sum\limits_{k=0}^\frac{p-3}{2}\tbinom{p-1}{2k+1}\theta(y,y)^k\theta(x,y)\theta(y,y)^{\frac{p-1}{2}-k}\varphi(z)\\ & -\sum\limits_{k=0}^\frac{p-1}{2}\tbinom{p-1}{2k}\theta(y,y)^k\theta(x,y)\theta(y,y)^{\frac{p-1}{2}-k}\varphi(z)+\sum\limits_{k=0}^\frac{p-3}{2}\tbinom{p-1}{2k+1}\theta(y,y)^{\frac{p-1}{2}-k}\theta(y,x)\theta(y,y)^k\varphi(z)\\ =& -\sum\limits_{k=0}^\frac{p-3}{2}\tbinom{p}{2k+1}\theta(y,y)^k\theta(x,y)\theta(y,y)^{\frac{p-1}{2}-k}\varphi(z) -\theta(y,y)^\frac{p-1}{2}\theta(x,y)\varphi(z) \\ &+ \sum\limits_{k=0}^\frac{p-3}{2}\tbinom{p}{2k+1}\theta(y,y)^{\frac{p-1}{2}-k}\theta(y,x)\theta(y,y)^k\varphi(z)+ \theta(y,x)\theta(y,y)^\frac{p-1}{2}\varphi(z) \\ =& D(x,y^{[p]})\varphi(z).
    \end{align*} Thus, \begin{align}
        \beta&(x,y,z) \notag \\ &=  \sum\limits_{i+j=\frac{p-3}{2}}\theta(y,z)\theta(y,y)^i\theta((x,\underbrace{y,\ldots,y}_{2j}),y)\varphi(y) \notag \\ &- \sum\limits_{i+j=\frac{p-3}{2}}\theta(y,z)\theta(y,y)^iD((x,\underbrace{y,\ldots,y}_{2j}),y)\varphi(y) + \theta((x,\underbrace{y,\ldots,y}_{p-1}),z)\varphi(y) \notag\\ & -\theta(x,z)\theta(y,y)^\frac{p-1}{2}\varphi(y) \notag.
    \end{align} Using again \hyperref[p5.14]{Proposition~\ref*{p5.14}} \begin{align}
        \sum\limits_{i+j=\frac{p-3}{2}}&\theta(y,z)\theta(y,y)^i\theta((x,\underbrace{y,\ldots,y}_{2j}),y)\varphi(y) \notag - \sum\limits_{i+j=\frac{p-3}{2}}\theta(y,z)\theta(y,y)^iD((x,\underbrace{y,\ldots,y}_{2j}),y)\varphi(y) \notag\\ =& 2\sum\limits_{i+j=\frac{p-3}{2}}\theta(y,z)\theta(y,y)^i\sum\limits_{k=0}^j\tbinom{2j}{2k}\theta(y,y)^k\theta(x,y)\theta(y,y)^{j-k}\varphi(y)\label{eqdem3} \\& + \sum\limits_{i+j=\frac{p-3}{2}}\theta(y,z)\theta(y,y)^i\sum\limits_{k=0}^{j-1}\tbinom{2j}{2k+1}\theta(y,y)^k\theta(x,y)\theta(y,y)^{j-k}\varphi(y)\label{eqdem4} \\ &-2 \sum\limits_{i+j=\frac{p-3}{2}}\theta(y,z)\theta(y,y)^i\sum\limits_{k=0}^{j-1}\tbinom{2j}{2k+1}\theta(y,y)^{j-k}\theta(x,y)\theta(y,y)^{k}\varphi(y)\label{eqdem5} \\ & -\sum\limits_{i+j=\frac{p-3}{2}}\theta(y,z)\theta(y,y)^i\sum\limits_{k=0}^j\tbinom{2j}{2k}\theta(y,y)^{j-k}\theta(x,y)\theta(y,y)^{k}\varphi(y)\label{eqdem6}
    \end{align} We rearrange the terms in \eqref{eqdem3}, \begin{align}
        \sum\limits_{i+j=\frac{p-3}{2}}&\theta(y,z)\theta(y,y)^i\sum\limits_{k=0}^j\tbinom{2j}{2k}\theta(y,y)^k\theta(x,y)\theta(y,y)^{j-k}\varphi(y)\notag\\ =& \theta(y,z)\theta(x,y)\theta(y,y)^\frac{p-3}{2}\varphi(y)\notag\\ &+ \sum\limits_{k=1}^\frac{p-3}{2}\tbinom{p-3}{2k}\theta(y,z)\theta(y,y)^k\theta(x,y)\theta(y,y)^{\frac{p-3}{2}-k}\varphi(y) \label{eqdem7}\\ &+ \sum\limits_{u=1}^\frac{p-3}{2}\sum\limits_{s=0}^{\frac{p-3}{2}-u}\tbinom{p-3-2u}{2s}\theta(y,z)\theta(y,y)^{u+s}\theta(x,y)\theta(y,y)^{\frac{p-3}{2}-u-s}\varphi(y)\label{eqdem8}
    \end{align} and we count the occurrences of the term $\theta(y,z)\theta(y,y)^{i-1}\theta(x,y)\theta(y,y)^{\frac{p-1}{2}-i}\varphi(y) $ in \eqref{eqdem7} and \eqref{eqdem8} for $2\leq i\leq \frac{p-1}{2} $. Let $2\leq i\leq \frac{p-1}{2}$, there is one term in \eqref{eqdem7} for $k=i-1$ with coefficient $\tbinom{p-3}{2i-2} $. In \eqref{eqdem8}, for each $1\leq u\leq i-1$, there is one term for $s=i-1-u$ with coefficient $\tbinom{p-3-2u}{2i-2-2u} $. Thus, the coefficient is \begin{align*}
        \sum\limits_{u=0}^{i-1}\tbinom{p-3-2u}{2i-2u-2}= \sum\limits_{k=0}^{i-1}\tbinom{p-1-2i+2k}{2k}.
    \end{align*} Then, \begin{align*}
         \sum\limits_{i+j=\frac{p-3}{2}}&\theta(y,z)\theta(y,y)^i\sum\limits_{k=0}^j\tbinom{2j}{2k}\theta(y,y)^k\theta(x,y)\theta(y,y)^{j-k}\varphi(y)\\ =& \theta(y,z)\theta(x,y)\theta(y,y)^\frac{p-3}{2}\varphi(y) + \sum\limits_{i=2}^\frac{p-1}{2}\sum\limits_{k=0}^{i-1}\tbinom{p-1-2i+2k}{2k}\theta(y,z)\theta(y,y)^{i-1}\theta(x,y)\theta(y,y)^{\frac{p-1}{2}-i}\varphi(y).
    \end{align*} Now, we rearrange the terms in \eqref{eqdem4}, \begin{align}
         \sum\limits_{i+j=\frac{p-3}{2}}&\theta(y,z)\theta(y,y)^i\sum\limits_{k=0}^{j-1}\tbinom{2j}{2k+1}\theta(y,y)^k\theta(x,y)\theta(y,y)^{j-k}\varphi(y)\notag \\ =& \tbinom{p-3}{1}\theta(y,z)\theta(x,y)\theta(y,y)^{\frac{p-3}{2}}\varphi(y) \notag \\ &+\sum\limits_{k=1}^{\frac{p-3}{2}-1}\tbinom{p-3}{2k+1}\theta(y,z)\theta(y,y)^k\theta(x,y)\theta(y,y)^{\frac{p-3}{2}-k}\varphi(y) \label{eqdem9}\\ &+ \sum\limits_{u=1}^{\frac{p-3}{2}-1}\sum\limits_{s=0}^{\frac{p-3}{2}-1-u}\tbinom{p-3-2u}{2s+1}\theta(y,z)\theta(y,y)^{u+s}\theta(x,y)\theta(y,y)^{\frac{p-3}{2}-u-s}\varphi(y)\label{eqdem10}
    \end{align} and we count the occurrences of the term $\theta(y,z)\theta(y,y)^{i-1}\theta(x,y)\theta(y,y)^{\frac{p-1}{2}-i} $ in \eqref{eqdem9} and \eqref{eqdem10} for $2\leq i\leq \frac{p-3}{2} $. As before, we find that the coefficient of $\theta(y,z)\theta(y,y)^{i-1}\theta(x,y)\theta(y,y)^{\frac{p-1}{2}-i} $ is \begin{align*}
        \sum\limits_{u=0}^{i-1}\tbinom{p-3-2u}{2i-1-2u}= \sum\limits_{k=0}^{i-1}\tbinom{p-2-2i+2k+1}{2k+1}.
    \end{align*} Then, \begin{align*}
        \sum\limits_{i+j=\frac{p-3}{2}}&\theta(y,z)\theta(y,y)^i\sum\limits_{k=0}^{j-1}\tbinom{2j}{2k+1}\theta(y,y)^k\theta(x,y)\theta(y,y)^{j-k}\varphi(y)\\ = &\tbinom{p-3}{1}\theta(y,z)\theta(x,y)\theta(y,y)^{\frac{p-3}{2}}\varphi(y) \\&+ \sum\limits_{i=2}^\frac{p-3}{2}\sum\limits_{k=0}^{i-1}\tbinom{p-2-2i+2k+1}{2k+1}\theta(y,z)\theta(y,y)^{i-1}\theta(x,y)\theta(y,y)^{\frac{p-1}{2}-i}\varphi(y),
    \end{align*} and finally, \begin{align*}
        2\sum\limits_{i+j=\frac{p-3}{2}}&\theta(y,z)\theta(y,y)^i\sum\limits_{k=0}^j\tbinom{2j}{2k}\theta(y,y)^k\theta(x,y)\theta(y,y)^{j-k}\varphi(y) \\& + \sum\limits_{i+j=\frac{p-3}{2}}\theta(y,z)\theta(y,y)^i\sum\limits_{k=0}^{j-1}\tbinom{2j}{2k+1}\theta(y,y)^k\theta(x,y)\theta(y,y)^{j-k}\varphi(y)\\ =& -\theta(y,z)\theta(x,y)\theta(y,y)^\frac{p-3}{2}\varphi(y)+ 2\sum\limits_{k=0}^{\frac{p-3}{2}}\tbinom{2k}{2k}\theta(y,z)\theta(y,y)^{\frac{p-3}{2}}\theta(x,y)\varphi(y) \\&+ \sum\limits_{i=2}^\frac{p-3}{2}\sum\limits_{k=0}^{i-1}\left(2\tbinom{p-1-2i+2k}{2k}+\tbinom{p-2-2i+2k+1}{2k+1}\right)\theta(y,z)\theta(y,y)^{i-1}\theta(x,y)\theta(y,y)^{\frac{p-1}{2}-i}\varphi(y).
    \end{align*} In order to compute $$\sum\limits_{k=0}^{i-1}\left(2\tbinom{p-1-2i+2k}{2k}+\tbinom{p-2-2i+2k+1}{2k+1}\right), $$ we use the following identity: for all $m\leq p-1, r\leq p-1-m $, $$\tbinom{p-1-m}{r}\equiv(-1)^r\tbinom{m+r}{r}[p]. $$ So $\tbinom{p-1-2i+2k}{2k}\equiv \tbinom{2i}{2k}[p] $ and $\tbinom{p-1-(2i+1-(2k+1))}{2k+1}\equiv-\tbinom{2i+1}{2k+1}[p] $. Thus, \begin{align*}
        \sum\limits_{k=0}^{i-1}\left(2\tbinom{p-1-2i+2k}{2k}+\tbinom{p-2-2i+2k+1}{2k+1}\right)\equiv &2\sum\limits_{k=0}^{i-1}\tbinom{2i}{2k}- \sum\limits_{k=0}^{i-1}\tbinom{2i+1}{2k+1}[p] \\ \equiv& 2(2^{2i-1}-1)-(2^{2i}-1)[p]\\ \equiv& -1[p]. 
    \end{align*} Thus, \begin{align*}
        2\sum\limits_{i+j=\frac{p-3}{2}}&\theta(y,z)\theta(y,y)^i\sum\limits_{k=0}^j\tbinom{2j}{2k}\theta(y,y)^k\theta(x,y)\theta(y,y)^{j-k}\varphi(y) \\& + \sum\limits_{i+j=\frac{p-3}{2}}\theta(y,z)\theta(y,y)^i\sum\limits_{k=0}^{j-1}\tbinom{2j}{2k+1}\theta(y,y)^k\theta(x,y)\theta(y,y)^{j-k}\varphi(y)\\ =& -\sum\limits_{i=1}^{\frac{p-1}{2}}\theta(y,z)\theta(y,y)^{i-1}\theta(x,y)\theta(y,y)^{\frac{p-1}{2}-i}.
    \end{align*}  Now, we rearrange the terms in \eqref{eqdem5}, \begin{align}
        \sum\limits_{i+j=\frac{p-3}{2}}&\theta(y,z)\theta(y,y)^i\sum\limits_{k=0}^{j-1}\tbinom{2j}{2k+1}\theta(y,y)^{j-k}\theta(x,y)\theta(y,y)^{k}\varphi(y) \notag\\ =& \sum\limits_{i=0}^{\frac{p-3}{2}-1}\sum\limits_{k=1}^{\frac{p-3}{2}-i}\tbinom{p-3-2i}{2k-1} \theta(y,z)\theta(y,y)^{i+k}\theta(y,x)\theta(y,y)^{\frac{p-3}{2}-i-k}\varphi(y)\notag\\ =& \sum\limits_{k=1}^\frac{p-3}{2}\tbinom{p-3}{2k-1}\theta(y,z)\theta(y,y)^k\theta(y,x)\theta(y,y)^{\frac{p-3}{2}-k}\varphi(y) \label{eqdem11}\\ &+ \sum\limits_{u=1}^{\frac{p-3}{2}-1}\sum\limits_{k=1}^{\frac{p-3}{2}-u}\tbinom{p-3-2u}{2k-1} \theta(y,z)\theta(y,y)^{u+k}\theta(y,x)\theta(y,y)^{\frac{p-3}{2}-u-k}\varphi(y),\label{eqdem12}
    \end{align} and count the occurrences of the term $\theta(y,z)\theta(y,y)^{i-1}\theta(y,x)\theta(y,y)^{\frac{p-1}{2}-i} $ in \eqref{eqdem11} and \eqref{eqdem12} for $3\leq i\leq \frac{p-1}{2}$. Let $3\leq i\leq \frac{p-1}{2}$, as before, the coefficient of $\theta(y,z)\theta(y,y)^{i-1}\theta(y,x)\theta(y,y)^{\frac{p-1}{2}-i} $ is \begin{align*}
        \sum\limits_{u=0}^{i-2}\tbinom{p-3-2u}{2i-3-2u}.
    \end{align*} Then, \begin{align*}
         \sum\limits_{i+j=\frac{p-3}{2}}&\theta(y,z)\theta(y,y)^i\sum\limits_{k=0}^{j-1}\tbinom{2j}{2k+1}\theta(y,y)^{j-k}\theta(x,y)\theta(y,y)^{k}\varphi(y)\\ =& \tbinom{p-3}{1}\theta(y,z)\theta(y,y)\theta(y,x)\theta(y,y)^{\frac{p-3}{2}-1}\varphi(y)\\ &+ \sum\limits_{i=3}^{\frac{p-1}{2}}\sum\limits_{u=0}^{i-2}\tbinom{p-3-2u}{2i-3-2u}\theta(y,z)\theta(y,y)^{i-1}\theta(y,x)\theta(y,y)^{\frac{p-1}{2}-i}\varphi(y)
    \end{align*} We rearrange the terms in \eqref{eqdem6}, and as before, we obtain \begin{align*}
        \sum\limits_{i+j=\frac{p-3}{2}}&\theta(y,z)\theta(y,y)^i\sum\limits_{k=0}^j\tbinom{2j}{2k}\theta(y,y)^{j-k}\theta(x,y)\theta(y,y)^{k}\varphi(y)\\ =& \theta(y,z)\theta(y,x)\theta(y,y)^\frac{p-3}{2}\varphi(y) \\ &+ \sum\limits_{i=2}^\frac{p-1}{2} \sum\limits_{u=0}^{i-1}\tbinom{p-3-2u}{2i-2-2u}\theta(y,z)\theta(y,y)^{i-1}\theta(y,x)\theta(y,y)^{\frac{p-1}{2}-i}\varphi(y).
    \end{align*} We obtain \begin{align*}
        -2 \sum\limits_{i+j=\frac{p-3}{2}}&\theta(y,z)\theta(y,y)^i\sum\limits_{k=0}^{j-1}\tbinom{2j}{2k+1}\theta(y,y)^{j-k}\theta(x,y)\theta(y,y)^{k}\varphi(y) \\ & -\sum\limits_{i+j=\frac{p-3}{2}}\theta(y,z)\theta(y,y)^i\sum\limits_{k=0}^j\tbinom{2j}{2k}\theta(y,y)^{j-k}\theta(x,y)\theta(y,y)^{k}\varphi(y)\\ = &- \sum\limits_{i=3}^{\frac{p-1}{2}}\left(2\sum\limits_{u=0}^{i-2}\tbinom{p-3-2u}{2i-3-2u}+\sum\limits_{u=0}^{i-1}\tbinom{p-3-2u}{2i-2-2u}\right)\theta(y,z)\theta(y,y)^{i-1}\theta(y,x)\theta(y,y)^{\frac{p-1}{2}-i}\varphi(y)\\ &-\theta(y,z)\theta(y,x)\theta(y,y)^\frac{p-3}{2}\varphi(y) \\&- \left(\tbinom{p-3}{2}+1+2\tbinom{p-3}{1}\right)\theta(y,z)\theta(y,y)\theta(y,x)\theta(y,y)^{\frac{p-3}{2}-1} .
    \end{align*} Similarly, \begin{align*}
        2\sum\limits_{u=0}^{i-2}\tbinom{p-3-2u}{2i-3-2u}+\sum\limits_{u=0}^{i-1}\tbinom{p-3-2u}{2i-2-2u}&\equiv -(2^{2i-1}-2)+2^{2i-1}-1[p]\\&\equiv 1[p]. 
    \end{align*} We finally obtain \begin{align*}
        -2 \sum\limits_{i+j=\frac{p-3}{2}}&\theta(y,z)\theta(y,y)^i\sum\limits_{k=0}^{j-1}\tbinom{2j}{2k+1}\theta(y,y)^{j-k}\theta(x,y)\theta(y,y)^{k}\varphi(y) \\ & -\sum\limits_{i+j=\frac{p-3}{2}}\theta(y,z)\theta(y,y)^i\sum\limits_{k=0}^j\tbinom{2j}{2k}\theta(y,y)^{j-k}\theta(x,y)\theta(y,y)^{k}\varphi(y)\\ =& -\sum\limits_{i=1}^{\frac{p-1}{2}}\theta(y,z)\theta(y,y)^{i-1}\theta(y,x)\theta(y,y)^{\frac{p-1}{2}-i}\varphi(y).
    \end{align*} Moreover, using \hyperref[p5.14]{Proposition~\ref*{p5.14}}, \begin{align*}
        \theta((x,\underbrace{y,\ldots,y}_{p-1}),z)\varphi(y)=& \sum\limits_{k=1}^{\frac{p-1}{2}}\tbinom{p-1}{2k}\theta(y,z)\theta(y,y)^{k-1}\theta(x,y)\theta(y,y)^{\frac{p-1}{2}-k}\varphi(y)\\&- \sum\limits_{k=0}^{\frac{p-3}{2}}\tbinom{p-1}{2k+1}\theta(y,z)\theta(y,y)^{\frac{p-3}{2}-k}\theta(y,x)\theta(y,y)^k\varphi(y)\\&+\theta(x,z)\theta(y,y)^\frac{p-1}{2}\varphi(y)\\=& \sum\limits_{k=1}^{\frac{p-1}{2}}\theta(y,z)\theta(y,y)^{k-1}\theta(x,y)\theta(y,y)^{\frac{p-1}{2}-k}\varphi(y)\\&+ \sum\limits_{k=0}^{\frac{p-3}{2}}\theta(y,z)\theta(y,y)^{\frac{p-3}{2}-k}\theta(y,x)\theta(y,y)^k\varphi(y)+\theta(x,z)\theta(y,y)^\frac{p-1}{2}\varphi(y).
    \end{align*} Finally, \begin{align*}
        \beta(x,y,z)= 0,
    \end{align*} which concludes the proof.    
\end{proof}

We have the following two commutative diagrams, \[\begin{tikzcd}
	{C^{-1}(T,V)} & {C^1(T,V)} & {C^3(T,V)} & {C^5(T,V)} \\
	{C^{-1}_\mathrm{res}(T,V)} & {C^{1}_\mathrm{res}(T,V)} & {C^{3}_\mathrm{res}(T,V)} & {C^{5}_\mathrm{res}(T,V)}
	\arrow["{\delta^{-1}} ", from=1-1, to=1-2]
	\arrow["{\mathrm{id}}"', from=1-1, to=2-1]
	\arrow["{\delta^1} " ,from=1-2, to=1-3]
	\arrow["{\mathrm{id}}"', from=1-2, to=2-2]
	\arrow["{\delta^3} " ,from=1-3, to=1-4]
	\arrow[from=1-3, to=2-3]
	\arrow[from=1-4, to=2-4]
	\arrow["{\delta_\mathrm{res}^{-1}} "' ,from=2-1, to=2-2]
	\arrow["{\delta_\mathrm{res}^{1}} "',from=2-2, to=2-3]
	\arrow["{\delta_\mathrm{res}^{3}} "' ,from=2-3, to=2-4]
\end{tikzcd}\] and \[\begin{tikzcd}
	{} & {C^0(T,V)} & {C^2(T,V)} & {C^4(T,V)} \\
	& {C^{0}_\mathrm{res}(T,V)} & {C^{2}_\mathrm{res}(T,V)} & {C^{4}_\mathrm{res}(T,V)}
	\arrow["{\delta^{0}} ",from=1-2, to=1-3]
	\arrow["{\mathrm{id}}"', from=1-2, to=2-2]
	\arrow["{\delta^{2}} ",from=1-3, to=1-4]
	\arrow["{\mathrm{id}}"', from=1-3, to=2-3]
	\arrow[from=1-4, to=2-4]
	\arrow["{\delta_\mathrm{res}^{-1}} "',from=2-2, to=2-3]
	\arrow["{\delta_\mathrm{res}^{-1}} "',from=2-3, to=2-4]
\end{tikzcd}\] These diagrams show that there are maps $H^k_\mathrm{res}(T,V)\to H^k(T,V) $ for $k\leq 3$ which are injective for $k=1,2 $. As in the Lie algebra setting, the map $H^3_\mathrm{res}(T,V)\to H^3(T,V) $ fails to be injective in general, and we can give a  necessary and sufficient condition of injectivity. Let $$H^1(T,V)\to \mathrm{Hom}_{p-\mathrm{sl}}(T,V^T), \overline{\varphi}\mapsto \tilde{\varphi}, $$ where $$V^T=\left\{ u\in V\mid \theta(x,y)(u)=0, \forall x,y\in T\right\}.  $$ This map is well-defined. Then, we can prove, similarly to \hyperref[propinj]{Proposition~\ref*{propinj}}, the following result.

\begin{prop}
    The map $H^3_\mathrm{res}(T,V)\to H^3(T,V) $ is injective if and only if the map $H^1(T,V)\to \mathrm{Hom}_{p-\mathrm{sl}}(T,V^T) $ is surjective.
\end{prop}

Let $(V,\theta)$ and $(W,\theta') $ be two restricted representations of a restricted Lie triple system $T$ and let $f:V\to W $ be an isomorphism of representations. As in the non restricted case, we have natural maps $C^n_\mathrm{res}(T,V)\to C^n_\mathrm{res}(T,W) $ for $-1\leq n\leq 5$. As $C^n_\mathrm{res}(T,V)=C^n(T,V) $ for $n\leq 2$, the maps are the same. For  other degrees, we have $$C^3_\mathrm{res}(T,V)\to C^3_\mathrm{res}(T,W), \,\, (\varphi,w)\mapsto (f\circ \varphi,f\circ w). $$ It remains to show that $f\circ w $ has the $\star$-property with respect to $f\circ \varphi$ and it is a direct verification. The map $C^4_\mathrm{res}(T,V)\to C^4_\mathrm{res}(T,W)$ is the same and it is immediate that $f\circ w$ has the $\star'$-property with respect to $f\circ \varphi$, whenever $w$ has the $\star$'-property with respect to $\varphi$. Then, we define $$C^5_\mathrm{res}(T,V)\to C^5_\mathrm{res}(T,W), \,\, (\varphi,\beta)\mapsto (f\circ \varphi,f\circ \beta). $$ All these maps are isomorphisms. Now, it remains to show that these maps commute with the coboundary operators. According to definitions of these coboundary maps it remains to see that $\widetilde{f\circ \varphi}= f\circ \tilde{\varphi} $ and $\overline{f\circ \psi}= f\circ \overline{\psi} $ for all $\varphi\in C^3(T,V),\psi\in C^4(T,V) $ and that the map $\beta'$ induced by $f\circ \varphi $ and $f\circ w $ is equal to $f\circ\beta$ where $\beta$ is induced by $ \varphi $ and $w $. Since $f$ is a morphism of representations, this is true. Thus, we have the following two commutative diagrams \[\begin{tikzcd}
	{C^{-1}_\mathrm{res}(T,V)} & {C^{1}_\mathrm{res}(T,V)} & {C^{3}_\mathrm{res}(T,V)} & {C^{5}_\mathrm{res}(T,V)} \\
	{C^{-1}_\mathrm{res}(T,W)} & {C^{1}_\mathrm{res}(T,W)} & {C^{3}_\mathrm{res}(T,W)} & {C^{5}_\mathrm{res}(T,W)}
	\arrow["{\delta^{-1}_\mathrm{res}}", from=1-1, to=1-2]
	\arrow["\simeq", from=1-1, to=2-1]
	\arrow["{\delta^{1}_\mathrm{res}}", from=1-2, to=1-3]
	\arrow["\simeq", from=1-2, to=2-2]
	\arrow["{\delta^{3}_\mathrm{res}}", from=1-3, to=1-4]
	\arrow["\simeq", from=1-3, to=2-3]
	\arrow["\simeq", from=1-4, to=2-4]
	\arrow["{\delta^{-1}_\mathrm{res}}"', from=2-1, to=2-2]
	\arrow["{\delta^{1}_\mathrm{res}}"', from=2-2, to=2-3]
	\arrow["{\delta^{3}_\mathrm{res}}"', from=2-3, to=2-4]
\end{tikzcd}\] and \[\begin{tikzcd}
	{C^{0}_\mathrm{res}(T,V)} & {C^{2}_\mathrm{res}(T,V)} & {C^{4}_\mathrm{res}(T,V)} \\
	{C^{0}_\mathrm{res}(T,W)} & {C^{2}_\mathrm{res}(T,W)} & {C^{4}_\mathrm{res}(T,W)}
	\arrow["{\delta^{0}_\mathrm{res}}", from=1-1, to=1-2]
	\arrow["\simeq", from=1-1, to=2-1]
	\arrow["{\delta^{2}_\mathrm{res}}", from=1-2, to=1-3]
	\arrow["\simeq", from=1-2, to=2-2]
	\arrow["\simeq", from=1-3, to=2-3]
	\arrow["{\delta^{0}_\mathrm{res}}"', from=2-1, to=2-2]
	\arrow["{\delta^{2}_\mathrm{res}}"', from=2-2, to=2-3]
\end{tikzcd}\] These diagrams show that there are isomorphisms $H^n_\mathrm{res}(T,V)\simeq H^n_\mathrm{res}(T,W) $ for $n\leq 3$.

\subsection{Connection between cohomology of  restricted Lie algebras and  cohomology of the induced restricted Lie triple systems}

In this subsection, we develop restricted analogs of the constructions introduced in \hyperref[subs4]{Subsection~\ref*{subs4}}. 

Let $(L,(\cdot)^{[p]}) $ be a restricted Lie algebra and let $\rho:L\to \mathrm{End}(M)$ be a restricted representation of $L$. Then $(L_\mathrm{trip},(\cdot)^{[p]}) $ is a restricted Lie triple system and the pair $(\theta,M) $, where $\theta(x,y)(m)=\rho(y)\rho(x)(m) $ for all $x,y\in L,m\in M$, is a restricted representation of $L_\mathrm{trip} $. Let $(\varphi,w)\in C^2_\mathrm{res}(L,M) $ be a restricted $2$-cochain. We define $$\phi_{2,\mathrm{res}}(\varphi,w)=(\phi_2\varphi,-w). $$ We first show that the map $-w $ has the $\star$-property with respect to $\phi_2\varphi$. Indeed, let $x,y\in L$ be elements in $L$, \begin{align*}
    -&w(x+y)\\=& -w(x)-w(y) \\ & + \sum_{\substack{x_j\in \left\{x,y \right\},\\x_1=x,x_2=y}}\frac{1}{\#(x)}\sum\limits_{k=0}^{p-2}(-1)^{k+1}\rho(x_p)\cdots \rho(x_{p-k+1})\varphi([x_1,\ldots,x_{p-k-1}],x_{p-k})\\ =& -w(x)-w(y) \\ &+ \sum_{\substack{x_j\in \left\{x,y \right\},\\x_1=x,x_2=y}}\frac{1}{\#(x)} \sum\limits_{k=0}^{\frac{p-3}{2}}\theta(x_{p-1},x_p)\cdots\theta(x_{p-2k+1},x_{p-2k+2}) \rho(x_{p-2k})\varphi((x_1,\ldots,x_{p-2k-2}),x_{p-2k-1})  \\& \,\,\,\,\,\,\,\,\,\,\,\,\,\,\,\,\,\,\,\,\,\,\,\,\,\,\,\,\,\,\,\,\,\,\,\,\,\,\,\,\,\,\,\,\,-\sum\limits_{k=0}^{\frac{p-
    3}{2}}\theta(x_{p-1},x_p)\cdots\theta(x_{p-2k+1},x_{p-2k+2})\varphi((x_1,\ldots,x_{p-2k-1}),x_{p-2k}) \\ =& -w(x) -w(y) \\&+ \sum_{\substack{x_j\in \left\{x,y \right\},\\x_1=x,x_2=y}}\frac{1}{\#(x)} \sum\limits_{k=0}^{\frac{p-3}{2}}\theta(x_{p-1},x_p)\cdots\theta(x_{p-2k+1},x_{p-2k+2}) \phi_2\varphi((x_1,\ldots,x_{p-2k-2}),x_{p-2k-1},x_{p-2k}).
\end{align*} Hence, for every $\varphi\in C^1_\mathrm{res}(L,M)$, \begin{equation}
    \phi_{2,\mathrm{res}}d^1_\mathrm{res}\varphi= \delta^1_\mathrm{res}\varphi.
\end{equation} Thus, the map $\phi_{2,\mathrm{res}} $ commutes with the coboundary operators $d^1_\mathrm{res},\delta^1_\mathrm{res} $. Moreover, let $(\varphi,\beta)\in C^3_\mathrm{res}(L,M) $ be a restricted $3$-cochain. We define $$\phi_{3,\mathrm{res}}(\varphi,\beta)=(\phi_3\varphi,\tilde{\beta}), $$ where $\tilde{\beta}(x,y,z)=\rho(z)\beta(x,y) $. Thus, for every $(\varphi,w)\in C^2_\mathrm{res}(L,M) $, \begin{equation}
    \phi_{3,\mathrm{res}}(d^2_\mathrm{res}(\varphi,w))=\delta^3_\mathrm{res}(\phi_{2,\mathrm{res}}(\varphi,w)).
\end{equation} Hence, we obtain the following commutative diagram \[\begin{tikzcd}
	{C^1_\mathrm{res}(L,M)} & {Z_\mathrm{res}^2(L,M)} & {C_\mathrm{res}^3(L,M)} \\
	{C^1_\mathrm{res}(L_\mathrm{trip},M)} & {Z_\mathrm{res}^3(L_\mathrm{trip},M)} & {C^{5'}_\mathrm{res}(L_\mathrm{trip},M)}
	\arrow["{d_\mathrm{res}^1}", from=1-1, to=1-2]
	\arrow["{\mathrm{id}}"', from=1-1, to=2-1]
	\arrow["{d_\mathrm{res}^2}", from=1-2, to=1-3]
	\arrow[from=1-2, to=2-2]
	\arrow[from=1-3, to=2-3]
	\arrow["{\delta_\mathrm{res}^1}"', from=2-1, to=2-2]
	\arrow["{\delta_\mathrm{res}^3}"', from=2-2, to=2-3]
\end{tikzcd}\] where \[C^{5'}_\mathrm{res}(L_\mathrm{trip},M)= \left\{\begin{aligned}
    &(\psi,\beta), \psi\in \mathrm{Hom}(L_\mathrm{trip}^{\otimes 5},M), \\&\beta: L_\mathrm{trip}\times L_\mathrm{trip}\times L_\mathrm{trip}\to M,\beta(\cdot,y,\cdot)\, \text{is bilinear}\, \forall y\in L_\mathrm{trip}\, \\ &\text{and  } \, \beta(x,\lambda y,z)=\lambda^p\beta(x,y,z) \forall x,y,z\in L_\mathrm{trip},\lambda\in \mathbf{K}
\end{aligned}
    \right\}.\]  Therefore, $\phi_{2,\mathrm{res}} $ induces a map $\overline{\phi_{2,\mathrm{res}}}: H_\mathrm{res}^2(L,M)\to H^3_\mathrm{res}(L_\mathrm{trip},M),\overline{(\varphi,w)}\mapsto\overline{\phi_{2,\mathrm{res}}{(\varphi,w)}} $.

Let $\overline{(\varphi,w)}\in \mathrm{Ker}(\overline{\phi_{2,\mathrm{res}}}) $. Then, there exists $f\in C^1(L,M)$ such that $$(\phi_2\varphi,-w)=\delta^1_\mathrm{res}f=(\delta^1f,(\tilde{f})_\mathrm{LTS}), $$ where $(\tilde{f})_\mathrm{LTS} $ denotes the map induced by $f$ in the definition of $\delta^1_\mathrm{res} $. Hence, $\phi_2\varphi=\delta^1f $ and $-w=(\tilde{f})_\mathrm{LTS} $. Since $(\tilde{f})_\mathrm{LTS}=-(\tilde{f})_\mathrm{Lie} $, $w=(\tilde{f})_\mathrm{Lie} ,$ where $(\tilde{f})_\mathrm{Lie}$ denotes the map induced by $f$ in the definition of $d^1_\mathrm{res}$. Thus, the second component of $(\varphi,w) $ is already the second component of a restricted $2$-coboundary of $L$. Set $g=\varphi-d^1f $. Since $\phi_2d^1f=\delta^1f $, we have $\phi_2g=0 $. Under the assumptions of \hyperref[propinj1]{Proposition~\ref*{propinj1}}, the cohomology class $\overline{g} $ belongs to $\mathrm{Im}(\iota^*)$. Assume also that $M^L=0 $, then $\overline{g}=0$, so there exists a map $u\in C^1(L,M) $ such that $\varphi-d^1f=d^1u $. Therefore, $(\varphi,w)= (d^1(f+u),\tilde{f}) $. In order to conclude that $(\varphi,w) $ is a restricted coboundary, it remains to prove that $\tilde{u}=0 $. Since $(\varphi,w) - d^1_\mathrm{res}f=(g,0) $, the pair $(g,0)$ is a restricted $2$-cocycle. On the other hand, $(g,\tilde{u})=d^1_\mathrm{res}(u) $ is a restricted $2$-coboundary and hence also a restricted $2$-cocycle. Therefore, the difference $(0,\tilde{u}) $ is a restricted $2$-cocycle. The second component of $d^2_\mathrm{res}(0,\tilde{u}) $ is given by $$(x,y)\mapsto\rho(x)\tilde{u}(y) $$ for all $x,y\in L $. Hence, $\rho(x)\tilde{u}(y)=0 $ for all $x,y\in L$ and therefore $\tilde{u} $ takes values in $M^L $. Since $M^L=0$, we obtain $\tilde{u}=0$. Consequently, $(\varphi,w)=d^1_\mathrm{res}(f+u)$. Hence, $\overline{(\varphi,w)}=0 $ in $H^2_\mathrm{res}(L,M)$.

\begin{prop}
     If $L=[L,L]+Z(L) $, $\mathrm{char}(\mathbf{K})> 3$ and $M^L=0$, then the map $H_\mathrm{res}^2(L,M)\to H_\mathrm{res}^3(L_\mathrm{trip},M)$ is injective.
\end{prop}

We summarize  our results in the following way: Assume that $H^2(L,M)\to H^3(L_\mathrm{trip},M) $ is injective and that $H^1(L,M)\to \mathrm{Hom}_{p-\mathrm{sl}}(L,M^L) $ is surjective. Let $(\varphi,w)\in Z^2_\mathrm{res}(L,M) $ such that $\overline{(\phi_2\varphi,-w)}=0 $ in $H^3_\mathrm{res}(L_\mathrm{trip},M)$. Then, there exists $h\in C^1(L,M) $ such that $\varphi=d^1h $. Hence, $(\varphi,w)-d^1_\mathrm{res}h=(0,w-\tilde{h}) $ and we can, as in the proof of \hyperref[propinj]{Proposition~\ref*{propinj}}, find $u\in Z^1(L,M) $ such that $\tilde{u}=w-\tilde{h} $. Thus, $\overline{(\varphi,w)}=0 $ in $H^2_\mathrm{res}(L,M)$.

\begin{thm}
    Assume that $H^2(L,M)\to H^3(L_\mathrm{trip},M) $ and $H^2_\mathrm{res}(L,M)\to H^2(L,M) $ are injective. Then, $H^2_\mathrm{res}(L,M)\to H^3_\mathrm{res}(L_\mathrm{trip},M) $ is injective.
\end{thm}

\begin{cor}
     If $M^{Z(L)}=0 $, then the map $H_\mathrm{res}^2(L,M)\to H_\mathrm{res}^3(L_\mathrm{trip},M) $ is injective.
\end{cor}

We can similarly define a map $H^1_\mathrm{res}(L,M)\to H^2_\mathrm{res}(L_\mathrm{trip},M)$. Since $C_\mathrm{res}^0(L,M)=C^0(L,M)=C_\mathrm{res}^0(L_\mathrm{trip},M)=C^0(L_\mathrm{trip},M) , C^1_\mathrm{res}(L,M)=C^1(L,M)$ and $C^2_\mathrm{res}(L_\mathrm{trip},M)=C^2(L_\mathrm{trip},M) $, it remains to show that $\phi_1\varphi $ is a restricted $2$-cocycle for every $\varphi\in C^1(L,M) $. Let $\varphi\in C^1(L,M) $. Then we only have to check that the second component of $\delta^2_\mathrm{res}(\phi_1\varphi) $ vanishes. It is equal to \begin{align*}
        -\phi_1\varphi(x,y^{[p]})+\theta(y,y)^\frac{p-1}{2}\phi_1\varphi(x,y)&= \rho(y^{[p]})\varphi(x) - \rho(y)^\frac{p-1}{2}\rho(y)\varphi(x)\\ &= 0.
 \end{align*} Thus, $\phi_1 $ induces a map $\overline{\phi_1}:H^1_\mathrm{res}(L,M)\to H^2_\mathrm{res}(L_\mathrm{trip},M) $. The same injectivity criterion as in the non-restricted case remains valid.

\begin{prop}
    The map $H_\mathrm{res}^1(L,M)\to H_\mathrm{res}^2(L_\mathrm{trip},M)$ is injective if and only if the map $H_\mathrm{res}^1(L,M^L)\to H_\mathrm{res}^1(L,M) $ is zero.
\end{prop}

\section{Algebraic interpretations}\label{s6}

In this section, we develop the analogs of the algebraic interpretations of the first cohomology groups of the Yamaguti cohomology discussed in Section \eqref{s4} in the restricted setting. 

\subsection{Restricted derivations}

\begin{defi}
    Let $T$ be a restricted Lie triple system. Let $D:T\to T$ be a derivation of $T$. We say that $D$ is a restricted derivation if, in addition, $$D(x^{[p]})=(D(x),x,\ldots,x) $$ for all $x\in T$. We denote by $\mathrm{Der}_\mathrm{res}(T) $ the space of all restricted derivations of $T$.
\end{defi}

\begin{lem}
    The inner derivations are restricted.
\end{lem}

\begin{proof}
    Let $x,y\in T$ and consider the inner derivation $\phi:z\mapsto [x,y,z] $. Then, $\phi(z^{[p]})= [x,y,z^{[p]}]=(x,y,z,\ldots,z) = (\phi(z),z,\ldots,z) $.
\end{proof}

We consider here the restricted cohomology with coefficient in $T$ with the adjoint representation $\theta(x,y)(z)= [z,x,y] $.  We have the following result.

\begin{prop}
  The first cohomology group $H_\mathrm{res}^1(T,T) $ is equal to the space $\mathrm{Der}_\mathrm{res}(T)/\mathrm{InnDer}(T) $.
\end{prop}

\begin{proof}
    The $1$-cocycles for the restricted cohomology with coefficient in $T$ are the restricted derivations. Indeed, let $\psi\in C^1(T,V)$ such that $\delta_\mathrm{res}^1(\psi)=0 $. Thus, $\delta^1(\psi)=0 $ and $\psi $ is a derivation. Moreover, $\tilde{\psi}=0 $ so $\psi(x^{[p]})= \theta(x,x)^\frac{p-1}{2}(\psi(x)) $, which means that $\psi$ is a restricted derivation. As in the non-restricted case, $\delta^{-1}_\mathrm{res}(C^{-1}_\mathrm{res}(T,V))= \mathrm{InnDer}(T) $. Thus, the result is proved.
\end{proof}

\subsection{Restricted extensions of modules}

\begin{defi}
    We say that a pseudo-module $(W,\theta)$ is restricted if $$\theta(x,y^{[p]})= \theta(y,y)^\frac{p-1}{2}\theta(x,y), $$ for all $x,y\in T$.
\end{defi}

A morphism between two restricted pseudo-modules is just a morphism between pseudo-modules.

\begin{defi}
A one-dimensional right pseudo-extension of a restricted module $V$ is an exact sequence of restricted pseudo-modules and their morphisms \[\begin{tikzcd}
	0 & V & W & {\mathbf{K}} & 0
	\arrow[from=1-1, to=1-2]
	\arrow["f", from=1-2, to=1-3]
	\arrow["g", from=1-3, to=1-4]
	\arrow[from=1-4, to=1-5]
\end{tikzcd}\] where $\mathbf{K}$ has the structure of a trivial $T$-module.
\end{defi} 

The definition of equivalence between one-dimensional pseudo-extensions of a restricted module is the same as in the non-restricted case. We will show that there is a one-to-one correspondence between equivalence classes of one-dimensional right pseudo-extension of a restricted module $(V,\theta)$ and its second restricted cohomology group $H^2_\mathrm{res}(T,V) $.

Let \[\begin{tikzcd}
	0 & V & W & {\mathbf{K}} & 0
	\arrow[from=1-1, to=1-2]
	\arrow["f", from=1-2, to=1-3]
	\arrow["g", from=1-3, to=1-4]
	\arrow[from=1-4, to=1-5]
\end{tikzcd}\] be a one-dimensional right pseudo-extension of a restricted module $(V,\theta)$, then we have seen that we can construct a cocycle $\varphi\in C^2(T,V) $ and that the cohomology class of $\varphi$ depends only on the equivalence class of the extension. It remains to see that $\varphi $ is a restricted $2$-cocycle, i.e. $\overline{\varphi}=0 $ and that its restricted cohomology class depends uniquely on the equivalence class of the extension. We recall that $\varphi$ is defined by $\varphi(x,y)=f^{-1}(\theta'(x,y)(u)) $, where $u\in W $ such that $g(u)=1$ and $(W,\theta') $ is a restricted representation of $T$. We have  \begin{align*}
   f( \overline{\varphi}(x,y))=&f( -\varphi(x,y^{[p]})+ \theta(y,y)^{\frac{p-1}{2}}\varphi(x,y))\\ =& -\theta'(x,y^{[p]})(u)+ f(\theta(y,y)^\frac{p-1}{2}(f^{-1}(\theta'(x,y)(u))))\\ =& -\theta'(y,y)^\frac{p-1}{2}\theta'(x,y)(u) +  \theta'(y,y)^\frac{p-1}{2}\theta'(x,y)(u)\\ =& 0.
\end{align*} So $\overline{\varphi}=0$ since $f$ is injective. Thus, $\varphi $ is a restricted $2$-cocycle. We have seen that, if we have two equivalent extensions, then the $2$-cocycles defined are in the same cohomology class, so it is the same for  restricted $2$-cocycles since $\delta^0_\mathrm{res}=\delta^0 $. So, we have constructed a map $\mathrm{Ext}_\mathrm{res}(V)\to H^2_\mathrm{res}(T,V) $. Now, we consider the inverse construction. 

Let $\varphi\in C^2_\mathrm{res}(T,V) $ be a restricted $2$-cocycle. We have already defined a pseudo-extension  \[\begin{tikzcd}
	& 0 & V & {V\oplus \mathbf{K}} & {\mathbf{K}} & 0 
	\arrow[from=1-2, to=1-3]
	\arrow["\iota", from=1-3, to=1-4]
	\arrow["\pi", from=1-4, to=1-5]
	\arrow[from=1-5, to=1-6]
\end{tikzcd}\] with the structure defined on $V\oplus \mathbf{K}$ by the semi-direct product and the representation $\theta':T\times T\to V\oplus \mathbf{K} $ by $\theta'(x,y)((u,\lambda))= (\lambda\varphi(x,y)+ \theta(x,y)(u),0) $. The map $\theta'$ is a restricted pseudo-representation. Let $x,y\in T,u\in V,\lambda\in \mathbf{K} $. On one hand, we have  \begin{align*}
    \theta'(x,y^{[p]})(u,\lambda)=& (\lambda\varphi(x,y^{[p]})+ \theta(x,y^{[p]})(u),0)\\ =& (\lambda\varphi(x,y^{[p]})+\theta(y,y)^\frac{p-1}{2}\theta(x,y)(u),0).
\end{align*} On the other hand, \begin{align*}
    \theta'(y,y)^\frac{p-1}{2}\theta'(x,y)(u,\lambda)=& \theta'(y,y)^\frac{p-1}{2}((\lambda\varphi(x,y)+ \theta(x,y)(u),0))\\ =&( \lambda\theta(y,y)^\frac{p-1}{2}\varphi(x,y)+\theta(y,y)^\frac{p-1}{2}\theta(x,y)(u),0).
\end{align*} Thus, $\theta'$ is a restricted pseudo-representation and  is a one-dimensional right pseudo-extension of the restricted module $V$. Let $\varphi,\varphi'\in C^2(T,V) $ be two restricted $2$-cocycles such that there exists $v\in V$ with $\varphi(x,y)-\varphi'(x,y)=\theta(x,y)(v) $. We have seen that in this case, the two one-dimensional right pseudo-extensions are equivalent. So, the two one-dimensional right pseudo-extension of the restricted module $V$ are equivalent. As the two constructions are inverse of each other, we have the following result.

\begin{prop}
    We have a one-to-one correspondence between $H^2_\mathrm{res}(T,V) $ and $\mathrm{Ext}_\mathrm{res}(V) $. 
\end{prop}

\subsection{Formal deformations of restricted Lie triple systems}

Formal deformations of restricted Lie algebras were introduced and studied in \cite{EM25}. In this subsection, we define formal deformations of a restricted Lie triple system $T$ and study their relationship with $H^3_\mathrm{res}(T,T) $. 

\begin{defi}
    Let $(T,[\cdot,\cdot,\cdot],(\cdot)^{[p]}) $ be a restricted Lie triple system. A formal deformation of $T$ is given by two maps \[
\begin{array}{c@{\qquad\text{and}\qquad}c}
\begin{array}{rcl}
[\cdot,\cdot,\cdot]_t : T\times T\times T & \longrightarrow & T[[t]] \\
(x,y,z) & \longmapsto & \sum\limits_{i\geq 0}t^i\varphi_i(x,y,z)
\end{array}
&
\begin{array}{rcl}
(\cdot)^{[p]}_t : T & \longrightarrow & T[[t]] \\
x & \longmapsto & \sum\limits_{i\geq 0}t^iw_i(x)
\end{array}
\end{array}
\] with $[\cdot,\cdot,\cdot]_0=[\cdot,\cdot,\cdot] $,$(\cdot)^{[p]}_0= (\cdot)^{[p]} $ and such that $(T[[t]],[\cdot,\cdot,\cdot]_t,(\cdot)^{[p]}_t) $ is a restricted Lie triple system.
\end{defi}

\begin{rmq}
    \begin{enumerate}
        \item The bracket $[\cdot,\cdot,\cdot]_t $ extends to $T[[t]] $ by $\mathbf{K}[[t]] $-linearity.
        \item The map  $(\cdot)^{[p]}_t $ extends to $T[[t]] $ by $p$-homogeneity and because of the identity \begin{align*}
        (x+y)^{[p]}_t= (x)^{[p]}_t+ (y)^{[p]}_t+ \sum_{\substack{x_j\in \left\{x,y \right\}\\x_1=x,x_2=y}}\frac{1}{\#(x)}(x_1,\ldots,x_p)_t,  
    \end{align*}
    \end{enumerate}
\end{rmq}

\begin{prop}
    Let $([\cdot,\cdot,\cdot]_t,(\cdot)^{[p]}_t) $ be a formal deformation of $(T,[\cdot,\cdot,\cdot],(\cdot)^{[p]}) $. Then, $(\varphi_1,w_1)\in C^3_\mathrm{res}(T,T) $, which means that $w_1 $ has the $\star$-property with respect to $\varphi_1 $.
\end{prop}

\begin{proof}
   It is well known that $\varphi\in C^3(T,T)$. Let $x,x_1,\ldots,x_p,y\in T,\lambda\in \mathbf{K} $.
Since $(\lambda x)^{[p]_t}= \lambda^px^{[p]_t} $, we see that $w_1(\lambda x)=\lambda^pw_1(x) $. The following computations are carried out modulo $t^2 $.
We have \begin{align*}
    &(x_1,\ldots,x_p)_t\\=& (x_1,\ldots,x_p)+ t\left(\sum\limits_{k=0}^\frac{p-3}{2}(\varphi_1([x_1,\ldots,x_{p-2k-2}],x_{p-2k-1},x_{p-2k}),x_{p-2k+1},\ldots,x_p)\right)\\ =& (x_1,\ldots,x_p) + t \left(\sum\limits_{k=0}^{\frac{p-3}{2}}\theta(x_{p-1},x_p)\cdots\theta(x_{p-2k+1},x_{p-2k+2}) \varphi_1([x_1,\ldots,x_{p-2k-2}],x_{p-2k-1},x_{p-2k}) \right),
\end{align*} where $\theta(x,y)(z)=[z,x,y] $ is the adjoint representation. The map $(\cdot)^{[p]}$ being a $p$-map, we have the identity \begin{align*}
    (x+y)^{[p]_t}=& x^{[p]_t}+y^{[p]_t}+  \sum_{\substack{x_j\in \left\{x,y \right\}\\x_1=x,x_2=y}}\frac{1}{\#(x)}(x_1,\ldots,x_p)_t. 
\end{align*} Comparing the constant term and the coefficient of $t$, we obtain \begin{align*}
    w_1&(x+y)\\ =&w_1(x)+w_1(y)\\ & + \sum_{\substack{x_j\in \left\{x,y \right\}\\x_1=x,x_2=y}}\frac{1}{\#(x)}\sum\limits_{k=0}^{\frac{p-3}{2}}\theta(x_{p-1},x_p)\cdots\theta(x_{p-2k+1},x_{p-2k+2}) \varphi_1([x_1,\ldots,x_{p-2k-2}],x_{p-2k-1},x_{p-2k}).
\end{align*} So, $w_1$ has the $\star$-property with respect to $\varphi_1$.
\end{proof}

\begin{prop}
    Let $([\cdot,\cdot,\cdot]_t,(\cdot)^{[p]}_t) $ be a formal deformation of $(T,[\cdot,\cdot,\cdot],(\cdot)^{[p]}) $. Then, $(\varphi_1,w_1)\in C^3_\mathrm{res}(T,T) $ is a restricted $3$-cocycle.
\end{prop}

\begin{proof}

We already know that $\delta^3\varphi_1=0 $, so it remains to show that the induced map $\beta$ vanishes. Let $x,y,z\in T$. The following computations are carried out modulo $t^2$. On  one hand, we have \begin{align*}
    (x,y,\ldots,y,z)_t=& (x,y,\ldots,y,z) +t  \sum_{i+j=\frac{p-3}{2}}\theta(y,z)\theta(y,y)^i\varphi_1((x,y,\ldots,y),y,y) \\ &+t \varphi_1((x,y,\ldots,y),y,z).
\end{align*} On the other hand, \begin{align*}
    [x,(y)_t^{[p]},z]_t= [x,y^{[p]},z]+ t[x,w_1(y),z] + t\varphi_1(x,y^{[p]},z).
\end{align*} Then, we obtain  \begin{align*}
    &\varphi_1(x,y^{[p]},z)-  \sum_{i+j=\frac{p-3}{2}}\theta(y,z)\theta(y,y)^i\varphi_1((x,y,\ldots,y),y,y) - \varphi_1((x,y,\ldots,y),y,z) + [x,w_1(y),z]\\ &=0
\end{align*} which means that $(\varphi_1,w_1) $ is a $3$-cocycle.
\end{proof}

Let $\psi_t:T[[t]]\to T[[t]] $ be a formal automorphism defined on $T$ by $$\psi_t(x)= \sum\limits_{i\geq 0}t^i\psi_i(x), $$ where $\psi_i:T\to T$ are linear maps, $\psi_0=\mathrm{id}$ and then $\psi_t$ is extended to $T[[t]]$ by $\mathbf{K}[[t]] $-linearity. 

\begin{defi}
    Let $T$ be a restricted Lie triple system. We say that two formal deformations $([\cdot,\cdot,\cdot]_t,(\cdot)^{[p]}_t)$ and $([\cdot,\cdot,\cdot]'_t,(\cdot)'^{[p]}_t) $ are equivalent if there exists a formal automorphism $\psi_t:T[[t]]\to T[[t]] $, which is a morphism of restricted Lie triple systems, i.e. $$\psi_t([x,y,z]_t)=[\psi_t(x),\psi_t(y),\psi_t(z)]'_t \,\,\,\,\,\, \text{and} \,\,\,\,\,\, \psi_t((x)^{[p]}_t)=(\psi_t(x))'^{[p]}_t, $$ for all $x,y,z\in T$.
\end{defi}

\begin{thm}
    Let $([\cdot,\cdot,\cdot]_t,(\cdot)^{[p]}_t)$ and $([\cdot,\cdot,\cdot]'_t,(\cdot)'^{[p]}_t $ be two equivalent formal deformations of \newline$(T,[\cdot,\cdot,\cdot],(\cdot)^{[p]}) $. Then, $(\varphi_1,w_1), (\varphi'_1,w'_1) $ are in the same cohomology class in $H^3_\mathrm{res}(T,T) $. More precisely, $(\varphi_1,w_1)- (\varphi'_1,w'_1)=\delta^1_\mathrm{res}(\psi_1) $.
\end{thm}

\begin{proof}
     Let $\psi_t:T[[t]] \to T[[t]] $ be a morphism of restricted Lie triple systems between $(T[[t]],[\cdot,\cdot,\cdot]_t,(\cdot)^{[p]}_t) $ and $(T[[t]],[\cdot,\cdot,\cdot]_t',(\cdot)'^{[p]}_t ) $. It is already known that it implies that $\varphi_1-\varphi_1'=\delta^1(\psi_1) $. Now, let us consider $w_1-w_1' $. The following computations are carried out modulo $t^2$, $$\psi_t(x^{[p]}+tw_1(x))= \psi_t(x)^{[p]}+tw_1'(\psi_t(x)). $$ So $$x^{[p]}+tw_1(x)+t\psi_1(x^{[p]})= (x+t\psi_1(x))^{[p]}+ tw_1'(x). $$ In \begin{align*}
        (x+t\psi_1(x))^{[p]}= x^{[p]}+ t^p\psi_1(x)^{[p]}+ \sum_{\substack{x_j\in \left\{x,t\psi_1(x) \right\}\\x_1=x,x_2=t\psi_1(x)}}\frac{1}{\#(x)}(x_1,\ldots,x_p),  
    \end{align*} the only term with coefficient $t$ in the sum is when $\#(x)=p-1 $. So, it is $\frac{1}{p-1}(x,\psi_1(x),x,\ldots,x)= (\psi_1(x),x,\ldots,x) $. Finally, we have $w_1(x)+\psi_1(x^{[p]})= (\psi_1(x),x,\ldots,x)+ w_1'(x) $, which means that $w_1-w_1'= \tilde{\psi_1} $. Thus, $(\varphi_1-\varphi_1',w_1-w_1') = (\delta^1(\psi_1),\tilde{\psi_1})=\delta_\mathrm{res}^1(\psi_1) $ and the proof is completed.
\end{proof}

\begin{defi}
    Let $([\cdot,\cdot,\cdot]_t,(\cdot)^{[p]}_t) $ be a formal deformation of a restricted Lie triple system $(T,[\cdot,\cdot,\cdot],(\cdot)^{[p]}) $. We say that this deformation is trivial if it is equivalent to the trivial deformation $([\cdot,\cdot,\cdot],(\cdot)^{[p]})$ that is, there exists a formal automorphism $\psi_t$ such that $$\psi_t([x,y,z])=[\psi_t(x),\psi_t(y),\psi_t(z)]_t \,\,\,\,\,\, \text{and} \,\,\,\,\,\, \psi_t((x)^{[p]})=(\psi_t(x))^{[p]}_t ,$$ for all $x,y,z\in T$.
\end{defi}

\begin{cor}
    Assume the formal deformation $([\cdot,\cdot,\cdot]_t,(\cdot)^{[p]}_t) $ is trivial. Then, $(\varphi_1,w_1) $ is a restricted coboundary.
\end{cor}

\begin{cor}
    Let $([\cdot,\cdot,\cdot]+t\varphi_1,(\cdot)^{[p]}+tw_1) $ be an infinitesimal deformation of a restricted Lie triple system $(T,[\cdot,\cdot,\cdot],(\cdot)^{[p]}) $. Then, it is a trivial deformation if and only if $(\varphi_1,w_1) $ is a restricted coboundary.
\end{cor}

\begin{proof}
    It remains to show that if $(\varphi_1,w_1) $ is a restricted coboundary, then the deformation is trivial. Assume that $(\varphi_1,w_1) $ is a restricted coboundary and let $\psi:T\to T \in C^1(T,T)$ such that $(\varphi_1,w_1)=(\delta^1\psi,\tilde{\psi}) $. Then $\psi_t= \mathrm{id}-t\psi:(T[[t]]/(t^2),[\cdot,\cdot,\cdot],(\cdot)^{[p]}) \to (T[[t]]/(t^2),[\cdot,\cdot,\cdot]+t\varphi_1, (\cdot)^{[p]}+tw_1)$ is a morphism of restricted Lie triple systems. Since $t^2=0$, this morphism is an automorphism. 
\end{proof}

Let $T,[\cdot,\cdot,\cdot],(\cdot)^{[p]}) $ be a restricted Lie triple system, let $n\geq 1$. We say that a deformation $([\cdot,\cdot,\cdot]_t,(\cdot)^{[p]}_t) $ of $T$ is of order $n$ if $(\varphi_i,w_i)=(0,0) $ for $i> n$ and if the deformation identities are satisfied modulo $t^{n+1}$. A $n$-th order deformation $([\cdot,\cdot,\cdot]_t,(\cdot)^{[p]}_t) $ is said to be extendable to order $n+1$ if there exists a pair $(\varphi_{n+1},w_{n+1})\in C^3_\mathrm{res}(T,T) $ such that $([\cdot,\cdot,\cdot]_t+ t^{n+1}\varphi_{n+1},(\cdot)^{[p]}_t+t^{n+1}w_{n+1} ) $ is an $(n+1)$-th order deformation. 

\begin{defi}
    Let $([\cdot,\cdot,\cdot]_t,(\cdot)^{[p]}_t) $ be a deformation of order $n\geq 1$. We define two maps by \begin{align*}
        \mathrm{Obs}_{n+1}^{(1)}(u,v,x,y,z)=& \sum\limits_{i=1}^{n}\varphi_i(\varphi_{n+1-i}(u,v,x),y,z)-\varphi_i(x,\varphi_{n+1-i}(u,v,y),z)\\&\,\,\,\,\,\,\,\,\,\,\,\,\,\, -\varphi_i(x,y,\varphi_{n+1-i}(u,v,z))-\varphi_i(u,v,\varphi_{n+1-i}(x,y,z))
    \end{align*} and \begin{align*}
        \mathrm{Obs}_{n+1}^{(2)}(x,y,z)=& \sum_{\substack{0\leq i_k\leq n \\ i_1+\cdots+i_{\frac{p+1}{2}}=n+1}}\varphi_{i_\frac{p+1}{2}}(\varphi_{i_{\frac{p-1}{2}}}(\cdots(\varphi_{i_2}(\varphi_{i_1}(x,y,y),y,y)\cdots),y,z)\\& \,\,\,\,\,\,\,\,\,\,\,\,\,- \sum\limits_{i=1}^n\varphi_i(x,w_{n+1-i}(y),z),
    \end{align*} for all $x,y,z,u,v\in T$.
\end{defi}

\begin{prop}\label{propObs}
    Let $([\cdot,\cdot,\cdot]_t,(\cdot)^{[p]}_t) $ be a $n$-order deformation of $T$. Let $(\varphi_{n+1},w_{n+1})\in C^3_\mathrm{res}(T,T) $ such that $([\cdot,\cdot,\cdot]_t+t^{n+1}\varphi_{n+1} ,(\cdot)^{[p]}_t+t^{n+1}w_{n+1}) $ is a deformation of order $n+1$. Then,  $$(\mathrm{Obs}_{n+1}^{(1)},\mathrm{Obs}_{n+1}^{(2)})=\delta^3_\mathrm{res}(\varphi_{n+1},w_{n+1}). $$ 
\end{prop}

\begin{proof}
    Let $[\cdot,\cdot,\cdot]'_t=\sum\limits_{0\leq i\leq n+1}t^i\varphi_i $, and $(\cdot)'^{[p]}_t= \sum\limits_{0\leq i\leq n+1}w_i $. Assume $(T[[t]],[\cdot,\cdot,\cdot]'_t,(\cdot)'^{[p]}_t) $ is a restricted Lie triple system, thus \begin{equation}\label{eqobs}
        [u,v,[x,y,z]'_t]'_t=[[u,v,x]'_t,y,z]'_t+ [x,[u,v,y]'_t,z]'_t + [x,y,[u,v,z]'_t]'_t.
    \end{equation}
    for all $x,y,z,u,v\in T$.  The coefficient of  $t^{n+1}$ in \eqref{eqobs} is \begin{align*}
        t^{n+1}\sum\limits_{i=0}^{n+1}&\varphi_i(u,v,\varphi_{n+1-i}(x,y,z))\\=& t^{n+1}\sum\limits_{i=0}^{n+1}\varphi_i(\varphi_{n+1-i}(u,v,x),y,z)+ t^{n+1}\sum\limits_{i=0}^{n+1}\varphi_i(x,\varphi_{n+1-i}(u,v,y),z)\\ &+t^{n+1}\sum\limits_{i=0}^{n+1}\varphi_i(x,y,\varphi_{n+1-i}(u,v,z)).
    \end{align*} This gives us \begin{align*}
        [u,v,\varphi_{n+1}&(x,y,z)]+\varphi_{n+1}(u,v,[x,y,z])+\sum\limits_{i=1}^{n}\varphi_i(u,v,\varphi_{n+1-i}(x,y,z))\\=& [\varphi_{n+1}(u,v,x),y,z]+ \varphi_{n+1}([u,v,x],y,z)+\sum\limits_{i=1}^{n}\varphi_i(\varphi_{n+1-i}(u,v,x),y,z)\\ &+ [x,\varphi_{n+1}(u,v,y),z]+ \varphi_{n+1}(x,[u,v,y],z) + \sum\limits_{i=1}^{n}\varphi_i(x,\varphi_{n+1-i}(u,v,y),z)\\ &+ [x,y,\varphi_{n+1}(u,v,z)]+\varphi_{n+1}(x,y,[u,v,z])+ \sum\limits_{i=1}^{n}\varphi_i(x,y,\varphi_{n+1-i}(u,v,z)),
    \end{align*} which means \begin{align*}
        \delta^3\varphi_{n+1}(u,v,x,y,z)=\mathrm{Obs}_{n+1}^{(1)}(u,v,x,y,z).
    \end{align*} Moreover, $(\cdot)'^{[p]}_t $ being a $p$-map means that \begin{equation}\label{eqobs2}
        [x,(y)'^{[p]}_t,z]'_t= (x,\underbrace{y,\ldots,y}_p,z)'_t.
    \end{equation} By expanding and collecting the terms in $t^{n+1} $ in the two sides of \eqref{eqobs2}, we obtain the equality \begin{align*}
        \sum\limits_{i=0}^{n+1}\varphi_i(x,w_{n+1-i}(y),z)= \sum_{\substack{0\leq i_k\leq n+1 \\ i_1+\cdots+i_\frac{p+1}{2}=n+1}}\varphi_{i_\frac{p+1}{2}}(\varphi_{i_\frac{p-1}{2}}(\cdots(\varphi_{i_1}(x,y,y),y,y),\cdots),y,z),
    \end{align*} which gives, by isolating $\varphi_{n+1} $, \begin{align*}
        \mathrm{Obs}_{n+1}^{(2)}= \beta(x,y,z),
    \end{align*} where $\beta$ is induced by $(\varphi_{n+1},w_{n+1}) $.  
\end{proof}

If $(\sum\limits_{i=0}^nt^i\varphi_i, \sum\limits_{i=0}^nt^iw_i) $ is a formal deformation, then $\mathrm{Obs}_{n+1}^{(1)}=0$ and $\mathrm{Obs}_{n+1}^{(2)}=0$. In this case, \hyperref[propObs]{Proposition~\ref*{propObs}} gives the following result: Let $(\varphi_{n+1},w_{n+1})\in C^3_\mathrm{res}(T,T) $, such that $(\sum\limits_{i=0}^{n+1}t^i\varphi_i, \sum\limits_{i=0}^{n+1}t^iw_i) $ is a formal deformation. Then, $\delta^3_\mathrm{res}(\varphi_{n+1},w_{n+1})=(0,0)$.

\subsection{Extensions of restricted Lie triple systems}

In this subsection, we investigate extensions of restricted Lie triple systems and their relationship with the third cohomology group of the cohomology defined in section \eqref{s5}.

\begin{defi}
    Let $T,\mathcal{U},\mathcal{M} $ be restricted Lie triple systems. We say that $T$ is a restricted extension of $\mathcal{U}$ by $\mathcal{M}$ if there exists an exact sequence of restricted Lie triple systems $$0\to \mathcal{M}\to T\to \mathcal{U}\to 0. $$ Two extensions are equivalent if the obvious diagram is commutative. 
\end{defi}

An extension \[\begin{tikzcd}
	0 & \mathcal{M} & T & {\mathcal{U}} & 0
	\arrow[from=1-1, to=1-2]
	\arrow["\iota", from=1-2, to=1-3]
	\arrow["\pi", from=1-3, to=1-4]
	\arrow[from=1-4, to=1-5]
\end{tikzcd}\] is called strongly abelian if $\iota(\mathcal{M}) $ is a strongly abelian ideal of $T$ that is, $[T,\iota(\mathcal{M}),\iota(\mathcal{M})]=0 $ and $\iota(\mathcal{M})^{[p]}=0 $. We denote by $\mathrm{Ext}_\mathrm{res}(\mathcal{U},\mathcal{M}) $ the set of equivalence classes of restricted strongly abelian extensions of $\mathcal{U} $ by $\mathcal{M}$. We define a $\mathcal{U}-$module structure on $\mathcal{M}$ as in the non-restricted setting. Let $l $ be a section of this exact sequence. Then, we define $$\theta(u,v)(m)=[m,l(u),l(v)]. $$ We already know that $(\mathcal{M},\theta)$ is a representation of $\mathcal{U}$. The pair $(\mathcal{M},\theta)$ is also a restricted representation. Let $x,y\in \mathcal{U},m\in \mathcal{M}$. Since $\iota(\mathcal{M})$ is strongly abelian in $T$ and $\pi(l(y))^{[p]}=\pi(l(y^{[p]})) $, we have \begin{align*}
    \theta(x,y^{[p]})(m)=& [m,l(x),l(y^{[p]})]\\ =& [m,l(x),l(y)^{[p]}]\\ =& (m,l(x),l(y),\ldots,l(y))\\ =& \theta(y,y)^\frac{p-1}{2}\theta(x,y)(m)
\end{align*} and similarly $\theta(x^{[p]},y)(m)= \theta(x,y)\theta(x,x)^\frac{p-1}{2}(m) $. We can now define as in the non-restricted case an element of $C^3_\mathrm{res}(\mathcal{U},\mathcal{M}) $ associated to this restricted representation. Let $x,x_1,x_2,x_3\in T $, we define $$f(x_1,x_2,x_3)= [l(x_1),l(x_2),l(x_3)]-l([x_1,x_2,x_3]) $$ and $$w(x)= l(x)^{[p]}-l(x^{[p]}). $$ 
\begin{prop}
    With the above notations, $w$ has the $\star$-property with respect to $f$.
\end{prop}

\begin{proof}
    
Let $x,y\in \mathcal{U}$, \begin{align*}
    w(x+y)=& (l(x)+l(y))^{[p]}- l((x+y)^{[p]})\\ =& w(x)+w(y) +  \sum_{\substack{x_j\in \left\{x,y \right\}\\x_1=x,x_2=y}}\frac{1}{\#(x)}((l(x_1),\ldots,l(x_p))-l(x_1,\ldots,x_p)).
\end{align*} We use the equation $$l([x,y,z])= [l(x),l(y),l(z)]-f(x,y,z) $$ and we obtain  \begin{align*}
    &l((x_1,\ldots,x_p))\\ &= (l(x_1),\ldots,l(x_p))-\sum\limits_{k=0}^{\frac{p-3}{2}}\theta(x_{p-1},x_p)\cdots\theta(x_{p-2k+1},x_{p-2k+2})f((x_1,\ldots,x_{p-2k-2}),x_{p-2k-1},x_{p-2k}).
\end{align*} Therefore, \begin{align*}
    w(x&+y)\\=&w(x)+w(y)\\ &+ \sum_{\substack{x_j\in \left\{x,y \right\}\\x_1=x,x_2=y}}\frac{1}{\#(x)} \sum\limits_{k=0}^{\frac{p-3}{2}}\theta(x_{p-1},x_p)\cdots\theta(x_{p-2k+1},x_{p-2k+2})f((x_1,\ldots,x_{p-2k-2}),x_{p-2k-1},x_{p-2k}).
\end{align*} Thus, $w $ has the $\star$-property with respect to $f$ and $(f,w)\in C^3_\mathrm{res}(\mathcal{U},\mathcal{M}) $.
\end{proof}

\begin{prop}
  The pair $(f,w)$ is a restricted $3$-cocycle.
\end{prop}

\begin{proof}
We already know that $f $ is a $3$-cocycle of the Yamaguti cohomology. So in order to see that $(f,w) $ is a $3$-cocycle of the restricted cohomology it remains to show that the induced map $\beta $ vanishes. Let $x,y,z\in \mathcal{U},$ \begin{align*}
    \beta(&x,y,z)\\ =& f(x,y^{[p]},z)- \sum\limits_{i+j=\frac{p-3}{2}}\theta(y,z)\theta(y,y)^if((x,\underbrace{y,\ldots,y}_{2j}),y,y) - f((x,\underbrace{y,\ldots,y}_{p-1}),y,z)- \theta(x,z)(w(y))\\ =& [l(x),l(y^{[p]}),l(z)]-l([x,y^{[p]},z])- \sum\limits_{i+j=\frac{p-3}{2}}(f((x,\underbrace{y,\ldots,y}_{2j}),y,y),\underbrace{l(y),\ldots,l(y)}_{2i},l(y),l(z)) \\ &-f((x,y\ldots,y),y,z) - [l(y)^{[p]},l(x),l(z)]+ [l(y^{[p]}),l(x),l(z)].
\end{align*} Using again the equation $$l([x,y,z])= [l(x),l(y),l(z)]-f(x,y,z), $$ we obtain that \begin{align*}
    - \sum\limits_{i+j=\frac{p-3}{2}}&(f((x,\underbrace{y,\ldots,y}_{2j}),y,y),\underbrace{l(y),\ldots,l(y)}_{2i},l(y),l(z)) -f((x,y\ldots,y),y,z) - [l(y)^{[p]},l(x),l(z)]\\ = & l((x,y,\ldots,y,z)\\ =& l([x,y^{[p]},z]).
\end{align*} Thus, $$\beta(x,y,z)=0 ,$$ which concludes the proof.

\end{proof}

\begin{lem}
    Two different sections of a restricted extension of a restricted Lie triple system $\mathcal{U}$ give the same restricted cohomology class.
\end{lem}

\begin{proof}
Let $l,l'$ be two sections and $g=l'-l $. The map $g$ takes values in $\mathcal{M}$. Let $(f,w) $ be a cocycle defined by $l$ and $(f',w')$ be a cocycle defined by $l'$. We already know that $f'-f= \delta^1 g $. Let $x\in \mathcal{U}$, \begin{align*}
    w'(x)&= l'(x)^{[p]}-l'(x^{[p]}) \\ &= (g(x)+l(x))^{[p]}- g(x^{[p]})-l(x^{[p]})\\ &= w(x) + g(x)^{[p]} - g(x^{[p]}) + \sum_{\substack{x_j\in \left\{g(x),l(x) \right\}\\x_1=g(x),x_2=l(x)}}\frac{1}{\#g(x)} (x_1,\ldots,x_p) \\ &= w(x) - g(x^{[p]})+ (g(x),l(x),\ldots,l(x))\\ &= w(x) + \tilde{g}(x)
\end{align*} because $\iota(\mathcal{M})$ is strongly abelian in $T$. Thus, $(f',w')= (f,w)+\delta^1_\mathrm{res}(g) $ and they are in the same restricted cohomology class. 
\end{proof}

\begin{prop}
    Two equivalent restricted extensions of a restricted Lie triple system $\mathcal{U}$ correspond to the same restricted cohomology class.
\end{prop}

\begin{proof}
Let $T$ and $T'$ be two equivalent abelian extensions \[\begin{tikzcd}
	0 & \mathcal{M} & T & {\mathcal{U}} & 0 \\
	0 & \mathcal{M} & {T'} & {\mathcal{U}} & 0
	\arrow[from=1-1, to=1-2]
	\arrow["\iota", from=1-2, to=1-3]
	\arrow["{\mathrm{id}}", from=1-2, to=2-2]
	\arrow["\pi", from=1-3, to=1-4]
	\arrow["F", from=1-3, to=2-3]
	\arrow[from=1-4, to=1-5]
	\arrow["{\mathrm{id}}", from=1-4, to=2-4]
	\arrow[from=2-1, to=2-2]
	\arrow["{\iota'}", from=2-2, to=2-3]
	\arrow["{\pi'}", from=2-3, to=2-4]
	\arrow[from=2-4, to=2-5]
\end{tikzcd}\] Let $l$ and $l'$ be two sections such that $\pi\circ l=\mathrm{id} $ and $\pi'\circ l'=\mathrm{id} $. We recall that these two abelian extensions define the same representation $\theta:\mathcal{U}\times \mathcal{U}\to \mathcal{M}$. Let $x,x_1,x_2,x_3\in \mathcal{U}$, let $$f(x_1,x_2,x_3)=[l(x_1),l(x_2),l(x_3)]-l([x_1,x_2,x_3]) ,\ \ w(x)= l(x)^{[p]}-l(x^{[p]}) $$and $$f'(x_1,x_2,x_3)= [l'(x_1),l'(x_2),l'(x_3)]-l'([x_1,x_2,x_3]), \ \ w'(x)=l'(x)^{[p]}-l'(x^{[p]}) $$ be the two cocycles defined above. As in the non-restricted case, $F\circ l $ is a section of $T'$ and $l'-F\circ l  $ takes values in $\mathcal{M}$, which is strongly abelian. Let $g=l'-F\circ l$. We have already seen that $f' $ and $f$ are in the same cohomology class. \begin{align*}
    w'(x) &=  (g(x)+F\circ l(x))^{[p]} - F\circ l(x^{[p]})- g(x^{[p]})\\ &= F( l(x)^{[p]}- l(x^{[p]}))+ \sum_{\substack{x_j\in \left\{g(x),F\circ l(x) \right\}\\x_1=g(x),x_2=F\circ l(y)}}\frac{1}{\#g(x)} (x_1,\ldots,x_p)  -g(x^{[p]}) \\ &= w(x) +(g(x),F(l(x)),\ldots,F(l(x))) - g(x^{[p]}) \\ &= w(x) + \tilde{g}(x).
\end{align*} Thus, the cohomology class of $(f,w) $ depends only on the equivalence class of the restricted abelian extension. 
\end{proof}

Let $\mathcal{U} $ be a Lie triple system, $\mathcal{M} $ a restricted $\mathcal{U} $-module and $(f,w)\in C^3_\mathrm{res}(\mathcal{U},\mathcal{M}) $ a restricted $3$-cocycle. We have seen that we can define on $\mathcal{M}\times \mathcal{U} $ the Lie triple system bracket \begin{align*}
    [(m_1,x_1)&,(m_2,x_2),(m_3,x_3)] \\ &= (\theta(x_2,x_3)(m_1)- \theta(x_1,x_3)(m_2)+D(x_1,x_2)(m_3)+f(x_1,x_2,x_3),[x_1,x_2,x_3]). 
\end{align*} Similarly, we can define a restricted structure on $\mathcal{M}\times \mathcal{U} $ by $$(m,x)^{[p]}= (w(x)+\theta(x,x)^{\frac{p-1}{2}}(m),x^{[p]}), $$ for all $(m,x)\in \mathcal{M}\times \mathcal{U}$.

\begin{prop}\label{prop6.18}
    Let $(m,x),(q,y)\in \mathcal{M}\times \mathcal{U} $ in the Lie triple system defined above. Then we have \begin{align*}
       & \bigg((m,x),(q,y),\ldots,(q,y)\bigg)\\ &= \bigg(\theta(y,y)^n(m)- \sum\limits_{k=0}^{n-1}\tbinom{2n}{2k+1}\theta(y,y)^k\theta(x,y)\theta(y,y)^{n-1-k}(q)\\ & +\sum\limits_{k=0}^{n-1}\tbinom{2n}{2(k+1)}\theta(y,y)^{n-k-1}\theta(y,x)\theta(y,y)^k(q) + \sum\limits_{i+j=n-1}\theta(y,y)^if((x,\underbrace{y,\ldots,y}_{2j}),y,y), (x,y\ldots,y)\bigg),
    \end{align*} where $(q,y) $ occurs $2n$ times in the bracket for $n\geq 1$.
\end{prop}

\begin{proof}
    We proceed by induction on $n$. The formula is true for $n=1$. Assume that the formula is true for $n\geq 1$. Let $(m,x),(q,y)\in \mathcal{M}\times \mathcal{U}$, we have \begin{align*}
        ((m,x)&,\underbrace{(q,y),\ldots,(q,y)}_{2n},(q,y),(q,y))\\=& [((m,x),\underbrace{(q,y),\ldots,(q,y)}_{2n}),(q,y),(q,y)]\\ =& \bigg(\theta(y,y)( \theta(y,y)^n(m)- \sum\limits_{k=0}^{n-1}\tbinom{2n}{2k+1}\theta(y,y)^k\theta(x,y)\theta(y,y)^{n-1-k}(q)\\ & +\sum\limits_{k=0}^{n-1}\tbinom{2n}{2(k+1)}\theta(y,y)^{n-k-1}\theta(y,x)\theta(y,y)^k(q) + \sum\limits_{i+j=n-1}\theta(y,y)^if((x,y,\ldots,y),y,y)) \\ &- \theta((x,y,\ldots,y),y)(q) + D((x,y,\ldots,y),y)(q)+ f((x,y,\ldots,y),y,y),(x,y,\ldots,y,y,y)\bigg)\\ =& \bigg(\theta(y,y)^{n+1}(m)- \sum\limits_{k=0}^{n}\tbinom{2(n+1)}{2k+1}\theta(y,y)^k\theta(x,y)\theta(y,y)^{n-k}(q)\\ & +\sum\limits_{k=0}^{n}\tbinom{2(n+1)}{2(k+1)}\theta(y,y)^{n-k}\theta(y,x)\theta(y,y)^k(q) + \sum\limits_{i+j=n}\theta(y,y)^if((x,y,\ldots,y),y,y), (x,y\ldots,y)\bigg).
    \end{align*} Thus the result follows.
\end{proof}

\begin{prop}
     Let $(m,x),(q,y),(q',z)\in \mathcal{M}\times \mathcal{U} $ in the Lie triple system defined above. Then, we have \begin{align*}
        & ((m,x),(q,y),\ldots,(q,y),(q',z))\\=& \bigg(\theta(y,z)\theta(y,y)^n(m)-\theta(y,z)\sum\limits_{k=1}^{n}\tbinom{2n+1}{2k}\theta(y,y)^{k-1}\theta(x,y)\theta(y,y)^{n-k}(q)\\+ & \theta(y,z)\sum\limits_{k=0}^{n-1}\tbinom{2n+1}{2(k+1)}\theta(y,y)^{n-k-1}\theta(y,x)\theta(y,y)^k(q) + \theta(y,z)\sum\limits_{i+j=n-1}\theta(y,y)^if((x,\underbrace{y,\ldots,y}_{2j}),y,y)\\ -& \theta(x,z)\theta(y,y)^n(q) + \sum\limits_{k=0}^n\tbinom{2n+1}{2k+1}\theta(y,y)^{n-k}\theta(y,x)\theta(y,y)^k(q') \\-& \sum\limits_{k=0}^n\tbinom{2n+1}{2k+1}\theta(y,y)^k\theta(x,y)\theta(y,y)^{n-k}(q')  + f((x,y\ldots,y),y,z), (x,y\ldots,y,z)\bigg),
     \end{align*} where $(q,y) $ occurs $2n+1$ times in the bracket. In particular, \begin{align}
           ((&m,x),(n,y),\ldots,(n,y),(q,z))\notag\\=&\bigg(\theta(y,z)\theta(y,y)^{\frac{p-1}{2}}(m)+ \theta(y,z)\sum\limits_{i+j=\frac{p-3}{2}}\theta(y,y)^if((x,y,\ldots,y),y,y) + f((x,y\ldots,y),y,z)\notag \\&-\theta(x,z)\theta(y,y)^\frac{p-1}{2}(n)+ \theta(y,x)\theta(y,y)^\frac{p-1}{2}(q)-\theta(y,y)^\frac{p-1}{2}\theta(x,y)(q), (x,y\ldots,y,z)\bigg)\label{eq6.23},
     \end{align} where $(n,y) $ occurs $p$ times in the bracket.
\end{prop}

\begin{proof}
    A direct computation using \hyperref[prop6.18]{Proposition~\ref*{prop6.18}} gives the result.
\end{proof}

\begin{prop}
    Let $u_1=(m_1,x_1),\ldots,u_{2n+1}=(m_{2n+1},x_{2n+1})\in \mathcal{M}\times \mathcal{U} $, where $n\geq 1$ is an integer. Then, we have {\small\begin{align*}
       & (u_1,\ldots,u_{2n+1})\\ &= \bigg( \sum\limits_{k=0}^{n-1}\theta(x_{2n},x_{2n+1})\cdots\theta(x_{2n-2k+2},x_{2n-2k+3}) f((x_1,\ldots,x_{2n-2k-1}),x_{2n-2k},x_{2n-2k+1})\\ &+\sum_{\substack{A+B=\left\{2,\ldots,2n+1\right\}}}(-1)^s\theta(x_{a_{s-1}},x_{a_s})\theta(x_{a_{s-3}},x_{a_{s-2}})\cdots \left\{\begin{matrix}
\theta (x_1,x_{a_1})\theta(x_{b_2},x_{b_1})\cdots\theta(x_{b_{t-1}},x_{b_{t-2}})(m_{b_t}) \\\theta (x_{b_1},x_1)\theta(x_{b_3},x_{b_2})\cdots\theta(x_{b_{t-1}},x_{b_{t-2}})(m_{b_t})\\ \theta(x_{a_1},x_{a_2})(m_1) \,\,\,\,\,\,\text{if $B=\varnothing$}
\end{matrix}\right. \\ & \,\,\,\,\,\,\,\,\,\,\,\,\,\,\,\,\,\,\,\,\,\,\,\, , (x_1,\ldots,x_{2n+1})\bigg).
    \end{align*} }
\end{prop}

\begin{proof}
   We proceed by induction on $n$. The formula is true for $n=1$. Assume that the result is true for $n\geq 1$. Let $u_1,\ldots,u_{2n+3}\in \mathcal{M}\times \mathcal{U} $, we have  \begin{align*}
       & (u_1,\ldots,u_{2n+3})\\&=[(u_1,\ldots,u_{2n+1}),u_{2n+2},u_{2n+3}]\\ &= \bigg(\theta(x_{2n+2},x_{2n+3})\bigg(\sum\limits_{k=0}^{n-1}\theta(x_{2n},x_{2n+1})\cdots\theta(x_{2n-2k+2},x_{2n-2k+3})\\ & \,\,\,\,\,\,\,\,\,\,\,\,\,\,\,\,\,\,\,\,\,\,\,\,\,\,\,\,\,\,\,\,\,\,\,\,\,\,\,\,\,\,\,\,\,\,\,\,\,\,\,\,\,\,\,\,\,\,\,\,\,\,\,\,\,\,\,\,\,\,\,\,\,\,\,\,\,\,\,\,\,\,\,\,\,\,\,\,\,\,\,\,\,\,\, \,\,\,\,\,\,\,\,\,\,\,\,\,\,\,\,\,\,\,\,\,\,\,\,\,\,\,\,\,\,\,\,\,\,\,\,\,\,\,\,\,\,\,\,\,\,\,\,\,f((x_1,\ldots,x_{2n-2k-1}),x_{2n-2k},x_{2n-2k+1})\\ &+\sum_{\substack{A+B=\left\{2,\ldots,2n+1\right\}}}(-1)^s\theta(x_{a_{s-1}},x_{a_s})\theta(x_{a_{s-3}},x_{a_{s-2}})\\ & \,\,\,\,\,\,\,\,\,\,\,\,\,\,\,\,\,\,\,\,\,\,\,\,\,\,\,\,\,\,\,\,\,\,\,\,\,\,\,\,\,\,\,\,\,\,\,\,\,\,\,\,\,\,\,\,\,\,\,\,\,\,\,\,\,\,\,\,\,\,\,\,\,\,\,\,\,\,\,\,\,\,\,\,\,\,\,\,\,\,\,\,\,\,\,\,\,\,\,\,\,\,\,\,\,\,\,\,\,\,\,\,\,\,\,\,\,\,\,\,\,\,\,\,\,\,\cdots \left\{\begin{matrix}
\theta (x_1,x_{a_1})\theta(x_{b_2},x_{b_1})\cdots\theta(x_{b_{t-1}},x_{b_{t-2}})(m_{b_t}) \\\theta (x_{b_1},x_1)\theta(x_{b_3},x_{b_2})\cdots\theta(x_{b_{t-1}},x_{b_{t-2}})(m_{b_t})\\ \theta(x_{a_1},x_{a_2})(m_1) \,\,\,\,\,\,\text{if $B=\varnothing$}
\end{matrix}\right.\bigg)\\ &- \theta((x_1,\ldots,x_{2n+1},x_{2n+3})(m_{2n+2})+D((x_1,\ldots,x_{2n+1}),x_{2n+2})(m_{2n+3}) \\ &+ f((x_1,\ldots,x_{2n+1}),x_{2n+2},x_{2n+3}), (x_1,\ldots,x_{2n+3})\bigg)\\ &= \bigg(\sum\limits_{k=0}^{n}\theta(x_{2n+2},x_{2n+3})\cdots\theta(x_{2(n+1)-2k+2},x_{2(n+1)-2k+3}) f([x_1,\ldots,x_{2(n+1)-2k-1}],x_{2(n+1)-2k},x_{2(n+1)-2k+1})\\ &+ \theta(x_{2n+2},x_{2n+3})\sum_{\substack{A+B=\left\{2,\ldots,2n+1\right\}}}(-1)^s\theta(x_{a_{s-1}},x_{a_s})\theta(x_{a_{s-3}},x_{a_{s-2}})\\ & \,\,\,\,\,\,\,\,\,\,\,\,\,\,\,\,\,\,\,\,\,\,\,\,\,\,\,\,\,\,\,\,\,\,\,\,\,\,\,\,\,\,\,\,\,\,\,\,\,\,\,\,\,\,\,\,\,\,\,\,\,\,\,\,\,\,\,\,\,\,\,\,\,\,\,\,\,\,\,\,\,\,\,\,\,\,\,\,\,\,\,\,\,\,\,\,\,\,\,\,\,\,\,\,\,\,\,\,\,\,\,\,\,\,\,\,\,\,\,\,\,\,\,\,\,\,\cdots \left\{\begin{matrix}
\theta (x_1,x_{a_1})\theta(x_{b_2},x_{b_1})\cdots\theta(x_{b_{t-1}},x_{b_{t-2}})(m_{b_t}) \\\theta (x_{b_1},x_1)\theta(x_{b_3},x_{b_2})\cdots\theta(x_{b_{t-1}},x_{b_{t-2}})(m_{b_t})\\ \theta(x_{a_1},x_{a_2})(m_1) \,\,\,\,\,\,\text{if $B=\varnothing$}
\end{matrix}\right.\\ &+ \sum_{\substack{A+B=\left\{\small
    2,\ldots,2n+1, 2n+3 \right\}\\ 2n+3\notin B}}(-1)^s\theta(x_{a_{s-1}},x_{a_s})\theta(x_{a_{s-3}},x_{a_{s-2}})\\ &\,\,\,\,\,\,\,\,\,\,\,\,\,\,\,\,\,\,\,\,\,\,\,\,\,\,\,\,\,\,\,\,\,\,\,\,\,\,\,\,\,\,\,\,\,\,\,\,\,\,\,\,\,\,\,\,\,\,\,\,\,\,\,\,\,\,\,\,\,\,\,\,\,\,\,\,\,\,\,\,\,\,\,\,\,\,\,\,\,\,\,\,\,\,\,\,\,\,\,\,\,\,\,\,\,\,\,\,\,\,\,\,\,\,\,\,\,\, \cdots \left\{\begin{matrix}
\theta (x_1,x_{a_1})\theta(x_{b_2},x_{b_1})\cdots\theta(x_{b_t},x_{b_{t-1}}) \\\theta (x_{b_1},x_1)\theta(x_{b_3},x_{b_2})\cdots\theta(x_{b_t},x_{b_{t-1}})\\ 
\end{matrix}\right.(m_{2n+2})\\ &+ \sum_{\substack{A+B=\left\{2,\ldots,2n+2\right\}}}(-1)^s\theta(x_{a_{s-1}},x_{a_s})\theta(x_{a_{s-3}},x_{a_{s-2}}) \\ &\,\,\,\,\,\,\,\,\,\,\,\,\,\,\,\,\,\,\,\,\,\,\,\,\,\,\,\,\,\,\,\,\,\,\,\,\,\,\,\,\,\,\,\,\,\,\,\,\,\,\,\,\,\,\,\,\,\,\,\,\,\,\,\,\,\,\,\,\,\,\,\,\,\,\,\cdots \left\{\begin{matrix}
\theta (x_1,x_{a_1})\theta(x_{b_2},x_{b_1})\cdots\theta(x_{b_{t}},x_{b_{t-1}}) \\\theta (x_{b_1},x_1)\theta(x_{b_3},x_{b_2})\cdots\theta(x_{b_t},x_{b_{t-1}})
\end{matrix}\right.(m_{2n+3}),(x_1,\ldots,x_{2n+3})\bigg).
    \end{align*} The result follows.
\end{proof}

\begin{thm}\label{thm6.21}
    The Lie triple system $\mathcal{M}\times \mathcal{U}$ defined by  \begin{align*}
    [(m_1,x_1)&,(m_2,x_2),(m_3,x_3)] \\ &= (\theta(x_2,x_3)(m_1)- \theta(x_1,x_3)(m_2)+D(x_1,x_2)(m_3)+f(x_1,x_2,x_3),[x_1,x_2,x_3])
\end{align*} and the following  $p$-map \begin{equation}\label{eq6.21}
    (m,x)^{[p]}= (w(x)+\theta(x,x)^{\frac{p-1}{2}}(m),x^{[p]})
\end{equation}  is a restricted Lie triple system.
\end{thm}

\begin{proof}
Let $(m,x),(n,y),(q,z)\in \mathcal{M}\times \mathcal{U}$. We have  \begin{align*}
    [(m,x),&(n,y)^{[p]},(q,z)]\\ =& [(m,x),(w(y)+\theta(y,y)^\frac{p-1}{2}(n),y^{[p]}),(q,z)] \\ =& (\theta(y^{[p]},z)(m)-\theta(x,z)(w(y)+\theta(y,y)^\frac{p-1}{2}(n))+D(x,y^{[p]})(q)+f(x,y^{[p]},z),[x,y^{[p]},z])\\ =& (\theta(y,z)\theta(y,y)^\frac{p-1}{2}(m)-\theta(x,z)(w(y))-\theta(x,z)\theta(y,y)^\frac{p-1}{2}(n)+\theta(y,x)\theta(y,y)^\frac{p-1}{2}(q)\\ &-\theta(y,y)^\frac{p-1}{2}\theta(x,y)(q)+f(x,y^{[p]},z), (x,y,\ldots,y,z)).
\end{align*} $(f,w)$ is a $3$-cocycle. So $$f(x,y^{[p]},z)= \sum\limits_{i+j=\frac{p-3}{2}}\theta(y,z)\theta(y,y)^if((x,y\ldots,y),y,y)+f((x,y\ldots,y),y,z)+\theta(x,z)(w(y)). $$ Thus, using the value of $((m,x),\underbrace{(n,y),\ldots,(n,y)}_p,(q,z))$ in \eqref{eq6.23}, we have
    $$[(m,x),(n,y)^{[p]},(q,z)]=((m,x),\underbrace{(n,y),\ldots,(n,y)}_p,(q,z)).$$ Now, we show that the Jacobson identities hold for this $p$-map. Let $(m,x),(n,y)\in \mathcal{M}\times \mathcal{U} $. We have\begin{align*}
       & (m+n,x+y)^{[p]}\\ &= (w(x+y)+\theta(x+y,x+y)^\frac{p-1}{2}(m+n),(x+y)^{[p]})\\ &= \bigg(w(x)+w(y) \\ & +  \sum_{\substack{x_j\in \left\{x,y \right\}\\x_1=x,x_2=y}}\frac{1}{\#(x)}\sum\limits_{k=0}^{\frac{p-1}{2}-1}\theta(x_{p-1},x_p)\cdots\theta(x_{p-2k+1},x_{p-2k+2}) f((x_1,\ldots,x_{p-2k-2}),x_{p-2k-1},x_{p-2k}) \\ & +  \sum_{\substack{x_j\in \left\{x,y \right\}}}\theta(x_{p-1},x_{p-2})\cdots\theta(x_2,x_1)(m)+  \sum_{\substack{x_j\in \left\{x,y \right\}}}\theta(x_{p-1},x_{p-2})\cdots\theta(x_2,x_1)(n), x^{[p]}+y^{[p]}\\ &+  \sum_{\substack{x_j\in \left\{x,y \right\}\\x_1=x,x_2=y}}\frac{1}{\#(x)}(x_1,\ldots,x_p) \bigg)\\ & = (m,x)^{[p]}+(n,y)^{[p]}\\ & +\bigg( \sum_{\substack{x_j\in \left\{x,y \right\}\\x_1=x,x_2=y}}\frac{1}{\#(x)}\sum\limits_{k=0}^{\frac{p-1}{2}-1}\theta(x_{p-1},x_p)\cdots\theta(x_{p-2k+1},x_{p-2k+2}) f((x_1,\ldots,x_{p-2k-2}),x_{p-2k-1},x_{p-2k})\\ & +  \sum_{\substack{x_j\in \left\{x,y \right\}}}\theta(x_{p-1},x_{p-2})\cdots\theta(x_2,x_1)(m) \\ &+  \sum_{\substack{x_j\in \left\{x,y \right\}}}\theta(x_{p-1},x_{p-2})\cdots\theta(x_2,x_1)(n), \sum_{\substack{x_j\in \left\{x,y \right\}\\x_1=x,x_2=y}}\frac{1}{\#(x)}(x_1,\ldots,x_p) \bigg).
    \end{align*}

Thus, it remains to show that \begin{align*}
        \sum_{\substack{u_j\in \left\{(m,x),(n,y) \right\}\\u_1=(m,x),u_2=(n,y)}}&\frac{1}{\#(m,x)}\\ &\sum_{\substack{A+B=\left\{2,\ldots,p\right\}}}(-1)^s\theta(x_{a_{s-1}},x_{a_s})\theta(x_{a_{s-3}},x_{a_{s-2}}) \\ & \,\,\,\,\,\,\,\,\,\,\,\,\,\,\,\,\,\,\,\,\,\,\,\,\,\,\,\,\,\,\,\,\,\,\,\,\,\,\,\,\,\,\,\,\,\,\,\,\,\,\,\,\,\,\,\,\,\,\,\,\,\,\,\,\,\,\,\,\,\,\,\,\,\,\,\,\,\,\,\,\,\,\,\,\,\cdots \left\{\begin{matrix}
\theta (x_1,x_{a_1})\theta(x_{b_2},x_{b_1})\cdots\theta(x_{b_{t-1}},x_{b_{t-2}})(m_{b_t}) \\\theta (x_{b_1},x_1)\theta(x_{b_3},x_{b_2})\cdots\theta(x_{b_{t-1}},x_{b_{t-2}})(m_{b_t})
\end{matrix}\right.\\ &= \sum_{\substack{u_j\in \left\{(m,x),(n,y) \right\}\\ \#(m,x)\neq 0,p}}\theta(x_{p},x_{p-1})\cdots\theta(x_3,x_2)(m_1).
    \end{align*} The same counting as in the proof of \hyperref[prop5.7]{Proposition~\ref*{prop5.7}} shows that these two terms are equal. Then, \begin{align*}
        ((m,x)+(n,y))^{[p]}= (m,x)^{[p]}+(n,y)^{[p]}+\sum_{\substack{u_j\in \left\{(m,x),(n,y) \right\}\\ u_1=(m,x),u_2=(n,y)}}\frac{1}{\#(m,x)}(u_1,\ldots,u_p)
    \end{align*} and $\mathcal{M}\times \mathcal{U} $ is a restricted Lie triple system.

\end{proof}

Hence, the structure defined above is a restricted Lie triple system structure. In particular, we have the following result, which improves our previous result in \cite{BM26}.

\begin{thm}
     Let $T$ be a restricted Lie triple system over a field $\mathbf{K} $ of characteristic $p>2 $. Let $(V,\theta)$ be a restricted representation of $T$. Define the operation $[\cdot,\cdot,\cdot]_V: (T\oplus V)\times (T\oplus V)\times (T\oplus V)\to T\oplus V  $ by \begin{center}
        $[(x,a),(y,b),(z,c)]_V=([x,y,z], \theta(y,z)(a)-\theta(x,z)(b)+ D(x,y)(c)) .$ 
    \end{center}  Then, there exists a $p$-map on $T\oplus V$ given by the formula \begin{center}
        $(x,v)^{[p]}=(x^{[p]},\theta(x,x)^{\frac{p-1}{2}}(v))$,
    \end{center} which endows $T\oplus V $ with a structure of restricted Lie triple system.
\end{thm}

\hyperref[thm6.21]{Theorem~\ref*{thm6.21}} gives us the following restricted abelian extension of $\mathcal{U}$ by $\mathcal{M}$: \[\begin{tikzcd}
	0 & \mathcal{M} & {\mathcal{M}\times\mathcal{U}} & {\mathcal{U}} & 0.
	\arrow[from=1-1, to=1-2]
	\arrow[from=1-2, to=1-3]
	\arrow[from=1-3, to=1-4]
	\arrow[from=1-4, to=1-5]
\end{tikzcd}\] where $\mathcal{M}\times\mathcal{U} $ has the $p$-map \eqref{eq6.21}.

\begin{lem}
    Let $\mathcal{U} $ be a restricted Lie triple system and $(\mathcal{M},\theta) $ be a restricted representation of $\mathcal{U}$. Let $(f,w),(f',w')\in C_\mathrm{res}^3(\mathcal{U},\mathcal{M}) $ be two $3$-cocycles. Then, the two restricted abelian extensions given by $(f,w)$ and $(f',w')$ are equivalent if and only if $(f,w)$ and $(f',w')$ are in the same cohomology class.
\end{lem}

\begin{proof}
    Assume that the restricted abelian extensions induced by $(f,w),(f',w') $ are equivalent. We have the following commutative diagram \[\begin{tikzcd}
	0 & \mathcal{M} & (\mathcal{M}\times \mathcal{U},[\cdot,\cdot,\cdot]_{f},(\cdot)^{[p]_w}) & {\mathcal{U}} & 0 \\
	0 & \mathcal{M} & {(\mathcal{M}\times \mathcal{U},[\cdot,\cdot,\cdot]_{f'},(\cdot)^{[p]_{w'}})} & {\mathcal{U}} & 0
	\arrow[from=1-1, to=1-2]
	\arrow["\iota", from=1-2, to=1-3]
	\arrow["{\mathrm{id}}", from=1-2, to=2-2]
	\arrow["\pi", from=1-3, to=1-4]
	\arrow["F", from=1-3, to=2-3]
	\arrow[from=1-4, to=1-5]
	\arrow["{\mathrm{id}}", from=1-4, to=2-4]
	\arrow[from=2-1, to=2-2]
	\arrow["{\iota'}", from=2-2, to=2-3]
	\arrow["{\pi'}", from=2-3, to=2-4]
	\arrow[from=2-4, to=2-5]
\end{tikzcd}\] There exists a linear map $g:\mathcal{U}\to \mathcal{M}$ such that $F(m,x)= (m+g(x),x)$. Zhang showed in \cite{Z14} that $f-f'=\delta^1 g$. Thus, it remains to show that $w-w'=\tilde{g}$. The map $F$ is a morphism of restricted Lie triple systems. So, $F((0,x)^{[p]_w})=F(0,x)^{[p]_{w'}}  $. By definition of the induced $p$-map, $$F((0,x)^{[p]_w})= F(w(x),x^{[p]})= (w(x)+g(x^{[p]}),x^{[p]}) $$ and $$F(0,x)^{[p]_{w'}}=(g(x),x)^{[p]_{w'}}=(w'(x)+\theta(x,x)^\frac{p-1}{2}(g(x)),x^{[p]}) .$$ So $$w(x)-w'(x)=\theta(x,x)^\frac{p-1}{2}(g(x))-g(x^{[p]})=\tilde{g}(x). $$ Thus, $(f,w)-(f',w')=\delta^1_\mathrm{res}g $. Hence, $(f,w)$ and $(f',w')$ are in the same cohomology class.

Conversely, assume that there is a linear map $g:\mathcal{U}\to \mathcal{M}$ such that $w-w'=\tilde{g} $ and $f-f'=\delta^1 g $. Then, the map $$
\begin{array}[t]{lrcl}
F : & \mathcal{M}\times \mathcal{U} &  \longrightarrow & \mathcal{M}\times \mathcal{U} \\
   & (m,x) &  \longmapsto & (m+g(x),x) \end{array}
$$ is a morphism of restricted Lie triple systems and the two restricted abelian extensions are equivalent.
\end{proof}

Thus, we have a well-defined map $H^3_\mathrm{res}(\mathcal{U},\mathcal{M})\to \mathrm{Ext}_\mathrm{res}(\mathcal{U},\mathcal{M}) ,$ which is injective. 

Now, let $\mathcal{U}$ be a restricted Lie triple system and $(\mathcal{M},\theta)$ be a restricted representation of $\mathcal{U}$. Let $$0\to \mathcal{M} \to T\to \mathcal{U}\to 0 $$ be a restricted abelian extension inducing a representation on $\mathcal{M}$ equal to $\theta$, i.e. $$\theta'(u,v)(m):=[m,l(u),l(v)]=\theta(u,v)(m), $$ where $l:\mathcal{U}\to T$ is a section. Let $(f,w)\in C^3_\mathrm{res}(\mathcal{U},\mathcal{M}) $ be the corresponding cocycle. Then, the extension $$0\to \mathcal{M}\to \mathcal{M}\times \mathcal{U}\to \mathcal{U}\to 0 ,$$ where the bracket on $\mathcal{M}\times \mathcal{U} $ is defined by  \begin{align*}
    [(m_1,x_1)&,(m_2,x_2),(m_3,x_3)]_f \\ &= (\theta(x_2,x_3)(m_1)- \theta(x_1,x_3)(m_2)+D(x_1,x_2)(m_3)+f(x_1,x_2,x_3),[x_1,x_2,x_3])
\end{align*} and the $p$-map is defined by $$(m,x)^{[p]_w}=(w(x)+\theta(x,x)^\frac{p-1}{2}(m),x^{[p]}), $$ is equivalent to $$0\to \mathcal{M} \to T\to \mathcal{U}\to 0 $$ with the isomorphism of restricted Lie triple systems $$F:\mathcal{M}\times \mathcal{U}\to T, (m,x)\mapsto m+l(x). $$ Indeed, \begin{align*}
    F((m,x)^{[p]_w})=& F(w(x)+\theta(x,x)^\frac{p-1}{2}(m),x^{[p]})\\ =& w(x)+\theta(x,x)^\frac{p-1}{2}(m)+l(x^{[p]})
\end{align*} and \begin{align*}
    (m+l(x))^{[p]}=& m^{[p]}+l(x)^{[p]}+\sum_{\substack{x_j\in \left\{m,l(x) \right\}\\x_1=m,x_2=l(x)}}\frac{1}{\#(m)}(x_1,\ldots,x_p)\\ =& l(x)^{[p]}+ \theta(x,x)^\frac{p-1}{2}(m),
\end{align*} since $\mathcal{M}$ is strongly abelian in $T$. Thus, $F$ is a morphism of restricted Lie triple systems according to the definition of $w$.

Thus, the two constructions are inverse of each other and finally, we have the following result. 

\begin{thm}
    Let $\mathcal{U} $ be a restricted Lie triple system and $(\mathcal{M},\theta) $ be a restricted representation of $\mathcal{U}$. There is a bijection between equivalence classes of restricted abelian extensions of $\mathcal{U}$ by $\mathcal{M}$ inducing the representation $\theta$ and $H^3_\mathrm{res}(\mathcal{U},\mathcal{M}) $.
\end{thm}

We now consider the particular case of central extensions of $\mathcal{U} $ by $\mathcal{M}$. As observed above, the representation on $\mathcal{M}$ induced by a central extension is trivial. Hence, the map $\mathrm{Ext_c}(\mathcal{U},\mathcal{M})\to H^3(\mathcal{U},\mathcal{M}) $ is well-defined. Let $\mathcal{U} $ be a restricted Lie triple system and let $\mathcal{M} $ be a vector space regarded as a trivial $\mathcal{U}$-module. For $(\varphi,w)\in C_\mathrm{res}^3(\mathcal{U},\mathcal{M}) $, the structure on $\mathcal{M}\times \mathcal{U} $ is given by $$[(m_1,x_1),(m_2,x_2),(m_3,x_3)]= (\varphi(x_1,x_2,x_3),[x_1,x_2,x_3]) $$ and $$(m,x)^{[p]}=(w(x),x^{[p]}). $$ The resulting abelian extension is central. The two constructions are inverse of each other and we obtain the following result.

\begin{thm}
     Let $\mathcal{U}$ be a restricted Lie triple system and $\mathcal{M}$ be a vector space regarded as a trivial $\mathcal{U}$-module. Then, there is a bijection between equivalence classes of restricted central extensions of $\mathcal{U}$ by $\mathcal{M}$ and $H^3_\mathrm{res}(\mathcal{U},\mathcal{M})$.
\end{thm}

\end{document}